\documentclass{amsart}
\usepackage{omri_commands}

\title{Critical exponent gap and leafwise dimension}
\author{Omri Nisan Solan}
\thanks{\texttt{omrinisan.solan@mail.huji.ac.il}, Einstein Institute of Mathematics, Hebrew University, Israel.}
\thanks{This research was supported by ERC 2020 grant HomDyn (grant no.~833423).}
\date{\today}
\usepackage{todonotes}
\begin{document}
\begin{abstract}
  We show that for every geometrically finite Kleinian group $\Gamma<\SL_2(\CC)$ there is a gap $\varepsilon_\Gamma>0$ such that for every $g\in \SL_2(\CC)$ the intersection $\SL_2(\RR)\cap g\Gamma g^{-1}$ is either a lattice in $\SL_2(\RR)$ or has critical exponent $\delta(\SL_2(\RR)\cap g\Gamma g^{-1}) \leq 1 - \varepsilon_\Gamma$.
\end{abstract}

\maketitle
\section{Introduction}
In his landmark work, Margulis \cite{margulis1977discrete} showed that there are no irreducible nonarithmetic lattices in higher-rank semisimple Lie groups (see Definition \ref{def: homogeneous dynamics} of arithmetic lattice). However, there are nonarithmetic lattices in $\SL_2(\RR)$ and $\SL_2(\CC)$ and more generally $\SO(n,1)$ for $n\geq 2$. This paper focuses on nonarithmetic lattices in $G = \SL_2(\CC)$.
There are several constructions for such lattices. One such construction is given by Gromov and Piatetski-Shapiro \cite{gromov1987non} as the fundamental group of a certain surgery of two arithmetic hyperbolic manifolds.
Some other constructions can be found e.g.\ in \cite{vinberg1967discrete, reid1991arithmeticity}.

A recent result of Mohammadi and Margulis \cite{mohammadi2022arithmeticity} and of Bader, Fisher, Miller and Stover \cite{bader2021arithmeticity} gives a geometric sufficient criterion for a lattice $\Gamma<G$ to be arithmetic, namely, if there are infinitely many totally geodesic surfaces in $\HH^3/\Gamma$. This is equivalent to $G/\Gamma$ having infinitely many periodic $\SL_2(\RR)$-orbits.
The Bader, Fisher, Miller and Stover result is more general in that it deals with lattices in $SO(n, 1)$ for any $n\ge 3$.

A notable distinction between arithmetic and nonarithmetic lattices is the following. Let $\SL_2(\RR).x$ be an orbit for some $x\in G/\Gamma$.
If $\Gamma$ is arithmetic and $\stab_{\SL_2(\RR)}(x)$ is Zariski dense in $\SL_2(\RR)$, then a theorem of Borel and Harish-Chandra \cite{BorelHarishChandra} says that $\stab_{\SL_2(\RR)}(x)$ must be a lattice as well. In the nonarithmetic case, this is no longer true.
One can quantify the ``size'' of a Zariski dense subgroup $\Lambda < \SL_2(\RR)$ by its critical exponent:
\begin{definition}[Critical exponent]\label{def: critical exponent}
  Let $\Lambda < \SL_2(\RR)$ be a discrete subgroup. Define the \emph{critical exponent} of $\Lambda$ by
  \[\delta(\Lambda) = \limsup_{R\to \infty}\frac{\log \#(B_{\SL_2(\RR)}(R)\cap \Lambda)}{R}.\]
  Here $B_{\SL_2(\RR)}(R)$ is a ball in $\SL_2(\RR)$ around the identity, with respect to the natural metric $d_{\SL_2(\RR)}$ which will be specified in Section \ref{sec: notations}.
  Alternatively, $\delta(\Gamma)$ is the abscissa of convergence of the Poincare series $L_\Gamma(s) = \sum_{\lambda \in \Lambda} d_{\SL_2(\RR)}(\lambda, I)^{-s}$. 
  If $\Lambda$ is a lattice, then $\delta(\Lambda) = 1$, and this is the maximal possible critical exponent.
\end{definition}


The main result of this paper is the following theorem.

\begin{theorem}[Critical exponent gap]\label{thm: main}
  Let $\Gamma<G$ be a geometrically finite Kleinian group.
  For every $g\in G$ define \[\Gamma_g = \SL_2(\RR)\cap g\Gamma g^{-1} = \stab_{\SL_2(\RR)}(\pi_\Gamma(g)),\] where $\pi_\Gamma(g)$ image of $g$ in $G/\Gamma$, and consider the critical exponent
  $\delta(\Gamma_g)$.
  Then there is an $\varepsilon_\Gamma>0$ such that for all $g\in G$ one of the following holds:
  \begin{enumerate}
    \item $\delta(\Gamma_g) \le 1-\varepsilon_\Gamma$;
    \item $\Gamma_g$ is a lattice.
  \end{enumerate}
\end{theorem}
To show that this $\varepsilon_\Gamma$ cannot be chosen uniformly even for nonarithmetic lattices $\Gamma$, we prove the following.
\begin{theorem}\label{thm: example}
  For every $\varepsilon>0$ there is a nonarithmetic lattice $\Gamma<G$ and $g\in G$ such that $\Gamma_g$ is not a lattice but $\delta(\Gamma_g) > 1-\varepsilon$.
\end{theorem}
\begin{remark}
  It seems likely that in the homogeneous space $G/\Gamma$ we construct in Theorem \ref{thm: example} there are infinitely many orbits $\SL_2(\RR).\pi_\Gamma(g)$ so that $\delta(\Gamma_g) > 1-\varepsilon$, but we do not know how to show it.
\end{remark}
\begin{remark}
  A gap in critical exponent was shown by Phillips and Sarnak \cite{phillips1985laplacian} for Schottky groups in $\SO(n, 1)$. 
\end{remark}
\subsection{Application to polynomial equidistribution}
We will relate Theorem \ref{thm: main} to a recent result of Lindenstrauss, Mohammadi, and Wang \cite{lindenstrauss2022effective}. Let $\Gamma<G$ be a lattice.
Let \[\tu(s) = \begin{pmatrix}
    1 & s \\0&1
  \end{pmatrix}, \quad\ta(t) = \begin{pmatrix}
    e^{t/2} & 0 \\0&e^{-t/2}
  \end{pmatrix}, \qquad\forall s,t\in \RR,\]
and let $x\in G/\Gamma$.
Ratner's Equidistribution Theorem (See \cite{ratner1990strict,ratner1990acta,ratner1991raghunathan}) shows that $\tu(s).x$ equidistributes in some homogeneous subspace of $G/\Gamma$. More formally, the sequence of measures \[\mu_{T,x} = \frac{1}{T}\int_{0}^T\delta_{\tu(s).x}\bd s\] converges to the Haar measure on a homogeneous subspace. Moreover, unless $x$ lies in a $\tu(s)$-invariant homogeneous subspace, $\mu_{T,x} \xrightarrow{T\to \infty} m_{G/\Gamma}$.
Lindenstrauss, Mohammadi, and Wang \cite{lindenstrauss2022effective} effectivized this claim whenever $\Gamma$ is arithmetic. \cite[Thm.~1.1]{lindenstrauss2022effective} can be seen to be equivalent to the effectivization of Ratner's Equidistribution Theorem. Informally and inaccurately, it states that either
$\ta(t)\mu_{1,x}$ is $\exp(-\star t)$ close to the Haar measure $m_{G/\Gamma}$ or one of the following algebraic obstructions occurs:
\begin{itemize}
  \item
        $\tu(s)x$ is $\exp(-\star t)$ close to a periodic orbit
        $\SL_2(\RR).x'$ of volume $\exp(-\star t)$ for all $t\ge 0$.
  \item $x$ lies too deep in a cusp of $G/\Gamma$.
\end{itemize}

Lindenstrauss, Mohammadi and Wang also give a version of their theorem for nonarithmetic lattices in $G$ (\cite[Thm.~1.3]{lindenstrauss2022effective}), but its statement is more complicated as it cites another, more complicated, type of obstruction, unrelated to periodic $\SL(2,\RR)$-orbits or cusp excursions, e.g.~that the initial point is close to a point $\pi_\Gamma(g)$ for which $\Gamma_g$ is Zariski dense but not a lattice. And indeed, this potentially is an obstruction:
suppose that $x\in G/\Gamma$ has a stabilizer $\Lambda = \stab_{\SL_2(\RR)}x < \SL_2(\RR)$ with critical exponent
$\delta(\Lambda)$, and suppose $\delta(\Lambda)$ is very close to $1$ (its maximal value).
This allows $\ta(t)\tu(s).x$ to return $\Theta(\exp(\delta(\Lambda)t))$ times to a ball $B_{\SL_2(\RR)}(1).x$ for $s\in [0,1]$. This gives an lower bound of $\exp((\delta(\Lambda)-1)t)$ on the distance of $\mu_{T,x}$ and the Haar measure $m_{G/\Gamma}$.

Therefore, if one wants also for a nonarithmetic lattice a polynomial equidistribution theorem analogous to \cite[Thm.~1.1]{lindenstrauss2022effective}, the first step is to bound $\delta(\Lambda)$, which is done by Theorem \ref{thm: main}.
In a follow-up to this paper, we will show the following polynomial unipotent equidistribution result.
\begin{corollary}[Polynomial unipotent equidistribution in nonarithmetic $G/\Gamma$]\label{cor: main}
  Let $\Gamma<G$, be a nonarithmetic lattice.
  For every $x_0 \in G/\Gamma$ and
  large enough $R$ (depending only on $\Gamma$ and the injectivity radius of $x_0$), for any $T \ge R^{A}$, at least one
  of the following holds.
  \begin{enumerate}
    \item For every $\varphi \in C^\infty_c(G/\Gamma)$,
          \[\left|\int_0^1\varphi(\ta(\log T)\tu(r)) \bd r - \int_{G/\Gamma}\varphi \bd m_{G/\Gamma}\right| < S(\varphi)R^{-\kappa_1}.\]
          where $S(\varphi)$ is a certain Sobolev norm.
    \item There exists $x_1\in G/\Gamma$ such that the orbit $\SL_2(\RR) x_1$ is periodic
          and
          \[d_{G/\Gamma}(x_0, x_1) < R^A(\log T)^AT^{-1}.\]
  \end{enumerate}
  The constants $A$ and $\kappa_1$ are positive and depend on $\Gamma$ but not on $x_0$.
\end{corollary}
The Sobolev norm used here is the same as in \cite{lindenstrauss2022effective}.
See \cite{mohammadi2022arithmeticity}, \cite{bader2021arithmeticity} for a finiteness result of the periodic orbits in Option $2$ in the above corollary.
\begin{remark}\label{rem: high crit lattices}
  The proof of Theorem \ref{thm: main} is via a limiting argument, and uses Ratner's Measure Classification Theorem. In the lattice case it invokes also elements of \cite{bader2021arithmeticity}. Hence, Theorem \ref{thm: main} is not effective, which implies the same regarding the constants in Theorem \ref{thm: main} and Corollary \ref{cor: main}.
  Similarly, the results of in \cite{mohammadi2022arithmeticity} and \cite{bader2021arithmeticity} prove that for a nonarithmetic lattice $\Gamma$ there are only finitely many $\SL_2(\RR)$ periodic orbits without any estimate on their number. In contrast, the constants in \cite{lindenstrauss2022effective} are explicit in principle; cf.\ also \cite[Thm.~1.4]{lindenstrauss2023polynomial}.
\end{remark}
\subsection{Structure of the proof of the gap in critical exponent}
\label{ssec: on the proof}
As mentioned above, the proof of Theorem \ref{thm: main} is based on an ergodic theoretic arguments, and in the nonarithmetic lattice case, it uses also results from \cite{bader2021arithmeticity}. 
In \cite{bader2021arithmeticity}, the first step assumes to the contrary that there is a nonarithmetic lattice $\Gamma<G$ for which there are infinitely many periodic orbits $(\SL_2(\RR).x_k)_{k=1}^\infty$ for $x_k = \pi_\Gamma(g_k)\in G/\Gamma$.
Then, using Ratner's theorem (or more precisely a result of Mozes and Shah that relies on this theorem as well as the Dani-Margulis linearization method), the authors show that the sequence of Haar measures on these periodic orbits converges to the Haar measure, i.e.
\[m_{\SL_2(\RR).x_k}\xrightarrow{k\to \infty} m_{G/\Gamma}.\]
In our case, $\Gamma_{g_k} = \stab_{\SL_2(\RR)}(x_k)$ are not lattices, so the Haar measures on them are infinite.
Instead, we will use for each $k$ the Bowen-Margulis-Sullivan measure $\mu_k$ corresponding\footnote{In fact, we use Bowen-Margulis-Sullivan measures corresponding to finitely generated subgroups of $\Gamma_{g_k}$ so that the measure will be finite. We ignore this subtlety for the introduction.} to $\Gamma_{g_k}$ on $\SL_2(\RR).x_k$.
It has an entropy $h_{\mu_k}(\ta(1)) = \delta(\Gamma_{g_k})$.
An $\ta$-invariant measure on $G/\Gamma$ can have any entropy $\leq 2$, so these entropies are certainly far from being maximal entropy. Thus we cannot show that such a measure is close to Haar using only the uniqueness of measure of maximal entropy on $G/\Gamma$ (See e.g. \cite{bowen1973maximizing}, \cite[\S11]{viana2016foundations}).
However, all of the entropy of these $\mu_k$ comes ``from the $\SL_2(\RR)$ direction''. This intuition can be formalized to say that the $\tu$-leafwise dimension (see Definition \ref{def: leafwise dimension}) of $\mu_k$ is almost everywhere $\delta(\Gamma_{g_k})$ which is close to the maximal value $1$.
This leads us to the ergodic component of the proof, Theorem \ref{thm: dimension implies invariance} below. This theorem enables us to utilize this large dimension to show that any weak-$*$ limit is $\SL_2(\RR)$-invariant, and is of interest by itself.
However, there is yet work to be done to show that the limit is the Haar measure $m_{G/\Gamma}$, as there may be an escape of mass, or positive mass to $\SL_2(\RR)$-periodic orbit.
To rule out these options we use linearization methods and Margulis functions.
Once we show that the limit is the Haar measure, we get a contradiction in the infinite volume case, and use \cite{bader2021arithmeticity} for the nonarithmetic lattice case. 

\subsection{Structure of the construction of a lattice \texorpdfstring{$\Gamma$}{Gamma} with small gap}
As for Theorem \ref{thm: example}, its proof can be divided to three main parts. 

\textbf{Part 1, construction of a homogeneous space:} We implement a construction of a nonarithmetic lattice given by Gromov and Piatetski-Shapiro \cite{gromov1987non}, who construct a nonarithmetic space $G/\Gamma$ in the following way:
Take two (carefully constructed) arithmetic spaces $G/\Gamma_1$ and $G/\Gamma_2$, and identify isomorphic codimension-$1$ submanifold $V_i\subseteq \HH^3/\Gamma_i$. 
Cut $\HH^3/\Gamma_i$ along $V_i$ for each $i=1,2$ to obtain two hyperbolic threefolds with isomorphic boundaries. Finally, glue these manifolds along their boundaries to obtain a compact hyperbolic threefold of the form $\HH^3/\Gamma$. The lattice $\Gamma$ is the non-arithmetic manifold we look for.

\textbf{Part 2, construction of an orbit:} 
In this part we construct a certain orbit in a similar way to Benoist and Oh \cite[\S12.5]{benoist2022geodesic}.
To construct the element $g$ required by Theorem \ref{thm: example}, we construct the orbit $\SL_2(\RR).\pi_\Gamma(g)$ as follows. Take a periodic $\SL_2(\RR)$-orbit $H.x_0$ in $G/\Gamma_i$ for some $i=1,2$.
Denote its projection to $G/\Gamma_i$ by $S_0$. This is an immersed hyperbolic surface. 
Cut $S_0$ along $V_i$ into finitely many pieces, and consider the image $S_1$ in $\HH^3/\Gamma$ of one of these pieces. This yields an immersed hyperbolic surface 
$S_2 \cong \HH^2 / \Gamma_{g_2}\subset \HH^3/\Gamma$ contains $S_1$. We show that by properly choosing $S_0$ and $S_1$ we can ensure that $\Gamma_{g_2}$ is not a lattice.

\textbf{Part 3, estimation of the critical exponent in the form of high Hausdorff dimension: }
The estimation of the critical exponent of $\Gamma_{g_2}$ uses Sullivan \cite[Thm.~1]{sullivan1984entropy}, which reduces the estimation of the desired critical exponent to an estimation of the Hausdorff dimension of the collection of geodesic in $S_0$ that originates form a point $p_0\in S_0$ and do not intersect $V_i$. 
Viewing this problem in the universal cover of $S_0$, the inverse image of $V_i$ is a union of geodesics. This reduces the question of giving a lower bound on the Hausdorff Dimension of set of rays from $p_0$ on $S_0$ avoiding $V_i$ to the following two claims on an immersion $\iota_0:\HH^2\to \HH^3/\Gamma_i$.
\begin{itemize}
  \item $\iota_0^{-1}(V_i)$ is composed of many hyperbolic lines in $\HH^2$. Then there is a lower bound on the distances of these lines from one another,
  \item For every collection $L$ of lines in $\HH^2$ that are far from one another, and every point $p_0\in \HH^2$ not on any of these lines, the dimension of the set of geodesic rays from $p_0$ that do not hit any of the lines is large. 
\end{itemize}
The first point follows from arithmetic considerations.
The second can be reduced to an estimate of the dimension of a certain Cantor set.
\subsection{Structure of the paper}
In Section \ref{sec: notations} we introduce several notations and conventions.
Section \ref{sec: macroscopic flavor} is divided into three parts: In Subsections \ref{ssec: leafwise measures} and \ref{ssec: leafwise dimension} we recall a nonstandard definition for the leafwise measures and some of its properties.
In Subsection \ref{ssec: Markov chain} we introduce the leafwise Markov chain and complete the proof of Theorem \ref{thm: dimension implies invariance}.

In Section \ref{sec: proof of rigidity theorems} we prove Theorem \ref{thm: main}.
The section is divided into three parts:
Subsection \ref{ssec: lemmas for rigidity} states several claims which will be of use in the next subsection.
Subsection \ref{ssec: measure on G mod Gamma} follows the discussion at Subsection \ref{ssec: on the proof}, and shows that a certain sequence of measures $\mu_k$ on $G/\Gamma$ converges to the Haar measure. We prove $\SL_2(\RR)$-invariance, use Lemma \ref{lem: non-degenerate limits} (whose proof is left to Section \ref{sec: proof of geometric lemma}) to exclude escape of mass in the limit and any nontrivial homogeneous component, and finally use Ratner's theorem \cite{ratner1991raghunathan}
to conclude that the limit is the Haar measure.
In Subsection \ref{ssec: using rigidity} we conclude the proof of Theorem \ref{thm: main} by adapting the work of \cite{bader2021arithmeticity} to our setup.
In Section \ref{sec: proof of geometric lemma} we use a linearization method and Margulis functions to prove Lemma \ref{lem: non-degenerate limits}.
Section \ref{sec: example}, which is independent of the rest of the paper, is dedicated to the proof of Theorem \ref{thm: example}.


\subsection{Further research}
A natural question is to prove an effective version of Theorem \ref{thm: main}.
\begin{ques}
  Find an effective formula for an $\varepsilon_\Gamma$ depending on the lattice $\Gamma<G$ such that there are only finitely many $\SL_2(\RR)$-orbits $\SL_2(\RR).\pi_\Gamma(g)$ such that $1-\varepsilon_\Gamma < \delta(\Gamma_g)$ but $\Gamma_g$ is not a lattice.
  We expect $\varepsilon_\Gamma$ to depend on the spectral gap of $G/\Gamma$, however, it may depend also on the arithmetic nature of $\Gamma$.
\end{ques}
The example given by Theorem \ref{thm: example} inspires us to formulate the following more optimistic question. We do not know if it helps to answer the previous one.
\begin{ques}
  Let $\Gamma<G$ be a lattice.
  Is it true that there are only finitely many $\SL_2(\RR)$-orbits of points $x=\pi_\Gamma(g)$ in $G/\Gamma$ with $\Gamma_g$ Zariski dense in $\SL_2(\RR)$ for which there does not exist an arithmetic lattice $\Gamma_1 < G$ such that $\SL_2(\RR).x$ lift bijectively to $G/\Lambda$ where $\Lambda = \Gamma_1\cap \Gamma$ and $\Lambda$ is Zariski-dense in $G$?
  Can one find a finite collection $\varpi$ of arithmetic lattices such that for every point $x=\pi_\Gamma(g)$ in $G/\Gamma$ with Zariski dense $\Gamma_g$ the  orbit $\SL_2(\RR).x$ lifts to $G/\Lambda$ with $\Lambda = \Gamma_1\cap \Gamma$ and $\Gamma_1\in \varpi$, except perhaps for finitely many $\SL_2(\RR)$-orbits?
\end{ques}
One can also consider the analogous of Theorem \ref{thm: main} to geometrically finite groups in more general Lie groups.

We now discuss a possible extension of Theorem \ref{thm: dimension implies invariance} referred to above, and use similar notations. Let $B = \RR \ltimes \RR$ using the exponent action of $\RR$ on $\RR$.
\begin{ques}\label{ques: dimension preservation}
  Let $(\mu_k)_{k=1}^\infty$ be $\ta$-invariant and ergodic probability measures on a locally compact second countable space $X$ on which $B$ acts continuously.
  Suppose that there is a weak-$*$ probability measure limit $\mu_\infty = \lim_{k\to\infty}\mu_k$ with ergodic decomposition $\int_X\mu_\infty^x \bd \mu_\infty(x)$.
  Show that
  \begin{align}\label{eq: dimension bound}
    \int_X \dim^\tu\mu_\infty^x \bd \mu_\infty(x) \ge \limsup_{k\to \infty}\dim^\tu\mu_k,
  \end{align}
  with the convention that $\dim^\tu\mu_\infty^x = 1$ if $\mu_\infty^x$ is not $\tu$-free.
\end{ques}
One can try to extend this to actions of more general semi-direct products.
\subsection*{Acknowledgment}
I thank my advisor, Elon Lindenstrauss, for introducing me to the topic, and for his guidance, support, and constructive feedback throughout the process of writing this paper.
I wish to thank Amir Mohammadi and Hee Oh for helpful comments, and in particular for suggesting that the result could be extended beyond the nonarithmetic lattice case to include the geometrically finite case.
\section{Notations}
\label{sec: notations}
\begin{definition}[Homogeneous dynamics notations]\label{def: homogeneous dynamics}
  Let $G = \SL_2(\CC)$ and $\Gamma < G$ a geometrically finite group, see \cite{bowditch1993geometrical} for various definitions of notion of geometrical finiteness.
  We say that $\Gamma$ is an \emph{arithmetic lattice} if there is an algebraic group $\bG/\QQ$ and a homomorphism with compact kernel $f:\bG(\RR)\to G$ whose image is open in $G$ and $\Gamma$ is commensurable to $f(\bG(\ZZ))$. From now on we assume that $\Gamma$ is not an arithmetic lattice.
  Recall that $\ta(t) = \diag(e^{t/2}, e^{-t/2})$ and
  $\tu(s) = \begin{pmatrix}
      1 & s \\0&1
    \end{pmatrix}$ for all $t,s\in \RR$, generates a subgroup $B<\SL_2(\RR)$.
  The group $B$ is isomorphic to $\RR\ltimes \RR$, with the exponent action.
  Denote by $\pi_\Gamma: G \to G/\Gamma$ the standard projection.
\end{definition}
\begin{definition}[Metric notations]
  For every metric space $X$, we will always denote its metric by $d_X$, and whenever there is a natural base point to the space we denote by $B_X(R)$ a ball of radius $R$ around the base point.

  Let $d_{G}$ be the unique Riemannian metric on $G$ that is right $G$-invariant and left $\SU(2)$-invariant, normalized so that
  \[d_{G}\left(\begin{pmatrix}
        e^{t/2} & 0 \\ 0 & e^{-t/2}
      \end{pmatrix}\right) = |t|,\]
  for all $t\in \RR$.
  This metric restricts to a Riemannian metric $d_{\SL_2(\RR)}$ on $\SL_2(\RR)$ and gives rise to the standard hyperbolic metrics $d_{\HH^3}, d_{\HH^2}$ on $\HH^3 = \SU(2)\backslash G$ and $\HH^2 = \SO(2)\backslash\SL_2(\RR)$, respectively. This makes $\HH^3$ a right $G$-space and $\HH^2$ a right $\SL_2(\RR)$-space.
  Since $d_G$ is right invariant, it descends to a Riemannian metric $d_{G/\Gamma}$ on $G/\Gamma$.
\end{definition}
\begin{definition}[Measure notations]\label{def: measure notations}
  For every measure space $(X, \mu)$ and a measurable function $f:X\to \RR$, we define $\mu(f) = \int_X f\bd \mu$ and $f \cdot \mu$ the measure $U\mapsto \int_U f\bd \mu$ on $X$.
  For every two measurable spaces $(X,\mu),(Y,\nu)$, we denote by $\mu\times \nu$ the product measure on $X\times Y$.
\end{definition}
\begin{definition}[Law of a random variable]
  Let $(X, \Sigma)$ be a space together with a $\sigma$-algebra. Let $\mu$ be a probability measure on $X$. Whenever we think of $X$ as a probability space, then any measurable function $y:X\to Z$ to any other space $(Z, \cB)$, is called a \emph{random variable}. Denote by $\Law(y) = y_*\mu$ the pushforward probability measure on $Z$.
  For every two random variables $y_1:X\to Z_1, y_2:X\to Z_2$,
  measurable with respect to the $\sigma$-algebras $\cB_1, \cB_2$ on $Z_1, Z_2$ respectively,
  we define a random variable
  $\Law(y_1|y_2):X\to \{\text{probability measures on }Z_1\}$ as follows.
  Let $\cB_2' = y_2^{-1}\cB_2$ be the $\sigma$-algebra of all the information on $X$ given by $y_2$. Let $x\mapsto \mu^x_{\cB_2'}$ be the conditional measure, and $\Law(y_1|y_2)(x) = (y_1)_*\mu^x_{\cB_2'}$.
  Similarly, if $y_1, y_2,\dots,y_n$ are random variables, then we denote $\Law(y_1|y_2,y_3,\dots,y_n) = \Law(y_1|(y_2,y_3,\dots,y_n))$, where $(y_2,y_3,\dots,y_n)$ is the tuple random variable.
\end{definition}
\begin{definition}[Entropy notations]
  For every $p_1,...,p_k \in [0,1]$ with $p_1 + p_2 + \dots + p_k = 1$, denote
  \[H(p_1,\dots,p_k) = -\sum_{i=1}^k p_i \log p_i.\]
  For every space $X$, a measure $\mu$ on $X$ with countable support, denote 
  \[H(\mu) = -\sum_{p\in \supp(\mu)} \mu(\{p\})\log \mu(p).\]
  Removing the countable support assumption, let $\tau$ be a partition of $X$. Denote
  \[H_{\mu}(\tau) = -\sum_{A\in \tau}\mu(A)\log \mu(A).\]
  For every $x\in X$ denote by $[x]_\tau$ the unique element in $\tau$ containing $x$.
  Now, whenever we think of $X$ as a probability space and function from $X$ as random variables,
  let $y: X\to S$ be a random variable with a countable image. We denote \[H(y) =H(y_*\mu) = H_\mu(\{y^{-1}(x):x\in {\rm Im}(y)\}). \]
  Let $y_1, y_2$ be two random variables, such that given $y_2$, the random variable $y_1$ has countably many options, that is, $\Law(y_1|y_2)$ is almost surely a measure with countable support.
  Then we denote $H(y_1|y_2) = \int_X H(\Law(y_1|y_2))\bd \mu y_2$.
  Similarly, if $y_1,\dots,y_n$ are random variable, then we define
  \[H(y_1|y_2, y_3,\dots,y_n) = H(y_1|(y_2, y_3,\dots,y_n)) = \int_X H(\Law(y_1|y_2,y_3,\dots, y_n))\bd \mu y_2.\]
\end{definition}
\section{Leafwise measures}
\label{sec: macroscopic flavor}
The purpose of this section is to prove Theorem \ref{thm: dimension implies invariance} below.
We will define the leafwise measures and leafwise dimension in Subsections \ref{ssec: leafwise measures} and \ref{ssec: leafwise dimension}, and recall some of their properties.
At the end of Subsection \ref{ssec: leafwise dimension} we state Theorem \ref{thm: dimension implies invariance}. In subsection \ref{ssec: Markov chain} we introduce a different approach to the leafwise measures, and use it to prove Theorem \ref{thm: dimension implies invariance}.
\subsection{The leafwise measures}
\label{ssec: leafwise measures}
Let $X$ be a locally compact second countable (LCSC) space, $\RR\acts X$ be a continuous action via $\tu(s): X \to X$ for every $s\in \RR$. 
Let $\mu$ be a measure on $X$, not neccessarily $\tu$-invariant.
We assume that $\mu$ is \emph{$\tu$-free}, that is, $\mu$-almost every point $x\in X$ has $\stab_\tu(x) = \{0\}$.
We recall the following characterization for leafwise measures, which is equivalent to the one given in \cite[\S3]{einsiedler2023rigidity}.
\begin{definition}[Anti-convolution of a function with a measure]\label{def: anti-convolution}
  Let $f:\RR \to \RR$ be a nonnegative function with $\int_\RR f(s)\bd s = 1$.
  Denote
  \[S_f \mu = \int_{\RR}f(s)\tu(-s).\mu \bd s.\]
\end{definition}
\begin{theorem}[Fubini construction of Leafwise measures]\label{thm: Fubini construction of Leafwise measures}
  There is a measurable map $y\mapsto \mu_y^\tu$ which associates to every $y\in X$ a locally finite measure $\mu_y^\tu$ on $\RR$, which satisfies the following properties:
  \begin{enumerate}
    \item for every $s\in \RR$ and $y\in X$ we have
          \begin{align}\label{eq: translation invariance of leafwise}
            \mu_{\tu(s).y}^\tu \propto T^s_*\mu_y^\tu, \quad\text{where}\quad T^s:\RR\to \RR,\quad T^s(r) = r-s.
          \end{align}
    \item
          Let           
          \begin{align}\label{eq: def omega}
                      \omega = F_*(\mu \times m_\RR), \quad\text{where}\quad F:X \times \RR\to X \times \RR, \quad F(x,s) = (\tu(-s)x, s),
          \end{align}
          where $\mu \times m_\RR$ is the product measure on $X\times \RR$.
          Let $f: \RR \to \RR$ be a nonnegative integrable function with $\int_{\RR} f(s)\bd s = 1$, and set $\tilde f:X\times \RR\to \RR$ be defined by $\tilde f(x,s) = f(s)$. 
          For $S_f \mu$-almost every $y\in X$, we have
          \begin{align}
            \label{eq: integral defined}
            0                           & < \mu^\tu_y(f) < \infty\qquad \text{and}                                      \\
            \label{eq: disintegration formula}
            \tilde f\cdot\omega & = \int_{X}\delta_y\times\frac{f\cdot \mu^\tu_y}{\mu^\tu_y(f)} \bd S_f \mu(y).
          \end{align}
          Here we use the notations regarding measures from \S\ref{def: measure notations}, and $\pi_\RR:X \times \RR\to \RR$ is the projection.
  \end{enumerate}
  The map $y\mapsto \mu_y^\tu$ is unique in the following sense.
  If $y\mapsto \mu_y^{\tu,1}$ and $y\mapsto \mu_y^{\tu,2}$ are maps satisfying the above conditions, then for $\mu$-almost every $x\in X$ we have $\mu_y^{\tu,1}\propto \mu_y^{\tu,2}$.
\end{theorem}

\medskip
\noindent
\begin{remark}[On Eq. \eqref{eq: disintegration formula}]
  A different way to write the left-hand side is 
  \[(f\circ \pi_\RR)\cdot\omega = F_*(\mu \times (f\cdot m_\RR)),\]
  where $F$ is as in Eq. \eqref{eq: def omega}. 

  An alternative way to write the formula \eqref{eq: disintegration formula} is that for every compactly supported function and continuous function $g:X \times \RR \to \RR$
           \begin{align}\tag{\ref*{eq: disintegration formula}$'$}
             \int_{X\times \RR}f(s)g(x,s) \bd \omega(x,s) = \int_X \frac{1}{\mu^\tu_y(f)}\int_\RR g(y,s) f(s) \bd \mu^\tu_y(s) \bd S_f\mu(y),\quad\forall g\in C_c(X\times \RR).
           \end{align}
\end{remark}          
\begin{corollary}
  In the notations of Theorem \ref{thm: Fubini construction of Leafwise measures},
  note that $F^{-1}_* ((f\circ \pi_\RR)\cdot\omega) = (f\circ \pi_\RR)\cdot (\mu\times m_\RR)$. Applying this to Eq. \eqref{eq: disintegration formula} and projecting to $X$, we obtain
  \begin{align}\label{eq: Stationary measure formula}
    \mu = \int_{X}\frac{(s\mapsto \tu(s)y)_* (f\cdot \mu^\tu_y)}{\mu^\tu_y(f)}\bd S_f\mu(y).
  \end{align}
\end{corollary}
\begin{proof}[{Reduction of Theorem \ref{thm: Fubini construction of Leafwise measures} to {{\cite[\S3]{einsiedler2023rigidity}}}}]
  We will describe the statement of \cite[\S3]{einsiedler2023rigidity}, restricted to our $\RR$ action.
  Denote by $\cB_X$ the borel $\sigma$-algebra of $X$. 
  Consider the infinite measure $\mu \times m_\RR$ on $X\times \RR$. Define $\Psi:X \times \RR\to X$ by $\Psi(x,s) = \tu(-s).x$ and let $\cC = \Psi^{-1}(\cB_X)$. 
  Let $f_0: \RR \to \RR$ be a positive continuous integrable function. 
  Lift $f_0$ to a map $\tilde f_0:X\times\RR\to \RR$ by $\tilde f_0(x,s) = f_0(s)$. 
  The conditional measures of $\tilde f_0 \cdot (\mu \times m_\RR) = \mu \times (f_0\cdot m_\RR)$ with respect to the $\sigma$-algebra $\cC$ are denoted $(\tilde f_0 \cdot (\mu \times m_\RR))_y^\cC$, for $y\in X\times \RR$. This measure lies on the atom $[y]_\cC$ of $y$ which is of the form $\{(\tu(s).x, s):s\in \RR\} = \Psi^{-1}(x)$ for $x = \Psi(y)\in X$, and $(\tilde f_0 \cdot (\mu \times m_\RR))_y^\cC$ depends only on the atom, that is, only on $x$, and is supported on this atom.
  Define
  $\mu^\tu_x$ on $\RR$ so that 
  \begin{align}\label{eq: disintegration product coordinates}
    (\tilde f_0 \cdot (\mu \times m_\RR))_y^\cC = \tilde f_0\cdot  a^x_*\mu^\tu_x, \quad \text{where}\quad a^x:\RR\to X\times \RR, \quad a^x(s) = (\tu(s).x, s)
    .
  \end{align}
  Then \cite[\S3]{einsiedler2023rigidity} ensures that Eq. \eqref{eq: translation invariance of leafwise} holds in a $\tu$-invariant set $X'\subseteq X$ with $\mu(X')=1$. 
  
  We will now deduce our formulation of the result.
  To ensure that Eq. \eqref{eq: translation invariance of leafwise} holds everywhere, 
  we redefine $\mu^\tu_x := 0$ for $x\nin X'$. This implies that Eq. \eqref{eq: translation invariance of leafwise} holds for all $x\in X$.

  As to the second condition, 
  Eq. \eqref{eq: disintegration product coordinates} implies that $\tilde f_0 \cdot a^x_*\mu^\tu_x = a^x_*(f_0\cdot \mu^\tu_x)$ is a probability measure for all $x\in X'$. 
  That is, $\mu^\tu_x(f_0) = 1$. 
  In addition, 
  \begin{align}\label{eq: pre integral formula}
    \begin{split}
      \tilde f_0 \cdot (\mu \times m_\RR) &= 
      \int_X (\tilde f_0 \cdot (\mu \times m_\RR))_y^\cC \bd (\tilde f_0 \cdot (\mu \times m_\RR))(y) 
      \\&= 
      \int_{X} \tilde f_0 \cdot a^x_*\mu^\tu_x \bd \Psi_*(\tilde f_0 \cdot (\mu \times m_\RR))(x).
    \end{split}
  \end{align}
  To simplify Eq. \eqref{eq: pre integral formula}, first notice that \[\Psi_*(\tilde f_0 \cdot (\mu \times m_\RR)) = \Psi_*(\mu \times (f_0\cdot m_\RR)) = S_{f_0}\mu.\]
  Second, we can multiply Eq. \eqref{eq: pre integral formula} by $\tilde f_0^{-1}$ and obtain
  \begin{align}\label{eq: pre formula f0}
    \mu \times m_\RR = \int_X a^x_*\mu^\tu_x \bd S_{f_0}\mu(x).
  \end{align}
  Applying $F_*$ to Eq. \eqref{eq: pre formula f0} we obtain 
  \begin{align}\label{eq: formula f0}
    F_*(\mu \times m_\RR) = \int_X \delta_x\times  \frac{\mu^\tu_x}{\mu^\tu_x(f_0)} \bd S_{f_0}\mu(x),
  \end{align}
  where the denominator could be added since it is almost surely the constant $1$. 
  This formula is equivalent to Eq. \eqref{eq: disintegration formula}, for $f=f_0$, after multiplying with $\tilde f_0$. 
  To obtain it for general nonnegative $f\in L^1(\RR)$ with $\int_\RR f(x) \bd x = 1$, we multiply Eq. \eqref{eq: pre formula f0} with the function $\tilde f:X\times \RR\to \RR$ defined by $\tilde f(x,s) = f(s)$:
  \begin{align}\label{eq: disintegration changing f0}\begin{split}
    F_*(\mu \times (f\cdot m_\RR)) 
    &= 
    \tilde f \cdot F_*(\mu \times m_\RR)
    = \int_X \delta_x\times \frac{f\cdot \mu^\tu_x}{\mu^\tu_x(f_0)} \bd S_{f_0}\mu(x).
  \end{split}
  \end{align}
  From Eq. \eqref{eq: disintegration changing f0} we deduce that for $S_{f_0}\mu$-almost every $x\in X$ we have $f\cdot \mu^\tu_x$ is a finite measure. Since $f\cdot \mu^\tu_x = 0$ if and only if $\mu^\tu_x(f) = 0$ we deduce that we may restrict the integral in the right-hand side of Eq. \eqref{eq: disintegration changing f0} to $X_0 = \{x\in X:\mu^\tu_x(f) = 0\}$. 
  \begin{align}\label{eq: disintegration changing f1}\begin{split}
    F_*(\mu \times (f\cdot m_\RR)) 
    &
    = \int_{X_0} \delta_x\times \frac{f\cdot \mu^\tu_x}{\mu^\tu_x(f_0)} \bd S_{f_0}\mu(x)
      = \int_{X_0} \delta_x\times \frac{f\cdot \mu^\tu_x}{\mu^\tu_x(f)} \frac{{\mu^\tu_x(f)}}{{\mu^\tu_x(f_0)}} \bd S_{f_0}\mu(x)
      \\&= \int_{X_0} \delta_x\times \frac{f\cdot \mu^\tu_x}{\mu^\tu_x(f)} \bd \mu_f(x)
      = \int_{X} \delta_x\times \frac{f\cdot \mu^\tu_x}{\mu^\tu_x(f)} \bd \mu_f(x),
  \end{split}
  \end{align}
  where 
  \[\mu_f = \int_X \delta_x \bd \frac{{\mu^\tu_x(f)}}{{\mu^\tu_x(f_0)}} \bd S_{f_0}\mu(x), \]
  and last equality of Eq. \eqref{eq: disintegration changing f1} holds since $\mu_f$ is supported on $X_0$.
  To compute $\mu_f$, project Eq. \eqref{eq: disintegration changing f1} to $X$. 
  The projection of the right-hand side is $\mu_f$. 
  The projection of the left-hand side is $S_f\mu$, and hence $S_f\mu = \mu_f$.
  Therefore, Eq. \eqref{eq: disintegration changing f1} is equivalent to Eq. \eqref{eq: disintegration formula}. 
  Eq. \eqref{eq: integral defined} from the equality $S_f\mu = \mu_f$, the definition of $\mu_f$, and the fact that $S_f\mu$ is a probability measure. 
  
  To show the uniqueness of the measures $\mu_x^\tu$, note that we have established an equivalence between Eq. \eqref{eq: disintegration formula} applied to $f_0$ and 
  \begin{align*}
    (\tilde f_0 \cdot (\mu \times m_\RR))_y^\cC = f_0\cdot  \frac{a^x_*\mu^\tu_x}{\mu^{\tu}_x(f_0)}.
  \end{align*}
  for $\tilde f_0 \cdot (\mu \times m_\RR)$-almost every $y$ and $x = \Psi(y)$. 
  Since the conditional measures are uniquely defined almost everywhere, we deduce that $\mu^\tu_x$ is uniquely defined $S_{f_0}$ almost surely.
\end{proof}
The following claims follow from the uniqueness of the characterization of Theorem \ref{thm: Fubini construction of Leafwise measures}.
\begin{claim}[Equivariance of Leafwise measures]\label{claim: action equivariance of leafwise}
  If $\alpha:X\to Y$ is an injective map of LCSC spaces, $\RR\stackrel{\tu'}{\acts}Y$, and
  \[\alpha(\tu(s).x) = \tu'(s).\alpha(x),\qquad \forall x\in X, s\in \RR,\]
  then for $\mu$-almost every $x\in X$,
  \begin{align}\label{eq: leafwise measure pushforward injection}
    \mu^\tu_{x} \propto (\alpha_*\mu)^{\tu'}_{\alpha(x)}.
  \end{align}
  Moreover, there is a set $X'\subseteq X$ that is $\tu$ invariant and has $\mu(X')=1$ such that 
  Eq. \eqref{eq: leafwise measure pushforward injection} holds for all $x\in X'$. 
\end{claim}
\begin{claim}[Rescaling of the action]\label{claim: rescaling}
  Suppose that $\beta \ne 0$, and define the rescaled action $\RR\stackrel{\tu'}{\acts}X$ by $\tu'(s).x = \tu(\beta s).x$.
  Then $\mu$-almost every for all $x\in X$ have
  \begin{align}\label{eq: leafwise measure rescaled}
    \mu^{\tu'}_x \propto (s\mapsto \beta^{-1} s)_*\mu^\tu_x.
  \end{align}
  Moreover, there is a set $X'\subseteq X$ that is $\tu$ invariant and has $\mu(X')=1$ such that 
  Eq. \eqref{eq: leafwise measure rescaled} holds for all $x\in X'$. 
\end{claim}
\subsection{Leafwise dimension}
\label{ssec: leafwise dimension}
We will discuss measures $\mu$ on an LCSC space $X$. We require that there is a measurable action $B\acts X$ where $B\cong \RR\ltimes \RR$ is defined as in Homogeneous Dynamics Notations \ref{def: homogeneous dynamics}.
Our measures $\mu$ will be $\ta$-invariant and $\tu$-free, and we will analyze their $\tu$-leafwise measures.
\begin{definition}[Leafwise dimension]\label{def: leafwise dimension}
  Let $X$ be an LCSC space with a continuous action $B\acts X$.
  Let $\mu$ be an $\ta$-invariant $\tu$-free probability measure on $X$. We say that $\mu$ has \emph{$\tu$-leafwise dimension} $\delta$ and write $\dim^\tu(\mu) = \delta$ if for $\mu$-almost-all $x$,
  \begin{align}\label{eq: dimension definition}
    \lim_{t\to \infty} \frac{1}{t} \log \mu^\tu_x([ -e^{-t},e^{-t}]) = -\delta.
  \end{align}
  If $\mu$ is ergodic, then $\dim^\tu(\mu)$ exists. This existence is proved in the homogenous setting in \cite[Thm~7.6(i)]{einsiedler2010diagonal}. However, their proof works for our setting as well.
\end{definition}

One can relate leafwise dimension to entropy.
This will be useful in Section \ref{sec: macroscopic flavor}.
\begin{theorem}[Relation of leafwise dimension to entropy]\label{thm: leafwise dim = entropy}
  Let $\Lambda \subseteq \SL_2(\RR)$ be a discrete subgroup and $\mu$ an $\ta$-invariant and ergodic probability measure on $\SL_2(\RR)/\Lambda$. Then \[h_\mu(\ta(t)) = |t|\dim^\tu(\mu).\]
\end{theorem}
\noindent
This theorem is proved in \cite[Thm~7.6 (ii)]{einsiedler2010diagonal}.

\medskip

The main result of this section is the following:
\begin{theorem}\label{thm: dimension implies invariance}
  Let $(\mu_k)_{k=1}^\infty$ be $\ta$-invariant and ergodic probability measures on an LCSC space $X$ on which $B$ acts continuously. We further assume that for every $k$ the measure $\mu_k$ is $\tu$-free.
  Suppose that the $\tu$-leafwise dimensions
  \begin{align}\label{eq: high dimensions}
    \dim^\tu\mu_k\xrightarrow{k\to \infty}1.
  \end{align}
  Suppose that there is a weak-$*$ probability measure limit $\mu_\infty = \lim_{k\to \infty} (\mu_k)_{k=1}^\infty$.
  Then $\mu_\infty$ is $\tu$-invariant.
\end{theorem}
\subsection{The leafwise Markov chain}
In this subsection, we will prove Theorem \ref{thm: dimension implies invariance}.
To prove this theorem we need some macroscopic way to view the leafwise dimension, as opposed to Definition \ref{def: leafwise dimension}, which views it as a limit of a leafwise measure of very small intervals.
To do this we introduce a Markov chain, somewhat similar to the one introduced in Furstenberg \cite{furstenberg1970intersections}. The properties of the Markov chain are summarized in the following lemma.
\label{ssec: Markov chain}
\begin{lemma}[The leafwise Markov chain]\label{lem: high leafwise entropy meaning}
  Let $X$ be an LCSC topological space, let $B\acts X$ be a continuous action and let $\mu$ be an $\ta$-invariant, ergodic, and $\tu$-free probability measure.
  Then there is a function $p:X\to [0,1]$ such that
  \begin{align}
    \label{eq: operator entropy}
    \int_{X} H(p(x), 1-p(x)) \bd S_{\ind_{[0,1)}}\mu(x) & = \dim^u(\mu)\log 2,   \\
    \label{eq: operator stationary}
    \int_{X} \omega_x \bd S_{\ind_{[0,1)}}\mu (x)                             & = S_{\ind_{[0,1)}}\mu,
  \end{align}
  where \begin{align}\label{eq: def omega x}
    \omega_x = p(x)\delta_{\ta(\log 2).x} + (1-p(x))\delta_{\tu(1)\ta(\log 2).x}.
  \end{align}
\end{lemma}
Fix an $\ta$-invariant and ergodic $\tu$-free measure $\mu$ on $X$.
Consider the dynamical system with space $X\times [0,1)$, measure $\nu_0 = \mu\times m_{[0,1)}$, and action
$T_0(x,s) = (\ta(\log(2))x, 2s~{\rm mod}~1)$, which preserves $\nu_0$.
Conjugate it by \[F:X\times [0,1)\to X\times [0,1), \qquad F(x,s) = (\tu(-s)x, s).\]
We obtain a dynamical system $(T, X\times [0,1),\nu)$, where
\begin{align}\label{eq: nu from mu}
  \nu = F_*\nu_0 \stackrel{\eqref{eq: disintegration formula}}{=} \int_{X}\delta_y\times \frac{\mu^\tu_y|_{[0,1)}}{\mu^\tu_y([0,1))}\bd S_{\ind_{[0,1)}}\mu(y)
\end{align}
and $T(y, s) = (\tu(b_1(s))\ta(\log 2)x, 2s~{\rm mod}~1)$, where $b_1(s)$ is the $2^{-1}$ bit of the binary point in the binary expansion, $s = \sum_{i=1}^\infty 2^{-i} b_i(s)$, where $b_i(s)\in \{0,1\}$.
Let $p_0 = (y_0, s)$ be a sample point in the probability space $(X\times [0,1), \nu)$.
A different interpretation of Eq. \eqref{eq: nu from mu} is the following almost sure equalities,
\begin{align}
  \label{eq: y0 law} \Law(y_0)           & = S_{\ind_{[0,1)}}\mu,                         \\
  \label{eq: s given y0 law} \Law(s|y_0) & = \frac{\mu^\tu_y|_{[0,1)}}{\mu^\tu_y([0,1))}.
\end{align}
Denote by $\pi_X:X\times [0,1)\to X$ the projection. For every $n\ge 1$ let $p_n = T^np_0$ and for every $n\ge 0$ denote $y_n = \pi_X(p_n)$.
One can see that
\begin{align}\label{eq: y_n explicit}
  y_n = \ta(n\log 2)\tu(s_n)y_0,
\end{align}
where $s_n = \sum_{i=1}^n 2^{-i} b_i(s)$.
Let $X' = \{x\in X:\mu^\tu_{x}([0,1)) > 0\}$. It has a full $S_{\ind_{[0,1)}}\mu$-measure from Eq. \eqref{eq: integral defined}.
Define
\[p(x) = \begin{cases}
    \frac{\mu^\tu_{x}([0,1/2))}{\mu^\tu_{x}([0,1))}, & x\in X',         \\
    0,                                               & \text{otherwise,}
  \end{cases}
\]
and let $\omega_x$ be as in Eq. \eqref{eq: def omega x}.
\begin{claim}\label{claim: is Markov}
  The stochastic process $(y_n)_{n=0}^\infty$ is a stationary Markov process, with the $\nu$-almost always law
  \begin{align}\label{eq: Markov stationary}
    \Law(y_n|y_0, y_{1}, \dots,y_{n-1}) = \Law(y_n|y_{n-1})
    = \omega_{y_{n-1}}.
  \end{align}
\end{claim}
\begin{proof}
  It will be sufficient to show that the RHS and LHS of \ref{eq: Markov stationary} coincide.
  Indeed, by Eq. \eqref{eq: y_n explicit} and Claim \ref{claim: almost no periods},
  $\Law(y_n|y_0, y_1, \dots,y_{n-1}) = \Law(y_n|s_{n-1},y_0)$.
  Now,
  \begin{align}\label{eq: law is almost good}
    \Law(s_n|s_{n-1},y_0) \stackrel{\eqref{eq: s given y0 law}}{=} \begin{cases}
                                                                     s_{n-1},          & \text{with probability }\frac{\mu^\tu_{y_0}([s_{n-1},s_{n-1}+2^{-n}))}{\mu^\tu_{y_0}([s_{n-1}, s_{n-1}+2^{-n+1}))},            \\
                                                                     s_{n-1} + 2^{-n}, & \text{with probability }\frac{\mu^\tu_{y_0}([s_{n-1} + 2^{-n},s_{n-1}+2^{-n+1}))}{\mu^\tu_{y_0}([s_{n-1}, s_{n-1}+2^{-n+1}))}.
                                                                   \end{cases}
  \end{align}
  We wish to apply Claims \ref{claim: action equivariance of leafwise} and \ref{claim: rescaling} to $y_0$. Since $u(s) y_0 \sim \mu$ we deduce that $y_0 \in X'$ almost surely, where $X'$ is one of the sets in Claims \ref{claim: action equivariance of leafwise} and \ref{claim: rescaling}, and hence these claim are applicable.
  Similarly, $y_n\in X'$ almost surely for all $n\ge 0$.
  Apply Claim \ref{claim: action equivariance of leafwise} for the $\ta((n-1)\log 2)$-action,
  which takes the $\tu(s)$ action to $\tu'(s) = \tu(2^ns)$-action.
  Hence
  \begin{align*}
    \mu^\tu_{y_{n-1}}
     &\stackrel{\ref{claim: rescaling}}{\propto}
     (s\mapsto 2^{n-1}s)_*\mu^{\tu'}_{y_{n-1}}
      \stackrel{\eqref{eq: y_n explicit}}{=}
     (s\mapsto 2^{n-1}s)_*\mu^{\tu'}_{\ta((n-1)\log 2)\tu(s_{n-1})y_0}
    \\&
    \stackrel{\ref{claim: action equivariance of leafwise}}{\propto}
    (s\mapsto 2^{n-1}s)_*\mu^\tu_{\tu(s_{n-1})y_0}
    \stackrel{\eqref{eq: translation invariance of leafwise}}{\propto}
    (s\mapsto 2^{n-1}s)_*(s\mapsto s-s_{n-1})_*\mu^\tu_{y_0}
    \\&\,=
    (s\mapsto 2^{n-1}(s-s_{n-1}))_*\mu^\tu_{y_0}.
  \end{align*}
  Thus Eq. \eqref{eq: law is almost good} implies the desired.
\end{proof}
This shows Eq. \eqref{eq: operator stationary}.
\begin{proof}[Proof of Eq. \eqref{eq: operator entropy}]
  Let $C = \int_{X} H(p(x), 1-p(x))\bd S_{\ind_{[0,1)}}\mu(x)$.
  Then $C = H(y_1|y_0)$.
  The Markov chain property implies that
  \[C = H(y_n|y_{n-1}) = H(y_n|y_{n-1},\dots,y_0) = H(b_n(s)|b_{n-1}(s),\dots,b_1(s), y_0).\]
  Hence
  \begin{align}
    \label{eq: accumulated entropy}
    H(b_n(s),b_{n-1}(s),\dots,b_1(s)|y_0) = \sum_{m=1}^n H(b_m(s)|b_{m-1}(s),\dots,b_1(s), y_0) = nC.
  \end{align}
  Denote by $\tau_n = \{[m2^{-n}, (m+1)2^{-n}):m=0,\dots,2^n-1\}$ the partition of $[0,1)$.
  Rewriting Eq. \eqref{eq: accumulated entropy} using Eq. \eqref{eq: s given y0 law}, 
  we obtain
  \begin{align}\label{eq: entropy fixed}
    \int_{X}H_{\frac{\mu^\tu_{y_0}|_{[0,1)}}{\mu^\tu_{y_0}([0,1))}}\left(\tau_n\right)\bd S_{\ind_{[0,1)}}\mu(y_0) = nC.
  \end{align}
  It follows from Lemma \ref{lem: dimension via entropy} below and the Dominated Convergence Theorem that $C = \dim^\tu(\mu)\log 2$.
\end{proof}

\begin{lemma}\label{lem: dimension via entropy}
  For $\mu$-almost-all $x\in X$, and for all $t\in \RR$ such that $\mu^\tu_x([t, t+1])\neq 0$, denote $M_{x,t} = (x\to x-t)_*\frac{\mu^\tu_x|_{[t, t+1]}}{\mu^\tu_x([t, t+1])}$.
  Then $\frac{1}{n}H(M_{x,y}, \tau_n)\xrightarrow{n\to \infty} \dim^\tu(\mu)\log 2$.
\end{lemma}
This lemma will follow from the following Lemma. Let \[\tilde \tau_n = \{[m2^{-n}, (m+1)2^{-n}):m\in \ZZ\},\] denote the partition of $\RR$.
For every $s\in \RR$ denote by $\tilde \tau_m(s)$ the unique partition element in $\tilde \tau_m$ containing $s$.
\begin{claim}\label{claim: pure dimension implies L1 converges}
  Let $\nu$ be a locally finite measure on $\RR$ satisfying that for $\nu$-almost-all $s\in \RR$,
  \begin{align}\label{eq: fixed dimension}
    \frac{1}{t}\log \nu([s-e^t, s+e^t])\xrightarrow{t\to \infty}\Delta.
  \end{align}
  Define the function $u_n:\RR\to\RR$ by $u_n(s) = -\frac{1}{n}\log \nu(\tilde\tau_n(s))$.
  Then $u_n\xrightarrow{m\to \infty} \Delta\log 2$ locally in $L^1(\RR, \nu)$.
  That is,
  for every finite interval $I\subseteq \RR$,
  \[\int _I \left|u_n(s) - \Delta \log 2\right| \bd \nu(s)\xrightarrow{n\to \infty} 0. \]
\end{claim}
\begin{proof}[Proof of Lemma \ref{lem: dimension via entropy} using Claim \ref{claim: pure dimension implies L1 converges}]
  Let \[X_{good, 0} = \left\{x\in X:-\frac{1}{t}\log \mu^\tu_x([ -e^{-t}, e^{-t}])\xrightarrow{t\to \infty}\dim^\tu(\mu)\right\}.\]
  As mensioned in Definition \ref{def: leafwise dimension},  $\mu(X_{good, 0}) = 1$.
  Let
  \[X_{good, 1} = \left\{x\in X: \substack{\text{for $\mu^\tu_x$-almost all }s\in \RR\\\text{we have }\tu(s)x \in X_{good, 0}}\right\}.\]
  This set is $\tu$-invariant by Eq. \eqref{eq: translation invariance of leafwise}. We will now show that $\mu(X_{good, 1}) = 1$. Let $f:\RR\to \RR$ be a positive integrable function with $\int_\RR f\bd s = 1$.
  \begin{align*}
    0 & = \mu(X_{good, 0}^c)
    \stackrel{\eqref{eq: Stationary measure formula}}{=} \int_{X}\frac{\int_\RR f(s_0)\ind_{\tu(s_0)y\in X_{good, 0}^c}\bd \mu^\tu_y(s_0)}{\mu^\tu_y(f)}\bd S_f\mu(y)
    \\&=
    \int_{X}\int_{\RR}\frac{\int_\RR f(s_0)\ind_{\tu(s_0){\tu(s_1)x}\in X_{good, 0}^c}\bd \mu^\tu_{\tu(s_1)x}(s_0)}{\mu^\tu_{\tu(s_1)x}(f)}f(-s_1)\bd s_1\bd \mu(x)
  \end{align*}
  Thus, the positivity of $f$ implies that for $\mu\times m_\RR$-almost all $x, s_1\in X\times \RR$,
  \begin{align}\label{eq: leafwise measure all good}
    \mu^\tu_{\tu(s_1)x}(\{s_0\in \RR: \tu(s_0){\tu(s_1)x}\in X_{good, 0}^c\}) = 0.
  \end{align}
  By Eq. \eqref{eq: translation invariance of leafwise}, Eq. \eqref{eq: leafwise measure all good} is independent of $s_1$. Therefore, for $\mu$-almost-all $x\in X$ we have Eq. \eqref{eq: leafwise measure all good} with $s_1 = 0$, which is equivalent to $x\in X_{good, 1}$. Hence, $\mu(X_{good, 1}) = 1$.

  By Eq. \eqref{eq: translation invariance of leafwise}, for every $x\in X_{good, 1}$ the measure $\mu^\tu_x$ satisfies the condition of Claim \ref{claim: pure dimension implies L1 converges}. This conclusion implies the desired result.

\end{proof}
\begin{proof}[Proof of Claim \ref{claim: pure dimension implies L1 converges}]
  We will first show that it is enough to prove this claim under some simplifying assertions.
  It is sufficient to prove the convergence for intervals $I = (a, a+1)$ with $a\in 1/2\ZZ$.
  Then we may restrict $\nu$ to $I$, while preserving the property for $\nu$-almost all $s\in I$.
  Translating $\nu$,
  we may assume that $\nu$ is supported on $(0,1)$.
  Normalizing $\nu$ to a probability measure does not change the result as well.
  We may now prove that $u_n\xrightarrow{m\to \infty} \Delta\log 2$ in $L^1(\RR,\nu)$, under the assumption that $\nu$ is a probability measure on $(0,1)$.

  Let $\varepsilon>0$ be small numbers and $n$ an integer going to $\infty$.
  By Eq. \eqref{eq: fixed dimension}, for all sufficiently large $n$ there is $S\subseteq (0,1)$ of volume $\nu(S) > 1-\varepsilon$ such that for all $s\in S$ and for all $t\ge n\log 2$,
  \[\left|-\frac{1}{t}\log \nu([s-e^{-t}, s+e^{-t}]) - \Delta\right|<\varepsilon.\]

  We will next show that $\int_{0}^1 (u_n - \Delta\log 2)^-\bd \nu \xrightarrow{n\to \infty}0$, where for every $x\in \RR$ we have $x^+ = \max(x,0)$ and $x^- = (-x)^+$, so that $x=x^+-x^-$.
  Let \[J_- = \{s\in [0,1): u_n(s) < \Delta\log 2 - \varepsilon\}.\]
  Then $J_-$ is a union on $\tau_n$ elements.
  Note that $J_- \cap S = \emptyset$. 
  Indeed, if $s\in J_- \cap S$, then we get the following contradiction:
  \begin{align*}
    \Delta \log 2 - \varepsilon & > u_n(s) = -\frac{1}{n}\log \nu(\tau_n(s))
    \\&
    \ge -\frac{1}{n}\log\nu([x+e^{-n\log 2}, x-e^{-n\log 2}]) \ge (\Delta - \varepsilon) \log 2.
  \end{align*}
  Consequently, $\nu(J_-) \le  1-\nu(S) < \varepsilon$. This implies that
  \begin{align*}
    \int_{0}^1 (u_n - \Delta\log 2)^-\bd \nu & < \int_{J_-} (u_n - \Delta\log 2)^-\bd \nu + \int_{J_-^c} (u_n - \Delta\log 2)^-\bd \nu
    \\&
    \stackrel{u_n \ge 0}{<}\Delta\nu(J_-)\log 2 + \varepsilon(1-\nu(J_-))\le \varepsilon(\Delta\log 2 + 1-\nu(J_-)).
  \end{align*}
  Taking $\varepsilon\to 0$ implies that
  $\int_{0}^1 (u_n - \Delta\log 2)^-\bd \nu\xrightarrow{n\to \infty} 0$.

  We will now show that $\int_{0}^1 (u_n - \Delta\log 2)^+\bd \nu \xrightarrow{n\to \infty}0$.
  Define
  \[J_+' = \bigcup\{I\in \tau_n:\exists I' = I\pm 2^{-n} \text{ with }3\varepsilon \nu(I') > \nu(I)\},\]
  which satisfies
  \begin{align}\label{eq: J + tag volume estimate}
    \nu(J_+') = \sum_{I}\nu(I) < 3\varepsilon\sum_{I,I'}\nu(I') \le 6\varepsilon,
  \end{align}
  where the sums are over $I,I'$ as in the definition on $J_+'$. The rightmost inequality of Eq. \eqref{eq: J + tag volume estimate} holds because each $I_0\in \tau_n$ can appear as $I'$ at most twice.
  Let \[A = \{I\in \tau_n:I\subseteq J_+' \cup S^c\}, \qquad J_+ = \bigcup_{I\in A}I.\]
  We can estimate
  \[
    \zeta := \nu\left(J_+\right) \le \nu(J_+') + \nu(S^c) \le 6\varepsilon + \varepsilon.
  \]

  Thus
  \begin{align}
    \nonumber
    \int_{0}^1 (u_n - \Delta\log 2)^+\bd \nu
                            & =
    \int_{J_+ } (u_n - \Delta\log 2)^+\bd \nu + \int_{J_+^c}(u_n - \Delta\log 2)^+\bd \nu
    \\
    \label{eq: plus part 1} &
    = \sum_{I\in A}\nu(I)\left(-\frac{1}{n}\log \nu(I) - \Delta\right)^+ + \sum_{I\in \tau_n\setminus A}\nu(I)\left(-\frac{1}{n}\log \nu(I) - \Delta\right)^+.
  \end{align}
  To bound the sum over $I\in A$
  \begin{align}\nonumber
    \sum_{I\in A}\nu(I)\left(-\frac{1}{n}\log \nu(I) - \Delta\right)^+ &
    \le
    \frac{1}{n}\sum_{I\in A}-\nu(I)\log \nu(I)
    \\
    \label{eq: plus part 11}
                                                                       & =
    -\zeta \log \zeta
    + \frac{\zeta}{n}\sum_{I\in A}-\frac{\nu(I)}{\zeta}\log \frac{\nu(I)}{\zeta}\le
    -\zeta \log \zeta + \zeta \log 2.
  \end{align}
  The last inequality holds since $(\frac{\nu(I)}{\zeta})_{I\in A}$ is a probability vector on at most $2^n$ and hence the maximal entropy it can have is $n\log 2$.
  For the other summand,
  let $I\in \tau_n\setminus A$. Since $I\nin A$ we deduce that $I \cap J_+' = \emptyset$ and $I\cap S \neq \emptyset$.
  Thus there is $s\in I\cap S$ with $s\nin J_+'$.
  \begin{align}\nonumber
    -\frac{1}{n}\log \nu (I) & = -\frac{1}{n}\log \nu (\tau_n(s))
    \stackrel{s\nin J_+'}{\le} -\frac{1}{n}\log (\varepsilon\nu (\tau_n(s) \cup (\tau_n(s) - 2^{-n})\cup (\tau_n(s) + 2^{-n})))
    \\\label{eq: plus part 12}&
    \le -\frac{1}{n}\log (\varepsilon\nu ([s-2^{-n}, s+2^{-n}]))
    \stackrel{s\in S}{\le} -\frac{1}{n}\left( \log\varepsilon - n(\Delta + \varepsilon) \right) = \Delta + \varepsilon -\frac{\log \varepsilon}{n}
  \end{align}
  Combining Eqs. \eqref{eq: plus part 1}, \eqref{eq: plus part 11}, \eqref{eq: plus part 12} we deduce that for all $n$ sufficiently large
  \[\int_{0}^1 (u_n - \Delta\log 2)^+\bd \nu  \le -7\varepsilon \log (7\varepsilon) + 7\varepsilon\log 2 + \varepsilon -\frac{\log \varepsilon}{n}.\]
  Taking $\varepsilon\to 0$ implies that
  $\int_{0}^1 (u_n - \Delta\log 2)^+\bd \nu\xrightarrow{n\to \infty} 0$.
  The desired follows.
\end{proof}
\begin{proof}[Proof of Theorem \ref{thm: dimension implies invariance}]
  For every $k=0,1,\ldots$, let $p^k, \nu_x^k$ as in Lemma \ref{lem: high leafwise entropy meaning} constructed for $\mu^k$.
  Let $f \in C_c(X)$ be a continuous compactly supported function.
  Then
  \begin{align*}
    S_{\ind_{[0,1)}}\mu^k(f) & - \frac{1}{2}\left(\ta(\log 2).S_{\ind_{[0,1)}}\mu^k\right)(f) - \frac{1}{2}\left(\ta(\log 2).S_{\ind_{[0,1)}}\mu^k\right)(f)
    \\&=
    \int_X \left(\left(p(x) - \frac{1}{2}\right)f(\ta(\log 2).x) + \left(1-p(x) - \frac12\right)f(\tu(1) \ta(\log 2).x)\right) \bd S_{\ind_{[0,1)}}\mu^k(x)  \\
                             & \le \|f\|_\infty
    \int_X \left|p(x) - \frac{1}{2}\right|\bd \mu^k(x)\xrightarrow{k\to \infty}0
  \end{align*}
  The convergence follows from Eq. \eqref{eq: operator entropy}.
  Hence we get an equality of the weak-$*$ limits
  \begin{align}\label{eq: mu infinity stationary}
    S_{\ind_{[0,1)}}\mu^\infty = \frac{1}{2}\ta(\log 2).S_{\ind_{[0,1)}}\mu^\infty + \frac{1}{2}\ta(\log 2).S_{\ind_{[0,1)}}\mu^\infty.
  \end{align}
  inductively applying \eqref{eq: mu infinity stationary} to itself, and using $\ta(\log 2)\tu(1) = \tu(2)\ta(\log 2)$, we get
  \[S_{\ind_{[0,1)}}\mu^\infty = \frac{1}{2^n}\sum_{i=0}^{2^n-1} \tu(i) \ta(n\log 2).S_{\ind_{[0,1)}}\mu^\infty.\]
  Hence $(\tu(1) S_{\ind_{[0,1)}}\mu^\infty - S_{\ind_{[0,1)}}\mu^\infty)(f) \le \frac{2}{2^n}\|f\|_\infty$ for every $f\in C_c(X)$. Taking $n\to \infty$ we deduce that $S_{\ind_{[0,1)}}\mu^\infty$ is $\tu(1)$ invariant, and hence $\tu(k)$-invariant for every $k\in \ZZ$.
  Note that $\ta(-n\log 2)S_{\ind_{[0,1)}}\mu^\infty$ is $\ta(-n\log 2)\tu(k)\ta(n\log 2) = \tu(2^{-n} k)$-invariant, for every $k\in \ZZ$.
  On the other hand,
  \begin{align*}
    \ta(-n\log 2)S_{\ind_{[0,1)}}\mu^\infty
     & = \int_{0}^{1}\ta(-n\log 2)\tu(-s).\mu^\infty \bd s
    = \int_{0}^{1}\tu(-2^{-n}s)\ta(-n\log 2).\mu^\infty \bd s                                          \\
     & = \int_{0}^{1}\tu(-2^{-n}s).\mu^\infty \bd s\xrightarrow[n\to \infty]{\text{weak-}*}\mu^\infty.
  \end{align*}
  where the last equality follows from the $\ta$-invariance of $\mu^\infty$.
  Since $\mu^\infty$ is a weak-$*$ limit of measures with more and more invariance, we obtain that $\mu^\infty$ is invariant to $\tu\left(\bigcup_{n=0}^\infty 2^{-n} \ZZ\right)$, and hence to the entire $\tu$-action.
\end{proof}

\section{Proof of Theorem \ref{thm: main}}
\label{sec: proof of rigidity theorems}
In this section, we will prove Theorem \ref{thm: main}, conditioned on Lemma \ref{lem: non-degenerate limits}.
The proof is an adaptation of \cite{bader2021arithmeticity}, and is composed of four components, which we will enumerate from last to first. The last component is \cite[Thm.~1.6]{bader2021arithmeticity}, whose outcome contradicts the nonarithmeticity assumption we assumed. The third is a reduction of our problem to \cite[Thm.~1.6]{bader2021arithmeticity}. We will do this similarly to \cite[\S3]{bader2021arithmeticity}. However, two new ergodic components will be needed for this part, namely, Theorem \ref{thm: dimension implies invariance}, and Lemma \ref{lem: non-degenerate limits}.
This section will focus on the reduction to \cite[Thm.~1.6]{bader2021arithmeticity}. If $\Gamma$ is of infinite volume, the arguments in this Subsections \ref{ssec: lemmas for rigidity} and \ref{ssec: measure on G mod Gamma} will be sufficient to prove Theorem \ref{thm: main}, and \cite[Thm.~1.6]{bader2021arithmeticity} will not be required.

Note that if $\Gamma$ is an arithmetic lattice then Theorem \ref{thm: main} follows from Borel and Harish-Chandra's Theorem \cite{BorelHarishChandra} with $\varepsilon_\Gamma = 1$. For the rest of the section we assume that $\Gamma$ is not an arithmetic lattice.

\subsection{Entropy intepretation of the critical exponent}
\label{ssec: lemmas for rigidity}
In this section, we deduce the following claim from known results.
\begin{prop}\label{prop: critical exponent via entropy}
  Let $\Lambda_0 < \SL_2(\RR)$ be a Zariski-dense subgroup. 
  Then 
  \begin{align}\label{eq: entrop big}
    \sup_\mu h_\mu(\ta(1)) \ge \delta(\Lambda_0),
  \end{align}
  where the supremum goes over all $\ta$-invariant and ergodic probability measures $\mu$ on $\SL_2(\RR) / \Lambda_0$.
\end{prop}
\begin{proof}
  We will first use the following theorem and replace $\Lambda_0$ by a finitely generated group. 
  \begin{claim}[{\cite[Cor. 6]{sullivan1979density}}]\label{claim: crit approx}
    For every discrete subgroup $\Lambda<\SL_2(\RR)$,
    \[\delta(\Lambda) = \sup_{\substack{\Lambda'<\Lambda\\\text{finitely generated}}}\delta(\Lambda').\] 
  \end{claim}
  Let $\varepsilon>0$. By Claim \ref{claim: crit approx} there is $\Lambda'_0<\Lambda_0$ finitely generated such that $\delta(\Lambda'_0) > \delta(\Lambda_0) - \varepsilon$. 
  We may assume that we added enough generators so that $\Lambda'_0$ is Zariski dense in $\SL_2(\RR)$. 
  By \cite[Thm. 10.1.2]{beardon2012geometry}, $\Lambda'_0$ is geometrically finite. 
  
  To construct measures we use the following theorem.
  \begin{theorem}[Properties of Bowen-Margulis-Sullivan measures]\label{thm: bowen Margulis}
    Let $\Lambda < \SL_2(\RR)$ be a Zariski-dense geometrically finite discrete group.
    Then there is an $\ta$-invariant probability measure $\mu_\Lambda$ on $\SL_2(\RR)/\Lambda$ with
    \[h_\mu(\ta(t)) = |t|\delta(\Lambda).\]
  \end{theorem}
  See \cite[Thms.~1, 3(i,ii)]{sullivan1984entropy}. 
  Applying Theorem \ref{thm: bowen Margulis} to $\Lambda_0'$ 
  we get an $\ta$-invariant and ergodic measure $\mu_{\Lambda_0'}$ on $\SL_2(\RR)/\Lambda_0'$ with entropy $h_{\mu_{\Lambda_0'}}(\ta(1)) = \delta(\Lambda'_0) \ge \delta(\Lambda_0) - \varepsilon$. 
  The projection $\pi_* \mu_{\Lambda_0'}$ of $\mu_{\Lambda_0'}$ to $\SL_2(\RR) / \Lambda_0$ satisfies $h_{\pi_* \mu_{\Lambda_0'}}(\ta(1)) \ge \delta(\Lambda_0) - \varepsilon$ by the following simple fact on the entropy of ergodic systems and this completes the proof.
\begin{claim}\label{claim: countable factor map preserves entropy}
  Let $(X_1, \nu_1, T), (X_2,\nu_2, T_2)$ invertible ergodic systems and $f:X_1\to X_2$ a factor map with countable fibers, such that
  $f_*\nu_1 = \nu_2$ and $f\circ T_1 = T_2\circ f$.
  Then $h_{\nu_1}(T_1) = h_{\nu_2}(T_2)$.
\end{claim}
\end{proof}
\begin{remark}
  In fact, there is an equality in Eq. \eqref{eq: entrop big}, and the measure constructed in Theorem \ref{thm: bowen Margulis} are the measures of maximal entropy. 
  After using \cite[Cor.~6.13]{raghunathan1972discrete} to reduce to the torsion free case,
  these facts are proved by \cite{otal2004principe}. 
\end{remark}



\subsection{Beginning of the proof}
\label{ssec: measure on G mod Gamma}

The contrary Theorem \ref{thm: main} is the existence of a sequence $(g_k)_{k=1}^\infty \subset G$ such that $\delta(\Gamma_{g_k})\xrightarrow{k\to \infty}1$ but $\Gamma_{g_k}$ is never a lattice.
Note that if there is an element $g_0\in G$ such that $\delta(\Gamma_{g_0}) = 1$ but $\Gamma_g$ is not a lattice, the sequence $(g_k)_{k=1}^\infty$ may be the constant sequence $g_0$.
We assume to the contrary that such a sequence exists.

By Proposition \ref{prop: critical exponent via entropy} there is a $\ta$-invariant and ergodic measure $\mu_k'$ on $\SL_2(\RR)/\Gamma_{g_k}$ with $h_{\mu_k'}(\ta(1)) > \delta(\Gamma_{g_k}) - \frac{1}{k}$. 
Then \begin{align*}
  \lim_{k\to \infty}h_{\mu_k'}(\ta(1)) = 1.
\end{align*}
Define $r_k: \SL_2(\RR)/\Gamma_{g_k}\to G/\Gamma$ as the map sending $\pi_{\Gamma_{g_k}}(x)\mapsto \pi_\Gamma(xg_k)$.
Note that $r_k$ is one-to-one. 
Then define
$\mu_k = (r_k)_*(\mu_k')$.
We wish to show that $\mu_k\xrightarrow{k\to \infty}m_{G/\Gamma}$.
Hence
\begin{align*}
  1\xleftarrow{k\to \infty}h(\mu_k') &
  \stackrel{\ref{thm: leafwise dim = entropy}}=
  \dim^\tu(\mu_k')
  \stackrel{\ref{claim: action equivariance of leafwise}}=
  \dim^\tu((r_k)_*\mu_k')
  =
  \dim^\tu(\mu_k).
\end{align*}

Restricting to a subsequence of $k$ we may assume that $\mu_k\xrightarrow{k\to \infty} \mu_\infty$, where $\mu_\infty$ is a probability measure.
\begin{claim}[Almost no periods]\label{claim: almost no periods}
  For any action $B\acts X$ and a $\ta$-invariant and ergodic probability measure $\mu$ on $X$, either $\mu$ is $\tu$-free, or $\tu(s).x = x$ for $\mu$-almost every $x$. 
\end{claim}
The claim is fairly standard, see for instance \cite[Lem.~7.12]{einsiedler2010diagonal} in the homogeneous case.
We deduce that the measures $(\mu_k)_{k=1}^\infty$ are $\tu$-free. 
By Theorem \ref{thm: dimension implies invariance}, $\mu_\infty$ is $\tu$-invariant.
By Theorem \ref{thm: dimension implies invariance} applied for the negative time $\ta(t)$ action, together with the transpose inverse action of $B$ on $G/\Gamma$, we obtain that $\mu_\infty$ is $\tu^t$-invariant.
Hence $\mu_\infty$ is $\SL_2(\RR)$-invariant. 
To show that $\mu_\infty = m_{G/\Gamma}$, we need to eliminate degenerate limits to the sequence $\mu_k$. To this end we introduce the geometric nondegeneracy lemma. We will prove it in Section \ref{sec: proof of geometric lemma}.
\begin{lemma}[Geometric nondegeneracy]\label{lem: non-degenerate limits}
  Let $\mu_k$ be a sequence of $\ta$-invariant probability measures on $G/\Gamma$ such that $\dim^\tu\mu_k\xrightarrow{k\to \infty} 1$.
  Then $\mu_k$ has no escape of mass, that is,
  \[
    \sup_{\substack{K \subset G/\Gamma\\\text{compact}}}\liminf_{k\to \infty}{\mu_k(K)} = 1. 
  \]
  In addition, there is no nontrivial convergence to a periodic $\SL_2(\RR)$-orbit of $G/\Gamma$.
  In other words, let $g\in G$ such that $\Gamma_g$ is a lattice.
  The $\SL_2(\RR)$-orbit $\SL_2(\RR).\pi_\Gamma(g)\subset G/\Gamma$ is thus a closed set.
  Suppose that $\mu_k(\SL_2(\RR).\pi_\Gamma(g)) = 0$ for every $k$.
  Then
  \begin{align}\label{eq: no escape}
    \sup_{\substack{K \subset (G/\Gamma) \setminus \SL_2(\RR).\pi_\Gamma(g)\\\text{compact}}}\liminf_{k\to \infty}{\mu_k(K)} = 1. 
  \end{align}
\end{lemma}
Let $\mu_\infty = \int_{X} \mu_\infty^x\bd \mu_\infty(x)$ be the $\SL_2(\RR)$-ergodic decomposition.
Then by Lemma \ref{lem: non-degenerate limits}, we deduce that $\mu_\infty(\SL_2(\RR).\pi_\Gamma(g))=0$ for every periodic orbit $\SL_2(\RR).\pi_\Gamma(g)$.
Thus $\mu_\infty^x(\SL_2(\RR).\pi_\Gamma(g))=0$ for $\mu_\infty$-almost every $x\in G/\Gamma$.
By Ratner's Measure Classification Theorem \cite{ratner1991raghunathan}, $\mu_\infty^x$ is algebraic, i.e., there exists a connected intermediate group $\SL_2(\RR)\le L^x\le G$ such that $\mu_\infty^x$ is the Haar measure on a periodic $L^x$-orbit. However, we showed that $L^x\neq \SL_2(\RR)$ almost surely. 
One can see that there are no nontrivial connected intermediate subgroups between $\SL_2(\RR)$ and $G$.
If $\Gamma$ is not a lattice, then the Haar measure is infinite. Then $\mu_\infty^x$ is has no valid possibility, and we get a contradiction. 
For the rest of the section, $\Gamma$ is a nonarithmetic lattice. Then $\mu_\infty^x$ is the constant measure $m_{G/\Gamma}$, which implies that $\mu_\infty=m_{G/\Gamma}$.
\subsection{Completion of the proof using rigidity - lifting the measures to a projective bundle}
\label{ssec: using rigidity}
The weak-$*$ convergence $\mu_k\xrightarrow{k\to \infty}m_{G/\Gamma}$
is what we need to initiate the reduction to \cite{bader2021arithmeticity}, with the measures $\mu_k$ supported on the dense orbits $\SL_2(\RR).\pi_\Gamma(g_k)$, in place of the infinitely many periodic orbits. The proof continues similarly to \cite{bader2021arithmeticity}, except for one point, where we need to prove some extra invariance. We will recite some details up to that point from \cite[\S2]{bader2021arithmeticity}.
\begin{definition}[Witness of the non-arithmeticity]\label{def: nonarithmeticity}
  Let $\ell$ be the field generated by traces of the adjoint representations of $\Gamma$-elements. Here $G$ is thought of as a real algebraic group, thus $\ell\subseteq \RR$.
  The field $\ell$ is a number field (\cite{selberg1960discontinuous}, \cite{calabi1961compact}, \cite{raghunathan1968cohomology}, \cite{garland1966deformation}) contained in $\RR$. The inclusion map is a real embedding $\sigma:\ell\to \RR$.
  By \cite{vinberg1971rings}, there is an $\ell$-algebraic group $\bG$ which is an $\ell$-form of the image of $G$ under the adjoint homomorphism such that $\Ad(\Gamma)$ lands in $\bG(\ell)$.
  In other words, there is an isomorphism $\iota':\Ad G\to \bG(\RR)_0$ such that when we consider the composition $\iota = \iota'\circ \Ad:G\to \bG(\RR)$ we have $\iota(\Gamma)\subseteq\bG(\ell)$.
  Since $\Gamma$ is nonarithmetic, there is a place $\nu\in S(\ell)$ such that \[\iota_\nu = (\bG(\ell)\to \bG(\ell_\nu))\circ \iota|_{\Gamma}:\Gamma \to \bG(\ell_\nu)\] does not factor continuously through $\Gamma\to G$, and the image $\iota_\nu(\Gamma)$ does not lie in a compact subgroup of $\bG(\ell_\nu)$.
\end{definition}
\begin{claim}[Construction of invariant points]
  Up to restricting to a subsequence of $k$-s, there is an irreducible algebraic $\bG(\ell_\nu)$-representation $V$ independent of $k$ and a point $P_k\in \PP(V)$, for which the following holds.
  \begin{enumerate}
    \item
          Consider the action $\Gamma\acts \PP(V)$ induced by the homomorphism $\iota_\nu$ composed with the $\bG(\ell_\nu)$ action on $\PP(V)$. Then $P_k$ is $\Gamma_{g_k}^1$ invariant, where $\Gamma_{g_k}^1 = g_k^{-1}\Gamma_{g_k} g_k\subseteq \Gamma$.
    \item The representation $V$ is in fact a representation of $\bH(\ell_\nu)$, where $\bH$ is the image of $\bG$ under the adjoint representation. As such it is a faithful representation of $\bH(\ell_\nu)$.
  \end{enumerate}
\end{claim}
The proof mimics one-to-one parts of \cite[Prop.~3.4]{bader2021arithmeticity}, however, we recite the main details.
\begin{proof}
  For each $k=1,2,\dots$, consider $\Gamma^{1}_{g_k} = g_k^{-1}\Gamma_{g_k}g_k =  \Gamma \cap g_k^{-1}\SL_2(\RR)g_k \subseteq G$.
  Consider now the $\ell$-Zariski closure of $\iota(\Gamma^{1}_{g_k})$ in $\bG$.
  This is an $\ell$-algebraic subgroup $\bL_k \subseteq \bG$. The localization at $\sigma$ is \[\bL_k(\RR) = \iota\left(\overline{g_k^{-1}\Gamma_{g_k}g_k}^{\rm zar}\right) = \iota(g_k^{-1}\SL_2(\RR)g_k).\]
  Thus $\bL_k$ is an $\bSL_2$-form over $\ell$, and in particular, is three dimensional. Since $\SL_2(\RR)$ is not a normal subgroup of $G$, we deduce that $\bL_k$ is not a normal subgroup of $\bG$.
  Let $\fg=\mathrm{Ad}(\bG)$ and $\fl_k=\mathrm{Ad}(\bL_k)$ be the algebraic adjoint representations.
  The exterior product $\bigwedge^3\fg(\ell_\nu)$ is an $\ell_\nu$-algebraic representations of $\bG(\ell_\nu)$.
  The 3-dimensional subspace $\fl_k(\ell_\nu)$ of $\fg(\ell_\nu)$ induces a vector $p_k\in \bigwedge^3\fg(\ell_\nu)$, well defined up to multiplication by scalar.
  Since $\bL_k$ is not a normal subgroup of $\bG$, it follows that $p_k$ is not $\bG(\ell_\nu)$-invariant.
  Thus, it projects nontrivially to some nontrivial irreducible $\bG(\ell_\nu)$ sub-representation $V$ of $\bigwedge^3\fg(\ell_\nu)$.
  Denote this projection by $\pi_{V}$.
  Restricting to a subsequence we may assume that $V$ is constant, that is, independent of $k$.
  The point $P_k = [\pi_{V}(p_k)]\in \PP(V)$ is $\bL_k(\ell_\nu)$ invariant. It is not $\bG(\ell_\nu)$-invariant, as this would imply that $V$ is one-dimensional, but $\bG(\ell_\nu)$ is a semisimple group and has no nontrivial one-dimensional representations.
  Thus one observes the first point of the claim.
  The second follows from the construction, except for the faithfulness part. It follows in the same way as in \cite[Prop.~3.4]{bader2021arithmeticity}.
\end{proof}
Consider the right action $G\times \PP(V)\racts\Gamma$ by $(g,P)\gamma = (g\gamma, \iota_\nu(\gamma^{-1})P)$.
Consider the $\PP(V)$-bundle $(G\times \PP(V))/\Gamma$, and the projection $\tilde \pi:G\times \PP(V)\to (G\times \PP(V))/\Gamma$.
Forgetting the $\PP(V)$ coordinate yields a projection $\rho:(G\times \PP(V))/\Gamma\to G /\Gamma$.
It has a left action by $\SL_2(\RR)$, acting only on the $G$ coordinate.
The following claim is analogous to the result of \cite[Prop.~3.4]{bader2021arithmeticity} in our setting.
\begin{claim}
  There is an $\SL_2(\RR)$-invariant measure $\tilde \mu$ on $\PP(V)\times G /\Gamma$ such that $\rho_*\tilde \mu = m_{G/\Gamma}$.
\end{claim}
\begin{proof}
  Consider the point $\tilde Q_k = (g_k, P_k)\in G\times \PP(V)$, $Q_k = \tilde \pi(\tilde Q_k) \in (G\times \PP(V))/\Gamma$, and the map
  $\tilde r_k':\SL_2(\RR)\to (G\times \PP(V))/\Gamma$ defined by $h\mapsto hQ_k$.
  This map is invariant to $\Gamma_{g_k}$ from the right, indeed,
  \begin{align*}
    \tilde r_k'(h\gamma) & = h\gamma Q_k = \tilde\pi (h\gamma g_k, P_k) = \tilde \pi (hg_k g_{k}^{-1}\gamma g_k, P_k) \\&= \tilde \pi (hg_k g_{k}^{-1}\gamma g_k, \iota_\nu(g_{k}^{-1}\gamma^{-1} g_k) P_k) = \tilde \pi (hg_k, P_k).
  \end{align*}
  The fourth equality follows from the fact that $P_k$ is $g_k^{-1}\Gamma_{g_k}g_k$ invariant, and the last follows from the definition of the quotient map $\tilde \pi$, which is the quotient by the right $\Gamma$-action.
  Hence we may define $\tilde r_k:\SL_2(\RR)/\Gamma_{g_k}\to (G\times \PP(V))/\Gamma$ as the descent of $\tilde r_k'$.
  Thus we factored $r_k:\SL_2(\RR)/\Gamma_{g_k}\to G/\Gamma$ as a composition of $\SL_2(\RR)$-maps,
  $r_k = \rho\circ \tilde r_k$.
  In a diagram,
  \begin{align*}
    \xymatrix{\SL_2(\RR)/\Gamma_{g_k}\ar[r]^{\tilde r_k}\ar[rd]^{r_k} & (G\times \PP(V))/\Gamma\ar[d]^\rho \\&G/\Gamma}
  \end{align*}

  Therefore, the measure $\mu_k = (r_k)_* \mu_k'$ on $G/\Gamma$ lifts to probability measures $\tilde \mu_k = (\tilde r_k)_* \mu_k'$ satisfyies that $\rho_*\tilde \mu_k = \mu_k$.
  Similarly to $\mu_k$, we obtain $\dim^\tu(\tilde \mu_k)\xrightarrow{k\to \infty}1$.

  Since $(\mu_k)_{k=1}^\infty$ weak-$*$ converges to the probability measure $m_{G/\Gamma}$ and the fibers of $\rho$ are compact, we deduce that there is no escape of mass in $(\tilde \mu_k)_{k}$, and hence we may restrict to a subsequence of $k$-s and assume that $\tilde \mu_k$ weak-$*$ converges to a measure probability measure $\tilde \mu$, satisfying that $\rho_*\tilde \mu = m_{G/\Gamma}$.
  Since $\mu_k$ is $\tu$-free for all $k$, we deduce that $\tilde \mu_k$ is $\tu$-free as well.
  By Theorem \ref{thm: dimension implies invariance}, the measure $\tilde \mu$ is $\tu$-invariant. By Theorem \ref{thm: dimension implies invariance} applied for the negative time $\ta(t)$ action, together with the transpose inverse action of $B$ on $G/\Gamma$, we obtain that $\tilde \mu$ is $\tu^t$-invariant, and hence $\SL_2(\RR)$-invariant.
\end{proof}

The existence of such a measure $\tilde \mu$ is required in \cite[Thm.~1.6]{bader2021arithmeticity}.
The rest of the proof is now identical to the reduction of \cite[Thms.~1.1~and~1.5.]{bader2021arithmeticity} to \cite[Thm.~1.6]{bader2021arithmeticity}.
\cite[{\S}3.4]{bader2021arithmeticity} implies the compatibility assumption of \cite[Thm.~1.6]{bader2021arithmeticity}, and hence its result holds, and shows that in fact $\iota_\nu:\Gamma\to \bG(\ell_\nu)$ extends to a continuous homomorphism $G\to \bG(\ell_\nu)$, which contradicts the assertion made in Definition \ref{def: nonarithmeticity}. \qed

\section{Proof of Lemma \ref{lem: non-degenerate limits}}
\label{sec: proof of geometric lemma}
The proof of Lemma \ref{lem: non-degenerate limits} employs the linearization method.  Linearization is a general technique, introduced by Dani and Margulis \cite{dani16limit}, and it uses representations to control the distance to homogeneous subvarieties.

In Subsection \ref{ssec: margulis function} we introduce the notion of $(\varepsilon; T_0, T_1)$-additive Margulis function.
In Subsection \ref{ssec: proof assuming} we prove Lemma \ref{lem: non-degenerate limits}, assuming a representational description of certain geometric notions, and prove them in Subsections \ref{ssec: proof of geometric criterion - cusp}, \ref{ssec: proof of geometric criterion - SL2 orbit}.
\begin{remark}
  The results in this section are related to Mohammadi and Oh \cite[Thm.~1.5]{mohammadi2023isolations}.
  In \cite{mohammadi2023isolations} the authors prove a separation result for closed $\SL_2(\RR)$-orbits in geometrically finite quotients of $G$. To do so, Mohammadi and Oh show that the Bowen-Margulis-Sullivan measure on one $\SL_2(\RR)$-orbit must be separated from the other $\SL_2(\RR)$-orbit.
  In this section, we also prove a separation result of measures and closed orbits. Although we allow our measures to be more general than Bowen-Margulis-Sullivan's measures, they are the ones in our application.

  The proofs in this section and \cite{mohammadi2023isolations} share several similarities.
  First, the representation framework is similar. Second, both approaches use a Markov operator, though different ones.
  Third, we use a Margulis function similar to \cite{mohammadi2023isolations}; however, their Margulis function satisfies the Margulis inequality everywhere, while ours satisfies it only with high probability. The reason for this difference is how each paper effectivizes the high dimension of the leafwise measures $\mu^\tu_x$ for the Bowen-Margulis-Sullivan measures $\mu$.

  Here we use Lemma \ref{lem: high leafwise entropy meaning}.
  Mohammadi and Oh \cite{mohammadi2023isolations} use a different way to effectivize the dimension, by using a uniform bound
  $\frac{\mu^\tu_x([ -r,r])^{1/\delta'}}{r\mu^\tu_x([ -1,1])^{1/\delta'}} \le \mathtt{p}$ for $\mu$-almost-all $x$ and for all $r\in (0,2]$. Here $\delta'$ is either $\delta$ or $1-2(1-\delta)$.
  For our purposes this approach cannot be applied, since it requires a uniform bound for the $\mathtt{p}$ of all our $\mu_k$, which we were not able to obtain.



\end{remark}
\subsection{\texorpdfstring{$(\varepsilon; T_0, T_1)$}{(Epsilon; T0, T1)}-additive Margulis function}\label{ssec: margulis function}
In this section, we introduce the notion on $(\varepsilon; T_0, T_1)$-additive Margulis function and prove some results on the notion. It is similar, though not identical, to the similarly named notion in \cite{eskinMozesMargulisFunctions}.
\begin{definition}[$(\varepsilon; T_0, T_1)$-additive Margulis function]\label{def: ep-additive Margulis function}
  Let $(X,\mu)$ be a measure space, $x\mapsto \nu_x$ a measurable map from $X$ to measures on $X$ such that $\mu = \int_X\nu_x\bd \mu(x)$.
  In other words, $(X, x\mapsto \nu_x)$ is a Markov chain and $\mu$ is a stationary measure.
  A measurable function $\alpha: X \to [0,\infty)$ is called \emph{$(\varepsilon; T_0, T_1)$-additive Margulis function} for some $T_1> T_0>0$ large and $\varepsilon>0$ if the following conditions hold:
  \begin{enumerate}[label=M-\alph*), ref=(M-\alph*)]
    \item \label{cond: Lipshitz}For $\mu$-almost all $x\in X$, and for $\nu_x$-almost all $y\in X$, we have $\alpha(y) \in \alpha(x) + [ -T_1, T_1]$.
    \item \label{cond: decay on averege}
          \begin{align}\label{eq: Margulis bound}
            \mu\left( \left\{ x\in X: T_1 \le \alpha(x) < T_0 + \int_X\alpha(y)\bd \nu_x(y)\right\} \right)<\varepsilon.
          \end{align}
  \end{enumerate}
\end{definition}
Having equality to $0$ in Eq. \eqref{eq: Margulis bound} is an additive version of the standard definition of a Margulis function.
\begin{lemma}\label{lem: Margulis function goes down}
  In the setting of Definition \ref{def: ep-additive Margulis function},
  \begin{align}\label{eq: Margulis func control}
    \mu(\{x\in X:\alpha(x)\ge t\}) \le \frac{1}{\log \lfloor t / T_1\rfloor - 1} + \frac{T_0+T_1}{T_0}\varepsilon,
  \end{align}
  for all $t\ge 3T_1$.
\end{lemma}
\begin{proof}
  For every interval $I \subseteq \RR$ denote
  \begin{align*}
    A_I & = \{x\in X:\alpha(x) \in I\}
    ,                                                                              \\
    B_I & = \left\{x\in A_I:\alpha(x) < T_0 + \int_X\alpha(y)\bd \nu_x(y)\right\}.
  \end{align*}

  For every $t_1 \ge T_1$, $t_2 > t_1 + 2T_1$, we use the stationarity of $\mu$ and Condition \ref{cond: Lipshitz} to obtain
  \begin{align*}
    \int_{A_{[t_1 + T_1, t_2 - T_1]}} & \alpha(y) \bd \mu(y) \le
    \int_{A_{[t_1, t_2]}}\int_{X} \alpha(y)\bd \nu_x(y)\bd \mu(x) \\&\le
    \int_{A_{[t_1, t_2]}}(\alpha(x) - T_0)\bd \mu(x) + \mu(B_{[t_1, t_2]}) \cdot(T_0 + T_1).
  \end{align*}
  Canceling common terms yields
  \begin{align}\label{eq: bound on measure1}
    T_0 \mu(A_{[t_1, t_2]})
    \le
    (t_1 + T_1) \mu(A_{[t_1, t_1 + T_1)}) + t_2\mu(A_{[t_2 - T_1, t_2]}) + \varepsilon(T_0 + T_1).
  \end{align}
  Replacing $t_2$ by $t_2+nT_1$ for some $n$ yields different bounds
  \begin{align}\label{eq: bound on measure2}
    \begin{split}
      T_0 \mu(A_{[t_1, t_2 + nT_1]})
      \le
      (t_1 + T_1) \mu(A_{[t_1, t_1 + T_1)}) + (t_2 + nT_1)\mu(A_{[t_2 + (n-1)T_1, t_2 + nT_1]}) + \varepsilon(T_0 + T_1).
    \end{split}
  \end{align}
  Now, note that
  \[\liminf_{n\ge 0} (t_2 + nT_1)\mu(A_{[t_2 + (n-1)T_1, t_2 + nT_1]}) = 0.\] Otherwise we have
  $\mu(A_{[t_2 + nT_1, t_2 + (n+1)T_1]}) > \frac{\delta}{t_2 + (n+1)T_1}$ for some $\delta>0$ and all $n$ sufficiently large, but this is a divergent series.
  Taking liminf over $n$ in Eq. \eqref{eq: bound on measure2} gives us
  \begin{align}\label{eq: bound on measure3}
    T_0 \mu(A_{[t_1, \infty)}) \le (t_1 + T_1)\mu(A_{[t_1, t_1 + T_1)}) + \varepsilon(T_0 + T_1).
  \end{align}
  Let $n\ge 3$ and $n\ge m \ge 1$. Substitute $t_1 = mT_1$ to Eq. \eqref{eq: bound on measure3} and get
  \begin{align}\label{eq: bound on measure4}
    T_0 \mu(A_{[nT_1, \infty)}) \le T_0 \mu(A_{[mT_1, \infty)}) \le (m+1)T_1\mu(A_{[mT_1, (m+1)T_1)}) + \varepsilon(T_0 + T_1).
  \end{align}
  Now, for every $n$ consider
  \[\delta_n = \min_{m=1}^n (m+1)\mu(A_{[mT_1, (m+1)T_1)}).\]
  We deduce that
  $\mu(A_{[mT_1, (m+1)T_1)}) \ge \frac{T_1\delta_n}{(m+1)T_0}$,
  and by additivity,
  \begin{align*}
    1\ge \mu(A_{[T_1, (n+1)T_1)}) \ge \frac{T_1\delta_n}{T_0} \sum_{m=1}^n \frac{1}{m+1} \ge \frac{T_1\delta_n}{T_0}(\log n - 1).
  \end{align*}
  Altogether,
  $\delta_n \le \frac{T_0}{T_1(\log n - 1)}.$
  Plugging this to Eq. \eqref{eq: bound on measure4} yields
  \begin{align*}
    T_0 \mu(A_{[nT_1, \infty)}) \le \frac{T_0}{\log n-1} + \varepsilon(T_0 + T_1).
  \end{align*}
  Hence, for all $t\ge 3T_1$,
  \[\mu(A_{[t, \infty)}) \le \frac{1}{\log \lfloor t/T_1\rfloor - 1} + \varepsilon\frac{T_0+T_1}{T_0},\]
  as desired.
\end{proof}
\begin{remark}
  The summand
  $\frac{T_0+T_1}{T_0}\varepsilon$
  in Eq. \eqref{eq: Margulis func control}
  is tight, and cannot be improved.
\end{remark}
\begin{claim}\label{claim: maxing to Margulis}
  In the setting of Definition \ref{def: ep-additive Margulis function}, suppose that $\alpha:X\to [0,\infty)$ is an $(\varepsilon; T_0, T_1)$-additive Margulis function with $T_0 > T_1/2$, and $\beta:X\to [0,\infty)$ is a function satisfying Condition \ref{cond: Lipshitz} for $T_1$ and
  \begin{align}\label{eq: estimate semi-Margulis}
    \mu\left( \left\{ x\in X: \alpha(x) + T_1 \le \beta(x) < T_0 + \int_X\beta(y)\bd \nu_x(y)\right\} \right)<\varepsilon.
  \end{align}
  Then $\gamma = \max(0, \alpha - 2T_1, \beta - 5T_1)$ is a $(2\varepsilon; 2T_0 - T_1, T_1)$-additive Margulis function.
\end{claim}
\begin{proof}
  One can easily see that $\gamma$ satisfies Condition \ref{cond: Lipshitz} for $T_1$. Let $T_0' = 2T_0-T_1$.
  We claim that
  \[X_{\gamma-\rm bad} = \left\{ x\in X: T_1 \le \gamma(x) < T_0' + \int_X\gamma(y)\bd \nu_x(y)\right\},\]
  is contained in the union of the sets $X_{\alpha-\rm bad}$ and $X_{\beta-\rm bad}$ estimated in Eqs. \eqref{eq: Margulis bound} and \eqref{eq: estimate semi-Margulis} respectively.
  Indeed, let $x\in X_{\gamma-\rm bad}$ be generic in the sense that it satisfies Condition \ref{cond: Lipshitz} for $\alpha$ and $\beta$.
  We have
  \begin{align*}
    -T_0' & < \int_X\gamma(y)\bd \nu_x(y) - \gamma(x) = \int_X(\gamma(y) - (\gamma(x) - T_1))\bd \nu_x(y) - T_1
  \end{align*}
  Note that the integrand of the right-hand side is almost surely positive.
  Now denote $\alpha' = \alpha - 2T_1, \beta' = \beta - 5T_1$ and distinguish between the following three cases:
  \begin{enumerate}[label=Case-\alph*), ref=(Case-\alph*)]
    \item \label{case: alpha big}$\gamma(x) = \alpha'(x)$ and $\beta'(x) \le \alpha'(x) - 2T_1$.
    \item \label{case: beta big}$\gamma(x) = \beta'(x)$ and $\alpha'(x)\le \beta'(x) - 2T_1$.
    \item \label{case: both big}$\alpha'(x), \beta'(x) > \gamma(x) - 2T_1$.
  \end{enumerate}
  In \ref{case: alpha big}, for $\nu_x(y)$-almost all $y$, \[\alpha'(y) \ge \alpha'(x) - T_1 \ge \beta'(x) + T_1 \ge \beta'(y). \]
  In addition, since $\gamma(x) = \alpha'(x) \ge T_1$, Condition \ref{cond: Lipshitz} for $\alpha$ show that $\alpha'(y) \ge 0$, which implies that $\gamma(y) = \alpha'(y)$. Consequently,
  \begin{align*}
    -T_0' & < \int_X(\gamma(y) - (\gamma(x) - T_1))\bd \nu_x(y) - T_1 \\&=
    \int_X(\alpha'(y) - (\alpha'(x) - T_1))\bd \nu_x(y) - T_1         \\&=
    \int_X(\alpha(y) - (\alpha(x) - T_1))\bd \nu_x(y) - T_1
  \end{align*}
  Since $\alpha(x) =\alpha'(x) + 2T_1 = \gamma(x) + 2T_1 \ge T_1$ we have $x\in X_{\alpha-\rm bad}$.

  In \ref{case: beta big}, for almost all $y$ we have $\gamma(y) = \beta'(y)$ as in the previous case.
  The inequality $-T_0' < \int_X(\beta(y) - (\beta(x) - T_1))\bd \nu_x(y) - T_1$ follows as well.
  Since \[\alpha(x) = \alpha'(x) + 2T_1 \le \gamma(x) = \beta'(x) = \beta(x) - 5T_1,\] we deduce that $x\in X_{\beta-\rm bad}$.

  In \ref{case: both big},
  \begin{align*}
    -T_0' & < \int_X(\gamma(y) - (\gamma(x) - T_1))\bd \nu_x(y) - T_1
    \\&\le
    \int_X\left((\alpha(y) - (\gamma(x) - T_1))^+ + (\beta(y) - (\gamma(x) - T_1))^+\right)\bd \nu_x(y) - T_1
    \\&\le
    \int_X\left((\alpha(y) - (\alpha(x) - T_1)) + (\beta(y) - (\beta(x) - T_1))\right)\bd \nu_x(y) - T_1,
  \end{align*}
  where $x^+ = \max(x,0)$ for all $x\in \RR$.
  Hence either
  \begin{align}\label{eq: alpha nondecay}
    \int_X(\alpha(y) - (\alpha(x) - T_1))\nu_x(y) >\frac{T_1-T_0'}{2} = T_1-T_0
  \end{align}
  or
  \begin{align}\label{eq: beta nondecay}
    \int_X(\beta(y) - (\beta(x) - T_1))\nu_x(y) >\frac{T_1-T_0'}{2} = T_1-T_0
  \end{align}
  If Eq. \eqref{eq: alpha nondecay} holds, $x\in X_{\alpha-\rm bad}$. Indeed, to show that
  $\alpha(x) \ge T_1$, we use \[\alpha(x) = \alpha'(x)+2T_1\ge \gamma(x) \ge T_1.\]
  If Eq. \eqref{eq: beta nondecay} holds, $x\in X_{\beta-\rm bad}$. Indeed, to show that
  $\beta(x) \ge T_1 + \alpha(x)$, we use
  \[\beta(x) = \beta'(x)+5T_1\ge \gamma(x) + 3T_1 \ge \alpha'(x) + 3T_1 = \alpha(x) + T_1.\]

  Consequently, $X_{\gamma-\rm bad}\subseteq X_{\alpha-\rm bad}\cup X_{\beta-\rm bad}$, and hence
  $\mu(X_{\gamma-\rm bad}) \le 2\varepsilon$.
\end{proof}

\subsection{Proof assuming representational description}
\label{ssec: proof assuming}
The proof of Lemma \ref{lem: non-degenerate limits} will go as follows: given a closed $\SL_2(\RR)$-orbit $S$ we will find a height function on its complement $(G/\Gamma) \setminus S$ which measures both how close a point is to $S$ and how deep it is in the cusp. Then we show that it is an $(\varepsilon; T_0, T_1)$-additive Margulis function with respect to the leafwise Markov chain.

We will first describe representation-theoretic ways to view the cusps of $G/\Gamma$ and with the periodic $\SL_2(\RR)$-periodic orbits in it.
Note that the normalizer of $\SL_2(\RR)$ in $G$ is an index $2$ extension $\SL_2(\RR)\lhd \NSL2\RR < G$, and more explicitly,
\[\NSL2\RR = \SL_2(\RR)\sqcup \{ih:h\in M_{2\times 2}(\RR)\text{ with }\det h = -1\}.\]
Any periodic orbit $S = \SL_2(\RR).\pi_\Gamma(g_S)$ is contained in the periodic $\NSL2\RR$-orbit $\bar S = \NSL2\RR\pi_\Gamma(g_S)$.
Let ${\cusps_\Gamma}$ denote the set of cusps of $G/\Gamma$. 
\newcommand{\core}{{\rm core}\,\Gamma}
\newcommand{\coreG}{\widetilde{\core}}
\newcommand{\coreGn}{\widetilde{\core}^{\rm nbd}}
\begin{definition}[Convex core]
  Let $\core\subseteq \HH^3/\Gamma$ denote the convex core, let $\pi_{\SU(2)}:G\to \HH^3$ denote the standard projection, $\coreG = \pi_{\SU(2)}^{-1}(\core)$ and $\coreGn$ the closed unit neighborhood of $\coreG$. 
\end{definition}  
The following observation will help us use these notions.
\begin{obs}\label{obs: prob measures on core}
  Every $\ta$-invariant probability measure on $G/\Gamma$ is supported on $\coreG$.  
  In particular, 
  \begin{enumerate}
    \item\label{point: mu k and Smu k} for every $k$ the measure $\mu_k$ is supported on $\coreG$, and hence $S_{\ind_{[0,1)}}\mu_k$ is supported on $\coreGn$;
    \item\label{cor: S is in core} every periodic $\SL_2(\RR)$ orbit in $G/\Gamma$ is contained in $\coreG$. 
  \end{enumerate}
\end{obs}
\begin{lemma}[Description of the cusps of $G/\Gamma$ using representations]\label{lem: geometry near cusp}
  There exists a $2$ dimensional complex representation $V$ of $G$ equipped with a norm $\|-\|$, and a $\Gamma$-invariant subset $V_{\cusps_\Gamma}\subset V$ such that for every $g\in G$ all the vectors in $g.V_{\cusps_\Gamma}\cap \{v\in V:\|v\| <1\}$ has the same length.
  This implies that the function $\alpha_{\cusps_\Gamma}:G\to [0,\infty)$ defined by
  \begin{align*}
    \alpha_{\cusps_\Gamma} & = 
    \begin{cases}
      -\log \|g.v\|, & \text{if }\|g.v\| < 1 \text{ for some }v\in V_{\cusps_\Gamma}, \\
      0,             & \text{otherwise,}
    \end{cases}
  \end{align*}
  is well-defined.
  Moreover, we claim that $\alpha_{\cusps_\Gamma}$ is continuous, right $\Gamma$-invariant, and descends to a proper map $\alpha_{\cusps_\Gamma}|_{\coreGn}:\coreGn \to [0,\infty)$, that is, \[\alpha_{\cusps_\Gamma}|_{\coreGn}^{-1}([0,T]) = \alpha_{\cusps_\Gamma}^{-1}([0,T])/\Gamma \cap \coreGn\] is compact for every $T > 0$.

\end{lemma}
For the rest of the paper, we will use $\alpha_{\cusps_\Gamma}$ both as a function on $G$ and $G/\Gamma$.
\begin{remark}
  For every $v\in V_{\cusps_\Gamma}$ the set $\{\pi_\Gamma(g):\|g.v\| < 1\}$ is a cusp neighborhood, and
  $\pi_\Gamma(g)$ gets deeper in the cups the smaller $\|g.v\|$ is.
\end{remark}
\begin{lemma}[Description of $\NSL2\RR$-periodic orbits using representations]\label{lem: geometry near S}
  There exists a $4$-dimensional real representation $W$ equipped with norms $\|-\|$, such that the following happen:
  There is a vector $w_0\in W$ such that such that $\SL_2(\RR) = \stab_G(w_0)$, and $\NSL2\RR = \stab_G(\{\pm w_0\})$.
  Let $W/w_0 = W / \RR w_0 \cong w_0^\perp$ be the quotient space.
  It is an irreducible $\SL_2(\RR)$-representation \emph{(}equivalent to ${\rm Sym}^2$ of the standard representation\emph{)}.
  Let $\pi_{w_0}: W \to  W / w_0$, be the standard projection.
  Let $\bar S = \NSL2\RR\pi_\Gamma(g_{\bar S})$ be a periodic $\NSL2\RR$-orbit,
  $w_{\bar S} = g_{\bar S}^{-1}.w_0$ be a vector which is stabilized by
  $g_{\bar S}^{-1}\NSL2\RR g_{\bar S}$ up to sign and $W_{\bar S} = \Gamma.w_{\bar S}$.
  Define $\alpha_{\bar S}:G\to \RR\sqcup \{\infty\}$ by
  \begin{align}\label{eq: alpha S def}
    \alpha_{\bar S}(g) = \max_{w\in W_{\bar S}} -\log\|\pi_{w_0}(g.w)\|.
  \end{align}
  Then
  \begin{enumerate}
    \item \label{point: infty only on S}
          $\alpha_{\bar S}$ is continuous and attains $\infty$ only on $\pi_\Gamma^{-1}(\bar S)$.
          \item\label{point: S unique}
          There is $C_{\bar S}>0$ such if $\alpha_{\bar S}(g) > 2\alpha_{\cusps_\Gamma}(g) + C_{\bar S}$ for some $g\in G$ then for every $w\in W_{\bar S}$ exactly one of the following holds
          \begin{itemize}
            \item
                  $-\log\|\pi_{w_0}(g.w)\| = \alpha_{\bar S}(g)$,
            \item
                  $-\log\|\pi_{w_0}(g.w)\| < 2\alpha_{\cusps_\Gamma}(g) + C_{\bar S}$.
          \end{itemize}

  \end{enumerate}
\end{lemma}
We postpone these lemmas' proofs to Subsections \ref{ssec: proof of geometric criterion - cusp} and \ref{ssec: proof of geometric criterion - SL2 orbit}.
Now that we have the height function, we recall
that Lemma \ref{lem: high leafwise entropy meaning} gives us a the leafwise Markov chain on the space $G/\Gamma$ with stationary measure $S_{\ind_{[0,1)}}\mu_k$ and a transformation law $x\mapsto \nu^{(k)}_{x}$ given by
\begin{align}\label{eq: Markov law 2}
  \nu^{(k)}_x
  = \begin{cases}
      \ta(\log 2){x}       & \text{with probability } p^{(k)}(x),   \\
      \tu(1)\ta(\log 2){y} & \text{with probability } 1-p^{(k)}(x).
    \end{cases}
\end{align}
where $p^{(k)}:G/\Gamma\to [0,1]$ satisfies
\[\int_{{G/\Gamma}} H(p^{(k)}(x), 1-p^{(k)}(x))\bd S_{\ind_{[0,1)}}\mu_k(x) = \dim^u(\mu_k)\log 2.\]

Fix a positive integer $\ell$ to be specified later.
\begin{definition}[Iteration of the leafwise Markov chain $x\mapsto \nu_x^{(k)}$]
  The $\ell$ iteration of the Markov chain $x\mapsto \nu_x^{(k)}$ is defined by 
  \[x_0\mapsto \nu^{(k,\ell)}_{x_0} = \int_{{G/\Gamma}}\int_{G/\Gamma}\cdots \int_{G/\Gamma} \delta_{x_\ell} \bd \nu^{(k)}_{x_{\ell-1}}(x_\ell)\bd \nu^{(k)}_{x_{\ell-2}}(x_{\ell-1})\dots\bd \nu^{(k)}_{x_{2}}(x_1)\bd \nu^{(k)}_{x_{1}}(x_0).\]
  In other words, given $x_0$ we sample $x_1$ via $\nu^{(k)}_{x_0}$, then sample $x_2$ via $\nu^{(k)}_{x_1}$, and so on, until we sample $x_\ell$ via $\nu^{(k)}_{x_{\ell-1}}$, and $\nu^{(k,\ell)}_{x_0} = {\rm Law}(x_\ell|x_0)$. 
  Explicitly, $x_i = u(b_i)\ta(\log 2)x_{i-1}$ for all $i=1,\dots,\ell$ where 
  \[b_i = \begin{cases}
    0       & \text{with probability } p^{(k)}(x_{i-1}),   \\
    1 & \text{with probability } 1-p^{(k)}(x_{i-1}),
  \end{cases}\]
  chosen independently of $b_1,\dots,b_{i-1}$. 
  Altogether, 
  \[x_\ell = u(b_\ell)\ta(\log 2)u(b_{\ell-1})\ta(\log 2)\cdots u(b_1)\ta(\log 2)x_0 = 
  u\left(\sum_{i=1}^\ell 2^{\ell-i} b_i\right)
  \ta(n\log 2)x_0.
  \]
  Let $b = \sum_{i=1}^\ell 2^{\ell-i} b_i$ and denote $p^{(k,\ell)}_j(x_0) := \PP(b=j|x)$ for every $j=0,1,\dots,2^\ell - 1$ so that 
  \[
    x_\ell = 
    \begin{cases}
      \tu(j)\ta(\ell\log 2) x_0 \text{ with probability $p^{(k,\ell)}_i(x_0)$ for each }i=0,\dots,2^{\ell-1}.
    \end{cases}
  \]
\end{definition}
We have seen in Eq. \eqref{eq: accumulated entropy} that
\begin{align}\label{eq: accumulated entropy2}
  \begin{split}
    \int_{G/\Gamma} &H\left(p^{(k,\ell)}_0(x_0), p^{(k,\ell)}_1(x_0), \dots,p^{(k,\ell)}_{2^\ell-1}(x_0)\right)\bd (S_{\ind_{[0,1)}} \mu_k)(x_0)\\& = H(x_\ell|x_0)= \ell\dim^\tu\mu_k\log 2.
  \end{split}
\end{align}
Since the map $q_1,\dots,q_n \mapsto H(q_1,\dots,q_n)$ obtain its maximal value only at $H(1/n,\dots,1/n) = \log n$, we obtain the following observation.
\begin{obs}\label{obs: entropy maximaized}
  For every $\delta>0$ and $n>0$ there is $\varepsilon>0$ so that the following holds. Suppose that 
  $\int_{Z} H(p_1(z),p_2(z),\dots,p_n(z))\bd \nu \ge (1-\varepsilon) \log n$, where $(Z,\nu)$ is a probability space and $p_1,\dots,p_n:Z\to [0,1]$ has $p_1+\dots+p_n \equiv 1$. 
  Then 
  \begin{align}\label{eq: usually fixed}
    \nu(\{z\in Z: |p_i(z) - \frac{1}{n}| < \delta,~ \forall i=1,\dots, n\}) > 1-\delta.
  \end{align}
\end{obs}
Let $\delta>0$ to be determined later. 
By Observation \ref{obs: entropy maximaized} and Eq. \eqref{eq: accumulated entropy2}, for all $k$ large enough as a function of $\ell$ and $\delta$ we have
$S_{\ind_{[0,1)}} \mu_k(X_{{\rm good}}^{(k,\ell, \delta)}) > 1-\delta$, where
\[X_{{\rm good}}^{(k,\ell, \delta)} = \{y\in G/\Gamma: |p^{(k,\ell)}_i(y)-2^{-\ell}| < \delta \text{ for all }i=0,\dots,2^\ell-1\}.\]
We will now recall the following property of $\SL_2(\RR)$-representations.

\begin{claim}\label{claim: decay in rep}
  For every nontrivial irreducible real or complex representation $W$ of $\SL_2(\RR)$ with highest weight $n$ equipped with a norm $\|-\|$, there is $C_W > 0$ such that for every $m\ge 0$ and for every $w\in W\setminus \{0\}$,
  \begin{align*}
    \frac{1}{2^{m}}\sum_{i=0}^{2^m-1}\log \|\tu(i)\ta(m\log 2).w\| - \log \|w\| \ge \frac{n m\log 2}{2} - C_W.
  \end{align*}
\end{claim}
\begin{proof}
  Note that there is $C_0>0$ such that for all $s\in [ -1, 1], w\in W$ we have $\left|\log \|\tu(s).w\| - \log \|w\|\right| \le C_0$.
  Let $\chi_W$ denote the maximal weight character on $W$. This is a character satisfying $\chi_W(\ta(t).w) = e^{nt/2}\chi_W(w)$. Then
  \begin{align}\label{eq: estimate on unipotent sum}
    \begin{split}
      \frac{1}{2^{m}} &\sum_{i=0}^{2^m-1}\log \|\tu(i)\ta(m\log 2).w\|
      \ge
      \frac{1}{2^{m}}\int_{0}^{2^m}\log \|\tu(s)\ta(m\log 2).w\| \bd s - C_0
      \\&
      =
      \int_{0}^{1}\log \|\ta(m\log 2)\tu(s).w\| \bd s - C_0
      \ge
      \int_{0}^{1}\log |\chi_W(\ta(m\log 2)\tu(s).w)| \bd s - C_0
      \\&
      =
      \frac{mn}{2}\log 2 + \int_{0}^{1}\log |\chi_W(\tu(s).w)| \bd s - C_0
    \end{split}
  \end{align}

  Now consider the function \[f:W\setminus \{0\} \to \RR\cup \{-\infty\}, \qquad f(w) =  \int_{0}^{1}\log |\chi_W(\tu(s).w)| \bd s.\]
  It satisfies $\forall \alpha\in \RR^\times, f(\alpha w) = \log |\alpha| + f(w)$, hence is determined by its values on the unit sphere.
  One can see that it is continuous. We wish to show that it attains real values, (in contrast to $-\infty$).
  For that, we need to have that for every $w\neq 0$, the polynomial $s\mapsto \chi_W(\tu(s).w)$ does not vanish.
  This is a standard result on $\SL_2$-representations, which follows from their classification as homogeneous polynomials of degree $n$.
  Hence, $f$ has a lower bound on the unit sphere, that is, for some $C_1\in \RR$, for all $w\in W$ with $\|w\|=1$ we have $f(w) \ge -C_1$.
  Hence \[\forall w\in W\setminus 0, \qquad f(w) \ge -C_1 + \log \|w\|.\]
  Thus we can bound the right-hand side of \eqref{eq: estimate on unipotent sum} by
  \[      \frac{mn}{2}\log 2 + \int_{0}^{1}\log |\chi_W(\tu(s).w)| \bd s - C_0
    \ge \frac{mn}{2}\log 2 - C_1 + \log \|w\|-C_0.\]
  The desired inequality follows for $C_W = C_0 + C_1$.
\end{proof}
We now have a representation-theoretic tool to construct our $(\varepsilon; T_0, T_1)$-additive Margulis functions.

\begin{claim}\label{claim: alphas are Lipschitz}
  Let $k,\ell \ge 1$ and consider the Markov chain $(G/\Gamma, S_{\ind_{[0,1)}} \mu_k, \nu_y^{(k,\ell)})$. The functions $2\alpha_{\cusps_\Gamma}$ and $\alpha_{\bar S}$ satisfy Condition \ref{cond: Lipshitz} with $T_1 = \ell\log 2 + C_0$ respectively for some $C_0>0$.
\end{claim}
\begin{proof}
  There is $C_0>0$ such that
  \begin{itemize}
    \item for all $v\in V\setminus \{0\}$ and $s\in [ -1,1]$ we have $\left|\log \frac{\|\tu(s).v\|}{\|v\|}\right| < C_0/2$,
    \item for all $w\in W\setminus \{0\}$ and $s\in [ -1,1]$ we have $\left|\log \frac{\|\tu(s).w\|}{\|w\|}\right| < C_0$.
  \end{itemize}
  Note that for all $t\in \RR$,
  \begin{align*}
    \left|\log \frac{\|\ta(t).v\|}{\|v\|} \right| \le t/2\qquad & \text{for all}\qquad v\in V\setminus \{0\}, \\
    \left|\log \frac{\|\ta(t).w\|}{\|w\|} \right| \le t\qquad   & \text{for all}\qquad w\in W\setminus \{0\}.
  \end{align*}
  The desired follows from the definition of the functions and the Markov chain.
\end{proof}
\begin{claim}\label{claim: alpha cusp is Margulis}
  In the setting of Claim \ref{claim: alphas are Lipschitz}, there exists $\delta_0$ such that for all $\delta\in (0, \delta_0)$ the function $2\alpha_{\cusps_\Gamma}$ is $(\delta; T_0, T_1)$-additive Margulis function, for 
  \[T_1 = \ell\log 2 + C_0 + 1, T_0 = \ell\log 2 - 2C_V - 2^{\ell}\delta T_1,\] 
  provided that $T_0>0$ and $S_{\ind_{[0,1)}} \mu_k(X_{{\rm good}}^{(k,\ell, \delta)}) > 1-\delta$.
  Here $C_0$ is as in Claim \ref{claim: alphas are Lipschitz} and $C_V$ as in Claim \ref{claim: decay in rep}.
\end{claim}
\begin{proof}
  Let $T_1 = \ell\log 2 + C_0 + 1$.
  Let $x = \pi_\Gamma(g)\in X_{{\rm good}}^{(k,\ell, \delta)}$ with $2\alpha_{\cusps_\Gamma}(x) \ge T_1$.
  In particular, $\alpha_{\cusps_\Gamma}(x) = -\log \|g.v\|$ for some $v\in V_{\cusps_\Gamma}$.
  As in the proof of Claim \ref{claim: alphas are Lipschitz}, we deduce that
  for every
  $y=\tu(i)\ta(\ell\log 2)x \in \supp(\nu_x^{(k,\ell)})$ for $i=0,1,\dots,2^{\ell-1}$,
  \[ -2\log \|\tu(i)\ta(\ell\log 2)g.v\| > -2\log \|g.v\| - T_1 = 2\alpha_{\cusps_\Gamma}(x) - T_1 \ge 0.\]
  Thus, by its definition, $\alpha_{\cusps_\Gamma}(y) = -\log \|\tu(i)\ta(\ell\log 2)g.v\|$.
  We estimate
  \begin{align*}
    \int_{G/\Gamma} 2 & \alpha_{\cusps_\Gamma}(y)\bd\nu_x^{(k,\ell)}(y)
    \\&=
    2\alpha_{\cusps_\Gamma}(x) + 2\sum_{i=0}^{2^\ell - 1}p_i^{(k,\ell)}(x)(\alpha_{\cusps_\Gamma}(\tu(i)\ta(\ell\log 2){y}) - \alpha_{\cusps_\Gamma}(x))
    \\&\hspace{-15pt}\stackrel{x\in X_{{\rm good}}^{(k,\ell, \delta)}}{\le}
    2\alpha_{\cusps_\Gamma}(x) + 2\sum_{i=0}^{2^\ell - 1} 2^{-\ell}(\alpha_{\cusps_\Gamma}(\tu(i)\ta(\ell\log 2){y}) - \alpha_{\cusps_\Gamma}(x)) + 2^{\ell}\delta T_1
    \\&=
    2\alpha_{\cusps_\Gamma}(x) - 2\sum_{i=0}^{2^\ell - 1} 2^{-\ell}\log \frac{\|\tu(i)\ta(\ell\log 2)g.v\|}{\|g.v\|} + 2^{\ell}\delta T_1
    \\&\hspace{-1pt}\stackrel{\ref{claim: decay in rep}}{\le }
    2\alpha_{\cusps_\Gamma}(x) - \ell\log 2 + 2C_V + 2^{\ell}\delta T_1.
  \end{align*}
  Consequently, $2\alpha_{\cusps_\Gamma}$ is a $(\delta; T_0, T_1)$-additive Margulis function, with $T_0 = \ell\log 2 - 2C_V - 2^{\ell}\delta T_1$.
\end{proof}
\begin{claim}\label{claim: alpha total is Margulis}
  In the setting of Claim \ref{claim: alphas are Lipschitz},
  there exists $\ell\ge 1$ large and $\delta_0 > 0$ small such that for all $\delta\in (0,\delta_0)$ the function $\alpha_{{\cusps_\Gamma}, \bar S}=\max(0, 2\alpha_{\cusps_\Gamma}-2T_1, \alpha_{\bar S}-6T_1 - C_{\bar S})$ is a $(2\delta; T_0', T_1)$-additive Margulis function, for some $T_0' < T_1$, provided that $k\ge k_\delta$ for some $k_\delta$ depending on $\delta$.
\end{claim}
\begin{proof}
  Let $x = \pi_\Gamma(g)\in X_{{\rm good}}^{(k,\ell, \delta)}$ with $\alpha_{\bar S}(x) \ge 2\alpha_{\cusps_\Gamma}(x) + 2T_1 + C_{\bar S}$.
  Then $\alpha_{\bar S}(x) = -\log \|\pi_{w_0}(g.w)\|$ for some $w \in W_{\bar S}$.
  As in the proof of Claim \ref{claim: alphas are Lipschitz}, we deduce that
  for every $y=\tu(i)\ta(\ell\log 2)x$ for $i=0,1,\dots,2^{\ell-1}$,
  \begin{align*}
    -\log \|\pi_{w_0}(\tu(i)\ta(\ell\log 2)g.w)\| & > -\log \|\pi_{w_0}(g.w)\| - T_1 = \alpha_{\bar S}(x) - T_1
    \\&\ge 2\alpha_{\cusps_\Gamma}(x) + T_1 + C_{\bar S} > 2\alpha_{\cusps_\Gamma}(y) + C_{\bar S}.
  \end{align*}
  Hence, Lemma \ref{lem: geometry near S} point \ref{point: S unique} implies that
  $\alpha_{\bar S}(y) = -\log \|\pi_{w_0}(\tu(i)\ta(\ell\log 2)g.w)\|$.
  As in the proof of Claim \ref{claim: alpha total is Margulis}, we deduce that
  \[\int_{G/\Gamma} \alpha_{\bar S}(y)\bd\nu_x^{(k,\ell)}(y) \le \alpha_{\bar S}(x) - \ell\log 2 + C_W + 2^{\ell}\delta T_1.\]
  Let $T_0'' = \min(\ell\log 2 - C_W - 2^{\ell}\delta T_1, \ell\log 2 - 2C_V - 2^{\ell}\delta T_1)$.
  Let $\ell$ be sufficiently large so that
  \[\min(\ell\log 2 - C_W, \ell\log 2 - 2C_V) > \frac{T_1}{2} = \frac{\ell\log 2 + C_0 + 1}{2}.\]
  For $\delta>0$ sufficiently small, $T_0'' > T_1/2$.
  Thus, applying Claim \ref{claim: maxing to Margulis} for $\alpha = \alpha_{\cusps_\Gamma}$, $\beta = \max(\alpha_{\bar S}(x) - T_1 - C_{\bar S}, 0)$, $T_0''$ and $T_1$ we deduce that $\max(0, 2\alpha_{\cusps_\Gamma}-2T_1, \alpha_{\bar S}-6T_1 - C_{\bar S})$ is a $(2\delta; 2T_0'' - T_1, T_1)$-additive Margulis function, as desired.
\end{proof}
Claim \ref{claim: alpha total is Margulis} and Lemma \ref{lem: Margulis function goes down} imply that \begin{align}\label{eq: measure to 1}
\lim_{k\to \infty}S_{\ind_{[0,1)}} \mu_k(\alpha_{{\cusps_\Gamma}, \bar S}^{-1}([0,t])) \ge 1 - \frac{1}{\log \lfloor t/T_1\rfloor - 1}.
\end{align}
By Observation \ref{obs: prob measures on core} Point \ref{point: mu k and Smu k} the measure $S_{\ind_{[0,1)}}\mu_k$ is supported on $\coreGn$. 
Since $\alpha_{{\cusps_\Gamma}}$ is proper on $\coreGn$, 
attains values in $[0,\infty)$, and $\alpha_{{\cusps_\Gamma}} \le \alpha_{{\cusps_\Gamma}, \bar S} + O(1)$ we deduce from Eq. \eqref{eq: measure to 1} that $S_{\ind_{[0,1)}} \mu_k$ has no escape of mass in $\coreGn$. 

Since $\alpha_{\bar S}^{-1}(\infty) = \bar S$ and $\alpha_{\bar S} \le \alpha_{{\cusps_\Gamma}, \bar S} + O(1)$, we deduce from Eq. \eqref{eq: measure to 1} that $S_{\ind_{[0,1)}} \mu_k$ has no escape of mass in $\coreGn \setminus \bar S$. 
We deduce Eq. \eqref{eq: no escape}. \qed
%
%

\subsection{Proof of Lemma \ref{lem: geometry near cusp}}
\label{ssec: proof of geometric criterion - cusp}
We will prove a more general version of Lemma \ref{lem: geometry near cusp}, which works also for $\SL_2(\RR)$, and later use it to understand periodic $\SL_2(\RR)$-orbits.
\begin{definition}[General setting]\label{def: field setting}
  Let $F = \RR$ or $\CC$ and $H = \SL_2(F)$.
  $U_F = \left\{\tu_F(s):s\in F\right\}$ where $\tu_F(s) = \begin{pmatrix}
      1 & s \\0&1
    \end{pmatrix}$.
  Let $\ta(t) = \diag(e^{t/2}, e^{-t/2})$ for $t\in \RR$, and $K_H$ be either $\SO(2)$ or $\SU(2)$ the maximal compact subgroup in $H$.
  Let $\Gamma < H$ be a geometrically finite discrete subgroup and $\coreG$ the inverse image in $H/\Gamma$ of the convex core of $\HH^i/\Gamma$, where $i = \begin{cases}
    2,&\text{if }F = \RR,\\
    3,&\text{if }F = \CC.
  \end{cases}$.
\end{definition}
We will use this setting for the rest of the subsection.
\begin{claim}[$QR$ Decomposition]\label{claim: QR}
  Any element $g\in H$ can be represented uniquely and continuously as $g = k\ta(t)\tu_F(s)$ for $t\in \RR$,$ k\in K_H$, and $s\in F$ where $F$, $H$, $K_H$, and $\tu_F$ are as in Claim \ref{def: field setting}.
\end{claim}
\begin{definition}
  Let $B_F^* = \left\{\begin{pmatrix}
      a & b \\0&a^{-1}
    \end{pmatrix}\in\SL_2(F): |a| = 1 \right\}$.
  Let $M < B_F^*$ be an infinite discrete subgroup.

  Let $H_{\ge \tau} = \{k\ta(-t)\tu_F(s): k\in K_H, t \ge \tau, s\in F\}$ for every $\tau\in \RR\cup \{-\infty\}$. Then $H_{\ge \tau}$ is preserved by the right $B_F^*$ action, and for every infinite discrete group $M < B_F^*$ denote $D_{M, \tau} = H_{\ge \tau} / M$. 
\end{definition}

Claim \ref{claim: Siegel decomposition} is a reformulation of the thin-thick decomposition \cite{bowditch1993geometrical}, phrased in the language of $G/\Gamma$ instead of $\HH^i/\Gamma$. 
We will use the following notation that will help us to parametrize the cusps.
\begin{definition}[Quotient product]\label{def: quotient product}
  For every two discrete subgroup $M_1, M_2<H$ and an element $h_0\in H$ such that $h_0^{-1}M_1h_0\subseteq M_2$, the map $h\mapsto hh_0$ descends to a map
  \[x\mapsto x\bullet h_0 :H/M_1\to H/M_2.\]
  Sometimes we will use this notation to denote the restriction of $-\bullet h_0$ into subsets of $H/M_1$.
  To avoid confusion, whenever we use this notation, we will specify the source and target of
  $-\bullet h_0$.
\end{definition}


\begin{claim}[Thin-thick decomposition]\label{claim: Siegel decomposition}
  Let $H=\SL_2(F)$, with $F=\RR,\CC$.
  Let $\Gamma$ be a geometrically finite discrete subgroup in $H$.
  Then there is a finite set $\cusps_\Gamma$ parameterizing the cusps of $H/\Gamma$, that satisfies the following properties:
  \begin{enumerate}[label=\emph{S-\arabic*)}, ref=(S-\arabic*)]
    \item For every $c\in \cusps_\Gamma$ there is a chosen element $g_c\in H$.
    \item For every $c\in \cusps_\Gamma$ the group
          $M_c = B_F^*\cap g_c\Gamma g_c^{-1}$ is an infinite discrete subgroup in $B_F^*$ that satisfies $g_c^{-1}M_cg_c<\Gamma$.
    \item \label{point: Siegel bijection}
          Let $D_{c,\tau} = D_{M_c, \tau}$ for every $\tau \in \RR \cup \{-\infty\}$.
          The map
          \begin{align}\label{eq: - bullet c}
            -\bullet g_c: D_{c,-\infty} = H/M_c \to H/\Gamma,
          \end{align}
          restricts to a bijection on the image
          $(-\bullet g_c)|_{D_{c, 0}}: D_{c, 0} \xrightarrow{\sim} D_{c, 0}\bullet g_c,$ and the map $(-\bullet g_c)|_{D_{c, 0}}:D_{c, 0}\to H/\Gamma$ is proper.
    \item \label{point: Siegel disjoint} The images $D_{c, 0}\bullet g_c$ for $c\in \cusps_\Gamma$ are disjoint, and the complement\\
          $\coreGn \setminus \bigcup_{c\in \cusps_\Gamma}D_{c, 0}\bullet g_c$ is precompact.
  \end{enumerate}
\end{claim}
\begin{definition}
  Consider the representation $V_F = F^2$. Denote its standard basis by $e_1, e_2$.
  It has the standard Euclidean norm $\|-\|$.
  Note that $H_{\ge\tau} = \{h\in H:\|he_1\| < e^{-\tau/2}\}$.
  Let $\alpha_{\rm Siegel}:H\to \RR$ be $\alpha_{\rm Siegel}(h) = -\log \|he_1\|$ satisfying that whenever $h = k\ta(-t)\tu_F(s)$ is the $QR$-Decomposition of $h$, then $\alpha_{\rm Siegel}(h) = t/2$. This function is $B_F^*$-invariant from the right, hence descends to $\alpha_{\rm Siegel}:H/M\to \RR$ for every infinite discrete $M<B_F^*$.
\end{definition}
\begin{definition}[First definition of $\alpha_{\cusps_\Gamma}:G\to [0, \infty)$]
    Define $\alpha_{\cusps_\Gamma}:G/\Gamma \to [0, \infty)$ by
    \[
      \alpha_{\cusps_\Gamma}(x) = \begin{cases}
        \alpha_{\rm Siegel}(z), & x = z\bullet g_c\text{ for some }c\in \cusps_\Gamma, z\in D_{c,0}, \\
        0,                      & \text{otherwise.}
      \end{cases},
    \]
    Here $-\bullet g_c$ is defined as in Eq. \eqref{eq: - bullet c}.
  We see that it is proper and continuous.
\end{definition}
\begin{definition}
  For any $c\in \cusps_\Gamma$ consider the vector $v_c = g_c^{-1}.e_1\in V_F$.
  Let $V_{\cusps_\Gamma} = \bigcup_{c\in \cusps_\Gamma} \Gamma .v_c \subseteq V_F$.
\end{definition}
\begin{corollary}[Reformulation of Lemma \ref{lem: geometry near cusp} in the general setting]\label{cor: Siegel via reps}
  The function $\alpha_{\cusps_\Gamma}:H \to [0,\infty)$ satisfies
  \begin{align}\label{eq: alpha gamma def}
    \alpha_{\cusps_\Gamma}(h) =
    \begin{cases}
      - \log \|h.v\|, & \text{if $\|h.v\| < 1$ for some $v\in V_{\cusps_\Gamma}$}, \\
      0,              & \text{otherwise.}
    \end{cases}
  \end{align}
\end{corollary}
\begin{proof}
  Suppose that $v = h\gamma .v_c = h\gamma g_c^{-1} .e_1 \in h.V_{\cusps_\Gamma}$, satisfies that $\|v\| < 1$ for some $h\in H, c\in \cusps_\Gamma$.
  Then $h\gamma g_c^{-1} \in H_{\ge 0}$, and hence $\pi_\Gamma(h) = \pi_{M_c}(h\gamma g_c^{-1}) \bullet g_c$ and $\alpha_{\cusps_\Gamma}(h) = \alpha_{\rm Siegel}(h\gamma g_c^{-1}) = -\log \|v\|$.
  Here $-\bullet g_c$ is defined as in \eqref{eq: - bullet c}.
  This implies Eq. \eqref{eq: alpha gamma def}.
\end{proof}
This claim proves Lemma \ref{lem: geometry near cusp} when using $F=\CC$.
We recall the following corollary, which is well known but follows easily from Corollary \ref{cor: Siegel via reps}.
\begin{definition}
  Let $B_F = \left\{\begin{pmatrix}
      a & b \\0&a^{-1}
    \end{pmatrix}: b\in F, a\in F^\times \right\}$, and note that $B_F = \stab_H Fe_1$.
  Note that the restriction
  $\alpha_{\rm Siegel}|_{B_F}:B_F\to \RR$ sends $\alpha_{\rm Siegel}\left(\begin{pmatrix}
        a & b \\0&a^{-1}
      \end{pmatrix}\right) = -\log a$ is a homomorphism.
\end{definition}
\begin{corollary}\label{cor: divergence only via cusp}
  Let $h\in H$ so that $\ta(-t) \pi_\Gamma(h) \in \coreG$ for evert $t\ge 0$. The following are equivalent:
  \begin{enumerate}
    \item The trajectory $\ta(-t)\pi_\Gamma(h)$ diverges as $t\to \infty$.
    \item $h = h_0 g_c\gamma$ for $h_0\in B_F, c\in \cusps_\Gamma, \gamma\in \Gamma$.
  \end{enumerate}
\end{corollary}
\begin{proof}
  Since $\alpha_{\cusps_\Gamma}$ is proper on $\coreG$,
  the trajectory $\ta(-t)\pi_\Gamma(h)$, which lies in $\coreG$, diverges as $t\to \infty$,
  iff $\alpha_{\cusps_\Gamma}(\ta(-t)\pi_\Gamma(h))\xrightarrow{t\to \infty}\infty$.
  Let $t_0>0$ be such that for all $t\ge t_0$ we have $\alpha_{\cusps_\Gamma}(\ta(-t)\pi_\Gamma(h)) > 0$.
  Let $v\in V_{\cusps_\Gamma}$ satisfy that $\alpha_{\cusps_\Gamma}(\ta(-t_0)\pi_\Gamma(h)) = -\log \|\ta(-t_0)h .v\|$.
  The two functions $f_1: t\mapsto \alpha_{\cusps_\Gamma}(\ta(-t)\pi_\Gamma(h))$ and $f_2:t\mapsto -\log \|\ta(-t)h .v\|$
  \begin{itemize}
    \item are continuous,
    \item coincide at $t=t_0$,
    \item satisfy that $f_1(t)>0$ for all $t\ge t_0$, and
    \item for every $t\ge t_0$, if $f_2(t)>0$, $f_1(t) = f_2(t)$.
  \end{itemize}
  We deduce that for all $t\ge t_0$ we have $f_1(t) = f_2(t)$, that is,
  \[\alpha_{\cusps_\Gamma}(\ta(-t)\pi_\Gamma(h)) = -\log \|\ta(-t)h .v\|\xrightarrow{t\to \infty}\infty.\]
  This is equivalent to $\ta(-t)h.v\xrightarrow{t\to \infty} 0$.
  Note that \[\ta(-t)h.v\xrightarrow{t\to \infty} 0 \iff h.v = h\gamma g_c^{-1}.e_1\in Fe_1 \iff h\gamma g_c^{-1} \in B_F.\]
  The desired equivalence follows.
\end{proof}
\subsection{Proof of Lemma \ref{lem: geometry near S}}
\label{ssec: proof of geometric criterion - SL2 orbit}
We now return to our original setting with $\Gamma<G$, as in Definition \ref{def: homogeneous dynamics}.
Let $S = \SL_2(\RR).\pi_\Gamma(g_S)$ be a periodic $\SL_2(\RR)$-orbit in $G/\Gamma$.
Let $\Lambda = \stab_{\SL_2(\RR)} \pi_\Gamma(g_S) = g_S \Gamma g_S^{-1} \cap \SL_2(\RR)$, and $\pi_\Lambda:\SL_2(\RR)\to \SL_2(\RR)/\Lambda$ denote the standard projection.
Since $S$ is a periodic $\SL_2(\RR)$-orbit, it follows that  $\Lambda$ is a lattice in $\SL_2(\RR)$.
Let $\bar S = \NSL2\RR\pi_\Gamma(g_S)$ be a periodic $\NSL2\RR$-orbit.
\begin{remark}
  The distinction between $S$ and $\bar S$ bears no mathematical difficulties.
  If $\stab_{{\NSL2\RR}} \pi_\Gamma(g_S) \subseteq \SL_2(\RR)$,
  then $\bar S = S\sqcup g_0S$, where
  $g_0 = \begin{pmatrix}
      0 & i \\-i&0
    \end{pmatrix}\in \NSL2\RR \setminus \SL_2(\RR)$.
  Otherwise $\bar S = S = g_0S$.
\end{remark}
Let
\begin{align}\label{eq: - bullet gS}
  -\bullet g_S:\SL_2(\RR)/\Lambda\to G/\Gamma,
\end{align}
be defined as in Definition \ref{def: quotient product}.
In particular, it defines an isomorphism $-\bullet g_S:\SL_2(\RR)/\Lambda \xrightarrow{\sim} S$, which enables us to use Claim \ref{claim: Siegel decomposition} for $F=\RR$ to describe its cusps.
We will need the following claim which describes how the cusps of $S$ sit inside the cusps of $G/\Gamma$.
\begin{claim}[How the cusps of $S$ sit in the cusps of $G/\Gamma$]\label{claim: small cusps in larger cusps}
  Let $c_\Lambda\in \cusps_\Lambda$ be a cusp of $\SL_2(\RR) / \Lambda$.
  Then there are
  \begin{enumerate}
    \item a unique cusp $c_{\Gamma, c_\Lambda} = c_\Gamma \in \cusps_\Gamma$;
    \item $h_{c_\Lambda} \in B_\CC$;
    \item and $\gamma_{c_\Lambda}\in \Gamma$
  \end{enumerate}
  such that
  \begin{align}\label{eq: S cusp as Gamma cusp}
    h_{c_\Lambda}g_{c_\Gamma}\gamma_{c_\Lambda} & = g_{c_\Lambda} g_S,                                                \\
    \label{eq: lattice equality}
    M_{c_\Lambda}                           & = h_{c_\Lambda} M_{c_{\Gamma}} h_{c_\Lambda}^{-1} \cap B_\RR^*.
  \end{align}
\end{claim}
\begin{proof}
  By Corollary \ref{cor: divergence only via cusp} applied for $\SL_2(\RR)/\Lambda$, we deduce that $\ta(-t)\pi_\Lambda(g_{c_\Lambda})$ diverges in $\SL_2(\RR)/\Lambda$ as $t\to \infty$. Thus $\iota(\ta(-t)\pi_\Lambda(g_{c_\Lambda})) = \ta(-t)\pi_\Gamma(g_{c_\Lambda} g_S)$ diverges in $G/\Gamma$ as $t\to \infty$. By Observation \ref{obs: prob measures on core} Point \ref{cor: S is in core} we deduce that $\ta(-t)\pi_\Gamma(g_{c_\Lambda} g_S)\subseteq S \subseteq \coreG$ for all $t\ge 0$. 
  Corollary \ref{cor: divergence only via cusp} applied for $G/\Gamma$ now implies Eq. \eqref{eq: S cusp as Gamma cusp}.
  To show Eq. \eqref{eq: lattice equality}, use the notation $a^b = b^{-1}ab$ for conjugation and note that
  \begin{align*}
    M_{c_\Lambda} & = B_\RR^*\cap \Lambda^{g_{c_\Lambda}^{-1}}=
    B_\RR^*\cap \Gamma^{g_S^{-1}g_{c_\Lambda}^{-1}} =
    B_\RR^*\cap \Gamma^{g_{c_\Gamma}^{-1}h_{c_\Lambda}}         \\&=
    B_\RR^*\cap \left(\Gamma^{g_{c_\Gamma}^{-1}} \cap B_\CC^*\right)^{h_{c_\Lambda}} =
    B_\RR^*\cap \left(M_{c_\Gamma}\right)^{h_{c_\Lambda}}.
  \end{align*}
\end{proof}
\begin{obs}
  Using \eqref{eq: lattice equality} we may apply Definition \ref{def: quotient product} and define
  \begin{align}
    z\mapsto z \bullet h_{c_\Lambda} : \SL_2(\RR)/M_{c_{\Lambda}} \to G/M_{c_\Gamma}.
  \end{align}
  Eq. \eqref{eq: S cusp as Gamma cusp} shows that for every $z\in D_{c_\Lambda, -\infty} = \SL_2(\RR) / M_{c_{\Lambda}}$ with $g\in \SL_2(\RR)$,
  \begin{align}\label{eq: S Siegel domain through G mod Gamma}
    (z\bullet g_{c_\Lambda})\bullet g_S =
    (z \bullet h_{c_\Lambda}) \bullet g_{c_\Gamma}.
  \end{align}
  Equivalently, the following diagram is commutative,
  \begin{align*}\tag{\ref*{eq: S Siegel domain through G mod Gamma}$'$}
    \xymatrix{
    D_{c_\Lambda, -\infty}
    \ar[r]^{-\bullet h_{c_\Lambda}}
    \ar[d]^{-\bullet g_{c_\Lambda}} &
    D_{c_\Gamma, -\infty}
    \ar[d]^{-\bullet g_{c_\Gamma}}    \\
    \SL_2(\RR)/\Lambda
    \ar[r]^{-\bullet g_S}       &
    G/\Gamma
    }
  \end{align*}
\end{obs}
A useful corollary is the following:
\begin{corollary}[Only the cusps of $S$ reach deep into the cusps of $G/\Gamma$]
  \label{cor: only known cusps in big cusps}
  There is a constant $T_{\Gamma, S}\ge 0$ such that the following holds.
  For every point $x_0\in S$ with $\alpha_{\cusps_\Gamma}(x) > T_{\Gamma, S}$, is of the form
  \begin{align}\label{eq: cusps alignment}
    \begin{split}
      x_0 &= (z_0\bullet g_{c_{\Gamma, c_\Lambda}})\bullet g_S \quad\text{for}\quad z_0 \in D_{c_\Lambda, 0},\\
      \alpha_{\cusps_\Gamma}(x_0) &= \alpha_{\cusps_\Lambda}(z_0) + \alpha_{\rm Siegel}(h_{c_\Lambda}).
    \end{split}
  \end{align}
\end{corollary}
\begin{proof}
  For every cusp $c_\Lambda \in \cusps_\Lambda$, let $T_{c_\Lambda} = \max(-\alpha_{\rm Siegel}(h_{c_\Lambda}), 0) + 1$ and $c_\Gamma = c_{\Gamma, c_\Lambda}$.
  Let $z_1 \in D_{c_\Lambda, T_{c_\Lambda}}$, and let
  $x_1 = (z_1\bullet g_{c_\Gamma})\bullet g_S \stackrel{\eqref{eq: S Siegel domain through G mod Gamma}}{=} (z_1 \bullet h_{c_\Lambda})\bullet g_{c_\Gamma}$.
  Since
  \[
    \alpha_{\rm Siegel}(z_1 \bullet h_{c_\Lambda}) =
    \alpha_{\rm Siegel}(z_1) + \alpha_{\rm Siegel}(h_{c_\Lambda}) \ge 0
  \]
  we deduce that $z_1 \bullet h_{c_\Lambda} \in D_{c_\Gamma, 0}$ and hence $x_1$ satisfy Eq. \eqref{eq: cusps alignment}.
  Therefore, by Point \ref{point: Siegel disjoint}, the set
  \[S_{\rm bad} = \{x_0\in S \text{ that does not satisfy Eq. \eqref{eq: cusps alignment}}\}\]
  is precompact and we can take $T_{\Gamma, S} = \sup_{S_{\rm bad}}\alpha_{\cusps_\Gamma}$.


\end{proof}
\begin{obs}\label{obs: no local maxima for alpha high in the cusp on S}
  Observing the definition of $\alpha_{\rm Sigel}$, we notice that it has no local maxima on $\SL_2(\RR)$. In view of Corollary \ref{cor: only known cusps in big cusps} and Eq. \eqref{eq: cusps alignment}, $\alpha_{\cusps_\Gamma}$ has no local maxima on $S\cap \alpha^{-1}_{\cusps_\Gamma}((T_{\Gamma, S}, \infty))$.
\end{obs}
We introduce the representation $W$.
\begin{claim}\label{claim: W}
  There is a $4$-dimensional irreducible real representation $W = \RR^4$ of $G$ and a vector $w_0\in W$, such that $\SL_2(\RR) = \stab_{G}(w_0)$ and $\NSL2\RR = \stab_G(\{\pm w_0\})$. $G$ preserves the real quadratic form $Q(x_1, x_2, x_3, x_4) = x_1^2 - x_2^2 - x_3^2 - x_4^2$ on $W$ of type $(1,3)$, satisfying $Q(w_0) = -1$.
\end{claim}
\begin{proof}
  We first introduce $W$ as the space of Hermitian $2\times 2$ matrices, on which $G$ acts by $g.A = g^{-*}Ag^{-1}$.
  Here $g^*$ refers to the complex conjugate of the transposed matrix, and $g^{-*} = (g^*)^{-1}$.
  The quadratic form $\det$ is preserved by the $G$ action.
  Identify $W$ with $\RR^4$ using the basis
  \begin{align}\label{eq: W basis}
    \left(\begin{pmatrix}
              1 & 0 \\0&1
            \end{pmatrix}, \begin{pmatrix}
                             1 & 0 \\0&-1
                           \end{pmatrix}, \begin{pmatrix}
                                            0 & 1 \\1&0
                                          \end{pmatrix}, \begin{pmatrix}
                                                           0 & i \\-i&0
                                                         \end{pmatrix}\right).
  \end{align}
  The remaining follows with $w_0=\left(\begin{smallmatrix}
        0 & i \\-i&0
      \end{smallmatrix}\right)$.
\end{proof}
Recall that $w_{\bar S} = g_S^{-1}.w_0$ and $W_{\bar S} = \Gamma .w_{\bar S}$.
\begin{claim}\label{claim: discreteness of W bar S}
  The set $W_{\bar S}$ is discrete.
\end{claim}
\begin{proof}
  Denote by $\pi_{\SL_2(\RR)}:G\to G/\SL_2(\RR)$ the natural projection.
  Note that \[W_{\bar S}\subseteq R = \{w\in W:Q(w) = -1\}\stackrel{g.e_1 \mapsto \pi_{\SL_2(\RR)}(g)} \cong G/\SL_2(\RR),\] and $R$ is closed in $W$.
  Hence to verify that $\Gamma .w_{\bar S}$ is discrete in $W$, it is sufficient to verify
  that $\pi_{\SL_2(\RR)}(\Gamma g_S^{-1})$ is discrete.
  Suppose that there is a sequence of points
  \[\pi_{\SL_2(\RR)}(\gamma_ig_S^{-1}) \xrightarrow{i\to \infty} \gamma_\infty \pi_{\SL_2(\RR)}(g_0),\] for $\gamma_1,\gamma_2,\ldots\in \Gamma$ and $g_0\in G$.
  Hence there is a sequence of matrices $h_i\in \SL_2(\RR)$ such that
  \begin{align}\label{eq: convergence in G}
    \gamma_i g_S^{-1}h_i \xrightarrow{i\to \infty} g_0.
  \end{align}
  Inverting and projecting to $G/\Gamma$ we deduce that
  \[\pi_\Gamma(h_i^{-1}g_S)  \xrightarrow{i\to \infty} \pi_\Gamma(g_0^{-1}). \]
  Since $\pi_\Gamma(h_i^{-1}g_S)\in S$ and $S$ is closed, also $\pi_\Gamma(g_0^{-1})\in S$.
  Since $S$ is a closed orbit, for some $\varepsilon>0$ sufficiently small (depending on $g_0$),
  for every $p\in S$ with $d_{G/\Gamma}(\pi_\Gamma(g_0^{-1}), p)<\varepsilon$ we have $p = h.\pi_{\Gamma}(g_0^{-1})$ for some $h\in \SL_2(\RR)$ with $d_{\SL_2(\RR)}(h,I) = O(d_{G/\Gamma}(\pi_\Gamma(g_0^{-1}), p))$.
  Here $d_{\SL_2(\RR)}$ is the right invariant Riemannian metric on $\SL_2(\RR)$.
  It follows that for sufficiently large $i$, there is $h_i'\in H$ such that
  \begin{align}\label{eq: equality mod Gamma}
    \pi_\Gamma(h_i'^{-1}h_i^{-1}g_S) = \pi_\Gamma(g_0^{-1}),
  \end{align}
  and $d_{\SL_2(\RR)}(h_i',I) \xrightarrow{i\to \infty}0$.
  Eq. \eqref{eq: equality mod Gamma} is equivalent to
  \begin{align}\label{eq: equality mod Gamma2}
    g_S^{-1}h_ih_i' \in  \Gamma g_0.
  \end{align}
  However, Eq. \eqref{eq: convergence in G} and the size estimate for $h_i'$ imply that
  \begin{align}\label{eq: equality in G}
    \gamma_i g_S^{-1}h_i h_i' = g_0.
  \end{align}
  for all $i$ sufficiently large.
  Consequently, $\pi_{\SL_2(\RR)}(\gamma_i g_S^{-1}) = \pi_{\SL_2(\RR)}(g_0)$ for all $i$ sufficiently large.
  This means that every converging sequence in $\pi_{\SL_2(\RR)}(\Gamma g_S^{-1})$ fixes on its limit, which implies that this set is discrete, as desired.
\end{proof}
We can now prove the remaining of Lemma \ref{lem: geometry near S}.
To construct $\alpha_{\bar S}$ we consider the quotient map and projection $\pi_{w_0}:W=\RR^4\to \RR^3$ sending $(x_i)_{i=1}^4$ to $(x_i)_{i=1}^3$.
Thus we consider the standard norm on $\RR^3$ and
Define $\alpha_{\bar S}:G\to \RR\sqcup \{\infty\}$ by
\begin{align}\label{eq: alpha S def2}
  \alpha_{\bar S}(g) = \sup_{w\in W_{\bar S}} -\log\|\pi_{w_0}(g.w)\|.
\end{align}
\begin{claim}\label{claim: sup obtaind in alpha bar S}
  The supremum in Eq. \eqref{eq: alpha S def2} is attained and the function $\alpha_{\bar S}$ is continuous.
\end{claim}
\begin{proof}
  Let $R = \{w\in W:Q(w) = -1\}$ and $\ell: R \to \RR\cup \{\infty\}$ be defined by $g(w) = -\log \|\pi_{w_0}(w)\|$. Note that $R_{t}:= \ell^{-1}([t,\infty])$ is compact for all $t\in \RR$.
  Since $W_{\bar S}$ is discrete, $g.W_{\bar S}$ is discrete as well for all $g\in G$, hence $R_t\cap g.W_{\bar S}$ is a finite set, for every $t\in \RR$ and $g\in G$.
  Rewrite Eq. \eqref{eq: alpha S def2},
  \begin{align}\label{eq: alpha S bar restrict sup}
    \begin{split}
      \alpha_{\bar S}(g) & =
      \sup\{\ell(w): w\in g.W_{\bar S}\} =
      \sup\{\ell(w): w\in g.W_{\bar S}, \ell(w) \ge \alpha_{\bar S} - 1\} \\&=
      \sup\{\ell(w): w\in g.W_{\bar S}\cap R_{\alpha_{\bar S} - 1}\}.
    \end{split}
  \end{align}
  The rightmost supremum in Eq. \eqref{eq: alpha S bar restrict sup} ranges over a finite set and must be attained.
  Let $C\subseteq G$ be a compact subset. We will prove that $\alpha_{\bar S}$ is continuous on $C$. Then the infimum
  \[z_C = \inf_{g\in C}\alpha_{\bar S}(g) \ge \inf_{g\in C}\ell(g.w_{\bar S}),\]
  satisfies $z_C \in \RR\cup \{\infty\}$ by the compactness of $C$.
  The set \[W_{\bar S, C} = W_{\bar S} \cap C^{-1}.R_{z_C},\]
  is finite by the discreteness of $W_{\bar S}$.
  Then for all $g\in G$,
  \begin{align*}
    \alpha_{\bar S}(g) & =
    \sup\{\ell(w): w\in g.W_{\bar S}, \ell(w) \ge z_C\} \\&=
    \sup\{\ell(w): w\in g.(W_{\bar S}\cap g^{-1}R_{Z_C})\} =
    \sup\{\ell(w): w\in g.W_{\bar S,C}\}.
  \end{align*}
  However, $W_{\bar S,C}$ is finite, and hence $\alpha_{\bar S}$ is continuous on $C$.
  Since $G$ is locally compact, this implies that $\alpha_{\bar S}$ is continuous everywhere.
\end{proof}
\begin{proof}[Proof of Claim \ref{lem: geometry near S} Point \ref{point: infty only on S}]
  Let $g\in G$.
  Note that
  $\alpha_{\bar S}(g) = \infty $ if and only if there exists $\gamma\in \Gamma$ such that $\pi_{w_0}(g\gamma g_S^{-1}.w_0) = 0$.
  Since $Q(g\gamma g_S^{-1}.w_0) = Q(w_0) = -1$ and $\pm w_0$ are the only vectors $w\in \ker \pi_{w_0}$ satisfying $Q(w) = -1$ we deduce that $g\gamma g_S^{-1}.w_0 = \pm w_0$, and hence $g\gamma g_S^{-1} \in \NSL2\RR$.
  Altogether, we have the equivalence \[\alpha_{\bar S}(g) = \infty \Leftrightarrow g \in \NSL2\RR g_S\Gamma = \pi_\Gamma^{-1}(\bar S).\]
\end{proof}
It remains to show the Point \ref{point: S unique}.
Previously we have shown that if $\alpha_{\bar S}(g) = \infty$ then $g\in S$ and in particular $g\in \pi_\Gamma^{-1}(\coreG)$. 
The following claim is similar.
\begin{claim}\label{claim: close to orbit imp in coreGn}
  There is $C_{Q}>0$ independent of $\Gamma$ such that for every $g_0\in G$ if $\alpha_{\bar S}(g_0) \ge C_{Q}$ then $g_0\in \pi_\Gamma^{-1}(\coreGn)$. 
\end{claim}
\begin{proof}
  By Claim \ref{claim: sup obtaind in alpha bar S}, there is $w\in W_{\bar S}$ such that 
  $\alpha_{\bar S}(g_0) = -\log\|\pi_{w_0}(g_0.w)\|$. 
  This implies that $\|\pi_{w_0}(g_0.w)\| \le e^{-C_Q}$. 
  Since $R = \{w\in W: Q(w) = -1\}$ is a manifold that is transversal to $\ker \pi_{w_0}$, we deduce that there is a sign $\pm$ so that $\|g_0.w - \pm w_0\| = O(e^{-C_Q})$. 
  Hence there is $g\in G$ with $d_G(g,I) = O(e^{-C_Q})$ such that $gg_0.w = \pm w_0$. 
  If $C_Q$ is chosen sufficiently big we obtain an upper bound $d_G(g,I) \le 1$. 
  Hence $\alpha_{\bar S}(gg_0) = \infty$, 
  which implies that $\pi_\Gamma(gg_0)\in \bar S$. 
  By Observation \ref{obs: prob measures on core} Point \ref{cor: S is in core}, we deduce that $\pi_\Gamma(gg_0) \in \coreG$. 
  Since $d_G(g,I) \le 1$ we deduce that $\pi_\Gamma(g_0) \in \coreGn$.
\end{proof}
We will use the following claims:
\begin{claim}\label{claim: vector stability function}
  There is a form $\Psi: V \times W \to \RR$ that satisfies the following conditions
  \begin{enumerate}
    \item $\Psi$ is $G$-invariant in the sense that $\Psi(g.v, g.w) = \Psi(v,w)$ for all $v\in V$ and $w\in W$.
    \item $\Psi$ is linear in $W$ and Hermitian in $V$.
    \item It satisfies that $|\Psi(v,w_0)| = \inf_{g\in \SL_2(\RR)}\|g.v\|^2$.
    \item If the infimum is nonzero then it is attained.
  \end{enumerate}
\end{claim}
\begin{proof}
  Since $V = \CC^2$ and $W$ is the space of hermitian matrices, we can define the form $\Psi(v, w) = v^{*}wv$, where $v^*$ is the complex conjugate of $v$, thought of as a row vector.
  Writing $v = \binom{x}{y}\in \CC^2$, algebraic manipulations show that
  \[\Psi(v,w_0) = 2 {\rm Im} (x\bar y) = \det(v,\bar v)/i. \]
  Here, by $\det(v,\bar v)$ we refer to the determinant of the matrix whose columns are $v,\bar v$.
  Suppose that $\Psi(v,w_0) = 0$. Then $x \bar y\in \RR$, which is equivalent to saying that $v$ is proportional to a real vector.
  We know that $\SL_2(\RR)$ can shrink arbitrarily real vectors, which implies $\inf_{g\in \SL_2(\RR)}\|g.v\|^2 = 0$.
  Suppose now that $2 {\rm Im} (x\bar y)\neq 0$. Then the two vectors denote by $v_1 = \frac{v+\bar v}{2}$ and $v_2 = \frac{v-\bar v}{2i}$ are real.
  Hence,
  \[\Psi(v,w_0) = \det(v,\bar v)/i = -2\det(v_1, v_2).\]
  Note that $\|v\|^2 = \|v_1 + iv_2\|^2 = \|v_1\|^2 + \|v_2\|^2$.
  One can verify that
  \begin{align}\label{eq: det les sum of squeres}
    |2\det(v_1, v_2)| \le \|v_1\|^2 + \|v_2\|^2
  \end{align}
  for every pair of real vectors. Moreover, equality holds in Eq. \eqref{eq: det les sum of squeres} if and only if $v_1\perp v_2$ and $\|v_1\| = \|v_2\|$. Since whenever $v_1$ and $v_2$ are linearly independent, there is always $h\in \SL_2(\RR)$ such that $g.v_1\perp g.v_2$ and $\|g.v_1\| = \|g.v_2\|$, the desired result holds.
\end{proof}
\begin{claim}\label{claim: psi not small}
  There is $\varepsilon = \varepsilon_{\Gamma, \bar S} = e^{-2T_{\Gamma, S}}>0$ such that for all
  $v\in V_{\cusps_\Gamma}$ and $w\in W_{\bar S}$ we have that either $\Psi(v,w) = 0$ or $|\Psi(v,w)| \ge \varepsilon_{\Gamma, \bar S}$.
\end{claim}
\begin{proof}
  Let $v = \gamma_1 g_{c_\Gamma}^{-1} .e_1\in V_{\cusps_\Gamma}, w = \gamma_2 g_S^{-1} .w_0 \in W_{\bar S}$ be two vectors with $0\neq |\Psi(v,w)| < \varepsilon$.
  Let $g_0 = g_S\gamma_2^{-1}$, be a matrix satisfying that $g_0.w = w_0$. We deduce that $g_0v$ satisfy that $\Psi(v,w) = \Psi(g_0.v, w_0)$.

  By Claim \ref{claim: vector stability function}, there is $h_0\in \SL_2(\RR)$ such that
  $|\Psi(g_0.v, w_0)| = |\Psi(h_0 g_0.v, w_0)| = \|h_0g_0.v\|^2$.
  Then $\alpha_{\cusps_\Gamma}(h_0g_0) = -\log \|h_0g_0.v\| > T_{\Gamma, S}$.
  However, for every $h_1$ sufficiently small, \[\alpha_{\cusps_\Gamma}(h_1h_0g_0) = -\log \|h_1h_0g_0.v\| \le -\frac{1}{2}\log |\Psi(h_0 g_0.v, w_0)| = \alpha_{\cusps_\Gamma}(h_0g_0),\]
  which implies that $-\frac{1}{2}\log |\Psi(v,w)|$ is a local maximum of $\alpha_{\cusps_\Gamma}$ at $\pi_\Gamma(h_0g_0)$ along the $\SL_2(\RR)$-orbit $\SL_2(\RR).\pi_\Gamma(g_0) = S$.
  This contradicts Observation \ref{obs: no local maxima for alpha high in the cusp on S}, and hence the equation $0\neq |\Psi(v,w)| < \varepsilon$.


\end{proof}
\begin{claim}\label{claim: lin alg weird claim}
  Consider the space \[W^{0+ } = \{w:\Psi(e_1, w) = 0\} = \left\{\begin{pmatrix}
      0 & * \\ *&*
    \end{pmatrix}\right\} \subset W,\] which contains $w_0$ and is the space of $\ta(t)$ noncontracting elements in $W$.
  Then for all $k\in \SU(2), w\in W^{0+ }, t\ge 0$,
  \[\|\pi_{w_0}(k\ta(t).w)\| \le 4 e^t \|\pi_{w_0}(k.w)\|.\]
\end{claim}
\begin{proof}
  Here we will use the matrix description of $W$, and will distinguish between the matrix multiplication denoted without a dot, and the group action on the representation $w$, denoted $(g,w)\mapsto g.w: G\times W\to W$.
  We will prove the claim with the Hilbert-Schmidt norm given by
  \[\left\|\begin{pmatrix}
      a & b \\\bar b&d
    \end{pmatrix}\right\| = \sqrt{a^2 + 2|b|^2 + d^2}.\]
  This norm is proportional to the one given by the basis Eq. \eqref{eq: W basis}.
  Let $w = \begin{pmatrix}
      0 & a \\b&c
    \end{pmatrix}$ and define $w_1 = \begin{pmatrix}
      0 & a \\b&0
    \end{pmatrix}, w_2 = \begin{pmatrix}
      0 & 0 \\0&c
    \end{pmatrix}$.
  Then
  \begin{align*}
    \|\pi_{w_0}(k\ta(t).w)\| &
    \le \|\pi_{w_0}(k\ta(t).w_1)\| + \|\pi_{w_0}(k\ta(t).w_2)\|
    \\&\le \|\pi_{w_0}(k.w_1)\| + e^t\|\pi_{w_0}(k.w_2)\|
    \le e^t(\|\pi_{w_0}(k.w_1)\| + \|\pi_{w_0}(k.w_2)\|)
    \\&\le e^t(\|\pi_{w_0}(k.w)\| + 2\|\pi_{w_0}(k.w_2)\|)
    \le e^t(\|\pi_{w_0}(k.w)\| + 2\|k.w_2\|)
    \\&= e^t(\|\pi_{w_0}(k.w)\| + 2|c|)
    \intertext{
      To interpret $c$ in matrix terminology, note that \[c = \tr(w) = \tr(kwk^{-1}) = \tr(k^{-*}wk^{-1}) = \tr(k.w) = \tr \pi_{w_0}(k.w).\]
      Hence we may continue the estimate}
                             & = e^t(\|\pi_{w_0}(k.w)\| + 2|\tr \pi_{w_0}(k.w)|)
    \le
    e^t(\|\pi_{w_0}(k.w)\| + 2\sqrt 2 \|\pi_{w_0}(k.w)\|)
    \\&
    = (1+2\sqrt 2)e^t\|\pi_{w_0}(k.w)\|
    \le 4e^t\|\pi_{w_0}(k.w)\|
    .
  \end{align*}
\end{proof}
\begin{proof}[Proof of Lemma \ref{lem: geometry near S} Point \ref{point: S unique}]
  Let $T_0>0$ to be determined later, $\tilde X_{\rm comp, 1}\subseteq G$ be a compact set so that
  \[X_{\rm comp,1} = \pi_\Gamma(\tilde X_{\rm comp,1}) = \coreGn \setminus \bigcup_{\substack{c_{\Gamma}\in \cusps_\Gamma}}D_{c_\Gamma, T_0}\bullet g_{c_\Gamma} = \coreGn \cap \alpha_{\cusps_\Gamma}^{-1}([0, T_0/2]),\]
  where $\alpha_{\cusps_\Gamma}^{-1}$ is the preimage map of $\alpha_{\cusps_\Gamma}:G/\Gamma\to [0,\infty)$, and $-\bullet g_{c_\Gamma}$ is defined as in Claim \ref{claim: Siegel decomposition}.

  For every $g\in G$ denote by
  \[\beta(g) = \inf\left\{r>0:\substack{\text{there are } w_1, w_2 \in g.W_{\bar S}\text{ with }\{\pm w_1\}\neq \{\pm w_2\} \\\text{ such that }\|\pi_{w_0}(w_1)\|, \|\pi_{w_0}(w_2)\| < r}\right\}.\]
  Note that $\beta(g)>0$ for every $g\in G$. Indeed, otherwise there exist $w_1, w_2\in g.W_{\bar S}$ such that
  $\|\pi_{w_0}(w_1)\| = \|\pi_{w_0}(w_2)\| = 0$ and $\{\pm w_1\}\neq \{\pm w_2\}$.
  But $\ker \pi_{w_0} \cap \{w\in W:Q(w) = -1\} = \{\pm w_0\}$, which contradicts this existence of two different such vectors.
  Denote by $\delta_0 = \min_{\tilde X_{\rm comp, 1}}\beta > 0$, and set $C_{\bar S} = -\log \delta_0 + 2\log 2 + C_Q$.

  Suppose that $\alpha_{\bar S}(g) > 2\alpha_{\cusps_\Gamma}(g) + C_{\bar S}$ for some $g\in G$. By Claim \ref{claim: close to orbit imp in coreGn} we obtain that $g \in \pi_\Gamma^{-1}(\coreGn)$. 
  By Claim \ref{claim: sup obtaind in alpha bar S}, there is $\gamma_1\in \Gamma$ such that
  $\alpha_{\bar S}(g) = -\log \|\pi_{w_0}(g\gamma_1 .w_{\bar S})\|$.
  Suppose that $g\gamma_2 .w_{\bar S}$ is another vector with
  $-\log \|\pi_{w_0}(g\gamma_2 .w_{\bar S})\| \ge 2\alpha_{\cusps_\Gamma}(g) + C_{\bar S}$,
  and $\{\pm g\gamma_1 .w_{\bar S}\} \neq \{\pm g\gamma_2 .w_{\bar S}\}$.
  Let $w_i = \gamma_i .w_{\bar S} \in W_{\bar S}$.
  By the definitions of $C_{\bar S}$ and $\delta_0$ we deduce that $\beta(g) < \delta_0$, and hence $\pi_\Gamma(g)\nin X_{\rm comp, 1}$.
  \begin{remark}\label{rem: uniqueness holds on comp 1}
    This shows that Claim \ref{lem: geometry near S} Point \ref{point: S unique} holds provided that
    $\pi_\Gamma(g)\in X_{\rm comp, 1}$ for any $C_{\bar S}' > -\log \delta_0$, and in particular, for $C_{\bar S}' = -\log \delta_0 + C_Q$.
  \end{remark}
  Since $\pi_\Gamma(g)\in \coreGn \setminus X_{\rm comp, 1}$,
  we deduce that for some $v\in V_{\cusps_\Gamma}$ we have $\alpha_{\cusps_\Gamma}(g) = -\log \|g.v\| > T_0/2$.
  To estimate $\Psi(v, w_i)$, note that there is a constant $C_0>0$ such that for all $v'\in V, w'\in W$ with $\|v'\|, \|w'\|\le 1$ we have $|\Psi(v',w')| \le C_0$ (direct computation leads to $C_0 = \sqrt 2$).
  Then $|\Psi(v, w_i)| = |\Psi(g.v, g.w_i)| \le C_0 \|g.v\|^2 \|g.w_i\| \le C_0 e^{-T_0}\|g.w_i\|$ for $i=1,2$.
  To estimate $\|g.w_i\|$, note that there is a constant $C_1>0$ such if
  $w'\in W$ satisfies $\|\pi_{w_0}{w'}\|\le 1$ and $Q(w') = -1$ then $\|w'\| \le C_1$ (direct computation leads to $C_1 = \sqrt 3$).
  This shows that $\|g.w_i\| \le C_1$, and hence $|\Psi(v, w_i)| \le C_0C_1e^{-T_0}$.
  Set $T_0 = -\log (\varepsilon_{\Gamma, \bar S}/(C_0C_1)) + 1$ for $\varepsilon_{\Gamma, \bar S}$ as in Claim \ref{claim: psi not small}, and deduce that $\Psi(v, w_i) = 0$ for $i=1,2$.

  Suppose that \begin{align}\label{eq: prop v}
    v = \gamma .v_{c_\Gamma} = \gamma g^{-1}_{c_\Gamma} .e_1,
  \end{align} for some $\gamma\in \Gamma$ and $c\in \cusps_\Gamma$.
  $QR$-Decomposition \ref{claim: QR} shows that one can represent
  \begin{align}\label{eq: prog g}
    g\gamma g_{c_{\Gamma}}^{-1} = k\ta(-t)\tu_\CC(s)
  \end{align}
  for some
  $k\in K_G, t\in \RR, s\in \CC$.
  Since $\|g\gamma g_{c_{\Gamma}}^{-1} .e_1\| = \|g.v\| < e^{-T_0/2}$ and $\|g\gamma g_{c_{\Gamma}}^{-1} .e_1\| = \|k\ta(-t)\tu_\CC(s).e_1\| = e^{-t/2}$, it follows that $t > T_0$.

  For every $\tau>0$ write $g_\tau = k\ta(-t + \tau)\tu_\CC(s)g_{c_\Gamma}\gamma^{-1} = k\ta(\tau)k^{-1}g$.
  Let $\tau_0 = t - T_0 > 0$.
  Note that
  \begin{align}\label{eq: g tau is not to much in the cusp}
    \alpha_{\cusps_\Gamma}(g_{\tau_0}) = -\log \|g_{\tau_0} .v\| = -\log \|k\ta(-T_0)\tu_\CC(s).e_1\| = T_0/2.
  \end{align}
  
  Let us estimate $-\log \|g_{\tau_0}.w_i\|$.

  Note that $\|g_{\tau_0}.w_i\| = \|k\ta(\tau_0) k^{-1}g.w_i\|$.
  Denote $w_i' = k^{-1}g.w_i$ and notice that
  \[0 = \Psi(k^{-1}g.v, w_i') \stackrel{\eqref{eq: prop v} + \eqref{eq: prog g}}{=}
    \Psi(k^{-1}k\ta(t)\tu_\CC(s)g_{c_\Gamma}\gamma^{-1}\gamma g_{c_\Gamma}^{-1}.e_1, w_i') =
    \Psi(\ta(t).e_1, w_i') =
    e^t\Psi(e_1, w_i').
  \]
  Hence $\Psi(e_1, w_i') = 0$.
  Now we may apply Claim \ref{claim: lin alg weird claim}, and deduce that
  \[\|\pi_{w_0}(g_{\tau_0}.w_i)\| = \|\pi_{w_0}(k\ta(\tau_0).w_i')\| \stackrel{\ref{claim: lin alg weird claim}}{\le}4e^{\tau_0} \|\pi_{w_0}(k.w_i')\| = 4e^{\tau_0} \|\pi_{w_0}(g.w_i)\|.\]
  Therefore,
  \begin{align*}
    -\log \|\pi_{w_0}(g_{\tau_0}.w_i)\|
     & \ge -\log \|\pi_{w_0}(g.w_i)\|  - 2\log 2 - \tau_0
    \\&\ge 2\alpha_{\cusps_\Gamma}(g) + C_{\bar S} - 2\log 2 - \tau_0 \\
     & = 2\alpha_{\cusps_\Gamma}(g_{\tau_0}) -\log \delta_0 + C_Q.
  \end{align*}
  By Claim \ref{claim: close to orbit imp in coreGn} we deduce that $g_{\tau_0} \in \pi_\Gamma^{-1}(\coreGn)$, and by Eq. \eqref{eq: g tau is not to much in the cusp} we deduce that $\pi_\Gamma(g_{\tau_0})\in X_{\rm comp,1}$. 
  This contradicts Claim \ref{lem: geometry near S} Point \ref{point: S unique} as shown in Remark \ref{rem: uniqueness holds on comp 1} for $\pi_\Gamma(g_{\tau_0})\in X_{\rm comp,1}$.
\end{proof}
\section{Example with low \texorpdfstring{$\varepsilon_\Gamma$}{epsilon\_Gamma}}
\label{sec: example}
In this section, we Prove Theorem \ref{thm: example}. 
The section is divided into five subsections. 
In Subsection \ref{ssec: construction of a lattice} we construct a nonarithmetic lattice $\Gamma$ such that $G/\Gamma$ is glued from two homogeneous subspaces $G/\Gamma_1, G/\Gamma_2$.
In Subsection \ref{ssec: construction of an orbit and reduction} we construct an orbits $H.x$ in $G/\Gamma$, which comes from a piece of a periodic orbit in $G/\Gamma_1$. Then we reduce the problem of evaluating $\delta(\stab_H(x))$ into two independent problems. One arithmetic and one geometric. 
We then solve them in Subsections \ref{ssec: arith} and \ref{ssec: proof of geometric} respectively.

\subsection{Construction of a lattice}
\label{ssec: construction of a lattice}
In this subsection, we will construct a sublattice $\Gamma<G$ and show that it is nonarithmetic.
\subsubsection*{General setting}
Let $Q\left((x_i)_{i=1}^4\right) = x_1^2 - x_2^2 - x_3^2 -x_4^2$ be a quadratic form.
Let $W^{\rm op}$ be $\RR^4$ thought of as row vectors, on which $\SL_4(\RR)$ acts from the right, and we consider $Q$ as a quadratic form on $W^{\rm op}$.
Here we will use
\[G = \SO(3,1)^0 = \{g\in \SL_4(\RR):Q(w.g) = Q(w),\, \forall w\in W^{\rm op}\}^0,\]
which is isogenic to $\SL_2(\CC)$, via the action of $\SL_2(\CC)$ on $W$ as in Claim \ref{claim: W}.
Recall that $\HH^3$ is a right $G$-space, where here we identify
\[\HH^3 = \{w\in W^{\rm op}: Q(w) = 1, w_1 > 0\}.\]
Let $p_0 = (1,0,0,0)\in \HH^3$, and note that $K_G := \stab_G(p_0)$ is the maximal compact subgroup in $G$ and is a copy of $\SO(3)$ embedded in $G$ by the action on the last $3$ coordinates.
Let $H = \SO(2,1)^0$, embedded in $\SO(3,1)^0$ by action of the first $3$ coordinates. The $H$ action preserves the sign of the last coordinate, that is, it preserves
\begin{align*}
  \HH^2 = \{v\in \HH^3: v_4 = 0\}\subset \HH^3, \\
  (\HH^3)^\pm = \{v\in \HH^3: \pm v_4 > 0\}\subset \HH^3,
\end{align*}
The maximal compact subgroup in $H$ is $K_H = K_G \cap H$ and is isomorphic to $\SO(2)$ acting by rotations on the second and third coordinates of $W^{\rm op}$.
\subsubsection*{The arithmetic components}

Recall that $\Gamma(7) = \ker(\SL_4(\ZZ)\to \SL_4(\ZZ/7)) < \SL_4(\ZZ)$ is a finite index torsion free subgroup.
Let $A_1, A_2 > 0$ be big integers $\equiv 1 \mod8$ such that $A_1/A_2$ is not a rational square.
Let
\[Q_i = 7x_1^2 - x_2^2 - x_3^2 - A_ix_4^2 \in \ZZ[x_1,x_2, x_3,x_4] \qquad \text{for } i=1,2,\]
be quadratic forms on $W^{\rm op}$.
Define \[\SO(Q_i, \ZZ) = \{\gamma\in \SL_4(\ZZ):Q_i(w.\gamma) = Q_i(w) \ \forall w\in W\}.\]
This is a subgroup of $\SO(Q_i, \RR)$, which is a lattice in it by Borel and Harish-Chandra's Theorem \cite{BorelHarishChandra}.
Let $\SO(Q_i, \ZZ)' = \SO(Q_i, \ZZ)\cap \Gamma(7)$. This is a lattice in $\SO(Q_i, \RR)$ and torsion-free.
Let $\Gamma_i = g_i^{-1}\SO(Q_i, \ZZ)'g_i$ where $g_i = \diag(\sqrt 7, 1, 1, \sqrt A_i)$.
Then $\Gamma_i$ is a torsion-free lattice in $G$. Note that, $\Gamma_3 = H\cap \Gamma_1 = H\cap \Gamma_2$ is a lattice in $H$ similarly constructed from the quadratic form $Q_3 = 7x_1^2 - x_2^2 - x_3^2$.
\begin{claim}\label{claim: big volume}
  For $i=1,2$ we have $\vol(G/\Gamma_i) = \Omega(A_i^{1/2})$.
\end{claim}
The proof relies on a certain arithmetic aspect of $\Gamma_i$ and is given in Subsection \ref{ssec: arith}.

\begin{claim}\label{claim: cocompact}
  The lattices $\Gamma_1, \Gamma_2$ are cocompact in $G$ and $\Gamma_3$ is cocompact in $H$.
\end{claim}
\begin{proof}
  For a quadratic form $Q$ in $d$ variables, the lattice $\SO(Q,\ZZ)$ is cocompact if and only if $Q(v) \neq 0$ for all $v\in \QQ^d\setminus \{0\}$ (see \cite[Prop.~5.3.4]{witte2001introduction}).
  Hence it is sufficient to show that $Q_i(v)\neq 0$ for all $v\in \QQ^4$.
  However, one can normalize $v$ so that $v\in \ZZ^4$ and one of its coordinates is odd. Then there are no solutions modulo $8$.
\end{proof}

\subsubsection*{Construction of a hybrid manifold}
Fix $i=1,2$ and consider the manifold $M_i = \HH^3/\Gamma_i$ the submanifold $V = \HH^2/\Gamma_3$, and the cover $\bar M = \HH^3/\Gamma_3$ of $M_i$.
These are indeed manifolds since $\Gamma_i$ has no torsion elements.
Let $\bar\rho: \HH^3\to \bar M$, $\rho_i: \HH^3\to M_i$, $\tau_i:\bar M\to M_i$ for each $i$ and $\rho_3: \HH^2\to V$ denote the standard projections.
We think of $V$ as a subset of $\bar M$.
\[\xymatrix{&\HH^3\ar@{->>}[d]^{\bar \rho}\ar@{->>}[rd]^{\rho_2}\ar@{->>}[ld]_{\rho_1}&\\M_1&\bar M\ar@{->>}[l]^{\tau_1}\ar@{->>}[r]_{\tau_2}&M_2}\]
\begin{claim}\label{claim: V injects}
  The projection $\tau_i:\bar M\to M_i$ restricts to an embedding on $V$.
  \[\xymatrix{&V\ar@{^(->}[d]\ar@{^(->}[rd]\ar@{^(->}[ld]&\\M_1&\bar M\ar@{->>}[r]_{\tau_2}\ar@{->>}[l]^{\tau_1}&M_2}\]
\end{claim}
The proof relies on a certain arithmetic aspect of $\Gamma_i$ and is given in Subsection \ref{ssec: arith}.
Denote by $V_i = \tau_i(V)$. By Claim \ref{claim: V injects} this is a submanifold.

We can now describe a new hyperbolic threefold $R$.
\begin{definition}[A hybrid manifold]
  Cut $M_i$ along $V_i$. The resulting manifold $M_i^{\rm cut}$ is a hyperbolic threefold with a hyperbolic surface boundary composed of two isometric copies of $V_i$, namely, $V^+_i, V_i^-$. Near $V_i^\pm$, the manifold $M_i^{\rm cut}$ is locally isometric to $\HH^2 \sqcup (\HH^3)^\pm$. 
  Glue $M_1^{\rm cut}$ to $M_2^{\rm cut}$ by gluing $V_1^+$ to $V_2^-$ and $V_1^-$ to $V_2^+$. 
  The resulting manifold is an orientable compact hyperbolic threefold $R$.
\end{definition}
For $i=1,2$ the embeddings $\chi_i:M_i^{\rm cut}\to R$, and the projections $\sigma_i:M_i^{\rm cut}\to M_i$.
\subsubsection*{Connectivity of $R$}
\begin{theorem}\label{thm: connectivity of M_i}
  For each $i = 1,2$, the manifold $M_i\setminus V_i$ is connected provided that $A_i$ is sufficiently large.
\end{theorem}
\begin{proof}
  Assume to the contrary that $M_i\setminus V_i$ is not connected. This implies that $M_i = V_i \sqcup M_i^+\sqcup M_i^-$, where $M^\pm_i$ are the different connected components of $M_i\setminus V_i$.
  We will estimate $\vol(M_i^{\pm})$. Since the matrix $g_{-1} = \diag(1, 1, 1, -1)$ normalize $\Gamma_i$, it acts on $M_i$.
  Since the $g_{-1}$ action replaces the two sides of $V$ in $\bar M$, it replaces the two sides of $V_i$ in $M_i$. 
  Thus $\vol(M_i^+) = \vol(M_i^-) = \frac{1}{2}\vol(M_i)$.
  By Claim \ref{claim: big volume}, $\vol(M_i^\pm) = \frac{1}{2}\vol(M_i) = \Omega(A_i^{1/2})$.
  This implies that the Cheeger constant 
  \[h(M_i) := \inf_{S\subseteq M_i}\frac{\vol(\partial S)}{\min(\vol(S), \vol(M_i\setminus S))} \le \frac{\vol(V_i)}{\min(\vol(M_i^+), \vol(M_i^-))} =  O(A_i^{-1/2}). \]
  By Burger's inequality \cite{buser1982note}, we deduce that $\lambda_1(M_i) = O(h(M_i)^2 + h(M_i)) = O(A_i^{-1/2})$, where $\lambda_1(M_i)$ is the minimal nontrivial eigenvalue of minus the laplacian operator $-\Delta$ on $M_i$.
  By Property $(\tau)$ for congruence subgroups in arithmetic groups (See \cite{selberg1965estimation, kazhdan1967connection, burger1991ramanujan, clozel2003demonstration}), there is an absolute constant $\lambda_0$ such that $\lambda_1(M_i) \ge \lambda_0$. This contradicts our previous estimate $\lambda_1(M_i) = O(A_i^{-1/2})$, as desired.
\end{proof}
We conclude that $R$ is connected.
\begin{remark}[Avoiding property $(\tau)$]
  The use of property $(\tau)$ is the least elementary piece of the arguments in this section and can be avoided, as 
  Theorem \ref{thm: connectivity of M_i} is not necessary to the proof, and is only provided to give the reader a better picture of $R$ and simplify the terminology.
\end{remark}
Since $R$ is a connected compact hyperbolic threefold, we deduce that $R \cong \HH^3/\Gamma$ for some cocompact lattice $\Gamma < G$, which is our desired nonarithmetic lattice. Since $A_1/A_2$ is not a square we get the following theorem.
\begin{theorem}[{\cite[\S2.9]{gromov1987non}}]
  The lattice $\Gamma$ is non-arithmetic.
\end{theorem}
\subsection{Reduction of Theorem \ref{thm: example} into arithmetic and hyperbolic questions}
\label{ssec: construction of an orbit and reduction}
In this section, we will reduce the construction of an element $g$ as in Theorem \ref{thm: example} to an arithmetic question. 
\begin{definition}
  For every complete hyperbolic manifold $M$ and a point $p\in M$ denote by ${\rm Ray}_p$ the collection of geodesic rays $\gamma:[0,\infty)\to M$ originating from $p = \gamma(0)$.
  The derivative at $0$ gives a metric isomorphism ${\rm Ray}_p \cong S^{\dim M - 1}$. 
\end{definition}
\begin{claim}\label{claim: critical exponent by rays}
  Let $\Lambda < H$ be a subgroup and $U\subseteq \HH^2/\Lambda$ be a precompact open subset. Then for every $p\in U$ we have 
  \[H.\dim(\left\{\gamma\in {\rm Ray}_p:\gamma(t)\in U \ \forall t\ge 0\right\}) \le \delta(\Lambda).\]
\end{claim}
\begin{proof}
  Let $\pi_\Lambda:\HH^2\to \HH^2/\Lambda$ denote the standard projection.
  Let $U_0\subseteq \HH^2$ be precompact open set so that $\pi_\Lambda(U_0) = U$ and denote $\tilde U = \pi_\Lambda^{-1}(U) = \bigcup_{\lambda\in \Lambda} U_0.\lambda$. 
  Let $\Lambda' = \bra \lambda\in \Lambda:U_0.\lambda \cap U_0\neq \emptyset\ket$. 
  Since $U_0$ is precompact and $\Lambda$ discrete, it follows that the set of generators we wrote to $\Lambda'$ is finite and hence $\Lambda'$ is geometrically finite. 
  We will use Sullivan \cite[Thm.~1]{sullivan1984entropy} to give a lower bound on $\delta(\Lambda')\le \delta(\Lambda)$.

  Let $\tilde p \in U_0$ be a preimage of $p$ and note that there is a bijection between rays ${\rm Ray}_p$ and ${\rm Ray}_{\tilde p}$ that gives an equality of the Hausdorff dimensions 
  \[
    H.\dim\left(\left\{\gamma\in {\rm Ray}_p:\gamma(t)\in U \ \forall t\ge 0\right\}\right)  = 
    H.\dim\left(\left\{\gamma\in {\rm Ray}_{\tilde p}:\gamma(t)\in \tilde U \ \forall t\ge 0\right\}\right)
  .\]
  Denote $X = \left\{\gamma\in {\rm Ray}_{\tilde p}:\gamma(t)\in \tilde U \ \forall t\ge 0\right\}$ and let $\gamma \in X$. We will show that $\lim_{t\to \infty}\gamma(t) \in \partial \HH^2$ in fact lies in the limit set $D(\Lambda')$. 

  Since 
  $(\{t\in [0,\infty):\gamma(t)\in U_0.\lambda\})_{\lambda\in \Lambda}$ is an open cover of $[0,\infty)$ by bounded sets, 
  there is a sequence $t_0 = 0<t_1<t_2<\dots$ such that $\lim_{j\to \infty}t_j = \infty$
  and a sequence $(\lambda_j)_{j=0}^\infty \subseteq \Lambda$ so that $\lambda_0 = I$ and for all $j=0,1,\dots$ and for all $t\in [t_j, t_{j+1}]$ we have $\gamma(t)\in U_0.\lambda_j$. 
  Note that for all $j=1,2,\dots$ we have that $\gamma(t_j)\in U_0.\lambda_j \cap U_0.\lambda_{j-1}$. 
  Hence $\lambda_{j-1}\lambda_j^{-1}\in \Lambda'$. By induction we deduce that $\lambda_j \in \Lambda'$ for all $j\ge 0$. This implies that \[\lim_{t\to \infty}\gamma(t) = \lim_{j\to \infty}\tilde p.\lambda_j \in D(\Lambda').\] 
  Hence the limit embeds $X$ in $D(\Lambda')$ which implies that 
  \[H.\dim(X) \le H.\dim(D(\lambda')) \stackrel{\text{\cite[Thm.~1]{sullivan1984entropy}}}{=} \delta(\Lambda') \le \delta(\Lambda).\]
\end{proof}
Direct computation shows that the normalizer $N(H)$ of $H$ is given by 
\[N(H) = H \cup g_0 H, \qquad g_0 = \diag(1, 1,-1,-1).\]
\begin{obs}[The relation between $H$-orbits and immersed hyperbolic surfaces]\label{obs: immersed hyperbolic}
  Let $\pi_{K_G}: G / \Gamma \to \HH^3 / \Gamma$ be the standard projection. 
  Any $H$-orbit $H.\pi_\Gamma(g)$ in $G/\Gamma$ is projected to an immersion 
  $\iota_g:\HH^2 / \Gamma_g' \to \HH^3/\Gamma$, where $\Gamma_g' = g \Gamma g^{-1} \cap N(H)$. 
  The immersion is not necessarily bijective, however, the self-intersection is a countable union of geodesics, that is, the set $\{p\in\HH^2 / \Gamma_g: \iota^{-1}(\iota(p)) \neq \{p\}\}$ is a countable union of geodesics in $\HH^2 / \Gamma_g'$. 

  On the other hand, every immersion of open hyperbolic surface $\iota_U: U \to \HH^3/\Gamma$ that satisfies that the self-intersection is a countable union of geodesics factors as $\iota_U = \iota_g\circ \iota_0$ for some $g\in G$, where $\iota_0:U\to \HH^2/\Gamma'_g$ is an isometric embedding.
\end{obs}
\begin{definition}[Semi-periodic immersed hyperbolic surface]\label{def: semi periodic surface}
  Let $H.x_0$ be a periodic orbit in $G/\Gamma_1$, where $x_0 = \pi_{\Gamma_1}(g)$. Assume that it is not the periodic orbits $\pi_{\Gamma_1}(H)$ or $\pi_{\Gamma_1}(Hg_0)$, whose projection to $M_1$ lands in $V_1$. 
  Let $\iota_{0}:\HH^2 / \Lambda\to M_1$ denote the corresponding immersion of hyperbolic surface, where $\Lambda = (\Gamma_1)_{g}' = g\Gamma_1g_1^{-1} \cap N(H)$ is a lattice in $N(H)$. 
  Then $\HH^2 / \Lambda$ is a finite volume compact space. Let $U_1$ be a connected component of $(\HH^2/\Lambda)\setminus \iota_0^{-1}(V_1)$. 
  Then $\iota_0$ restricts to an embedding $\iota_0|_{U_1}:U_1\to M_1^{\rm cut}$. 
  Let $\iota_1 = \chi_1\circ \iota_0|_{U_1}:U_1\to \HH^3/\Lambda$. By Observation \ref{obs: immersed hyperbolic} the immersion $\iota_1$ 
  factors as $\iota_1 = \iota_{g_2} \circ \iota_2$ for some $g_2\in G$, $\iota_2:U_1\to \HH^2 / \Gamma_{g_2}'$. 
  Then $\iota_{g_2}:\HH^2/\Gamma_{g_2}'\to \HH^3/ \Gamma$ is a semi-periodic surface. 
\end{definition}
\begin{proof}[Reduction of Theorem \ref{thm: example} into two propositions]
  Let $g_0, \Lambda, U_1, \iota_1, g_2, \iota_{2}$ as in Definition \ref{def: semi periodic surface}. 
  Since $V_1$ is a hyperbolic surface in $M_1$ that differs from $\iota_0(\HH^2/\Lambda)$, we deduce that $\iota_0^{-1}(V_1)$ is a union of geodesics in $H/\Lambda$, and hence $U_1$ has finite diameter. 
  Hence $\iota_2(U_1)$ is precompact in $\HH^2/\Gamma_{g_2}'$. 
  Let $\rho:\HH^2/\Gamma_{g_2}\to \HH^2/\Gamma_{g_2}'$ denote the standard projection. It is a proper covering map of index at most $2$. Hence $\rho^{-1}(\iota_2(U_1))$ is precompact in $\HH^2/\Gamma_{g_2}$.
  Let $p\in \rho^{-1}(\iota_2(U_1))$. 
  It follows that 
  \begin{align}\label{eq: delta by rays}
    \begin{split}
      \delta(\Gamma_{g_2}) &\stackrel{\ref{claim: critical exponent by rays}}{\ge} 
      H.\dim(\left\{\gamma\in {\rm Ray}_{p}:\gamma(t)\in \rho^{-1}(\iota_2(U_1)) \ \forall t\ge 0\right\})
      \\&= 
      H.\dim(\left\{\gamma\in {\rm Ray}_{\rho(p)}:\gamma(t)\in \iota_2(U_1) \ \forall t\ge 0\right\})
      \\&=
      H.\dim(\left\{\gamma\in {\rm Ray}_{p'}:\gamma(t)\in U_1 \ \forall t\ge 0\right\}).
    \end{split}
  \end{align}
  for $p' = \iota_2{\rho(p)} \in U_1\subseteq \HH^2/\Lambda$. 
  Let $\pi_{\Lambda}: \HH^2 \to \HH^2/\Lambda$ denote the projection. 
  We will express the right-hand side of Eq. \eqref{eq: delta by rays} by this universal cover. 
  Let $\tilde p \in \pi_{\Lambda}^{-1}(p')$.
  Since $U_1$ is the connected component of $(\HH^2/\Lambda)\setminus \iota_0^{-1}(V_1)$, 
  we can lift each geodesic ray to the universal cover $\HH^2$ of $U_1$ get an equality 
  \begin{align}\label{eq: lift to cover}
    H.\dim(\left\{\gamma\in {\rm Ray}_{p'}:\gamma(t)\in U_1 \ \forall t\ge 0\right\}) = 
    H.\dim(\left\{\gamma\in {\rm Ray}_{\tilde p}:\gamma(t)\nin \cL \ \forall t\ge 0\right\}),
  \end{align}
  where $\cL = \pi_{\Lambda}^{-1}(\iota_0^{-1}(V_1))$ is a union of lines. 
  we now introduce the following propositions on $\cL$.
  \begin{prop}\label{prop: far lines}
    Let $\zeta:\HH^2\to M_1$ be a locally isometric immersion. Then the set 
    $\cL_\zeta = \zeta^{-1}(V_1)$ is a union of hyperbolic lines such that for every two geodesic lines $\ell_1\neq \ell_2\subseteq \cL$ we have 
    $d_{\HH^2}(\ell_1, \ell_2) > \frac{1}{2}\log A_1 + O(1)$. 
  \end{prop}
  \begin{prop}\label{prop: high dimension away of lines}
    Let $\cL\subseteq \HH^2$ be a union of lines so that for every two geodesic lines $\ell_1\neq \ell_2\subseteq \cL$ we have 
    $d_{\HH^2}(\ell_1, \ell_2) > \mathcal A$. Then 
    \[H.\dim(\left\{\gamma\in {\rm Ray}_{p}:\gamma(t)\nin \cL \ \forall t\ge 0\right\}) > 1 - O(1/\mathcal A),\]
    for every $p \in \HH^2 \setminus \cL$. 
  \end{prop}
  The combination of Propositions \ref{prop: far lines} and \ref{prop: high dimension away of lines}, together with Eqs. \eqref{eq: delta by rays} and \eqref{eq: lift to cover}, show that $\delta(\Gamma_{g_2}) > 1-O(1/\log A_1)$. 
  
  Leaving the proofs of these propositions to the next subsections, it is left to find $\iota_0, g_0, \Lambda, U_1, \iota_1, g_2, \iota_{2}$ as in Definition \ref{def: semi periodic surface} so that $\Gamma_{g_2}$ is not periodic.
  It follows from \cite[Thm.~4.1]{fisher2021finiteness} or \cite[Prop.~12.1.]{benoist2022geodesic}, that if $\iota_{0}(\HH^2 / \Lambda)$ intersects $V_1$ non-orthogonally then $\Gamma_{g_2}$ is not periodic. Such an immersion exists by the density of closed $H$-orbits in $G/\Gamma_1$.

\end{proof}

\subsection{Behavior of the arithmetic space near the cutting plane}
\label{ssec: arith}
In this section we prove Proposition  \ref{prop: far lines}, as well as claims \ref{claim: V injects} and \ref{claim: big volume}. 
We begin the section by linearising the distance from a hyperbolic plane in $\HH^3$. 
\subsubsection*{Linearization of the distance from a hyperbolic plane}
\begin{definition}[The representation $W$]
  Let $W\cong \RR^4$ denote the standard representation of $\SL_4(\RR)$ on which it acts from the left. 
  Note that the quadratic form $Q(x_1,x_2,x_3,x_4) = x_1^2 - x_2^2 - x_3^2 - x_4^2$ is preserved by the $G$ action (this time thought of as a quadratic form on $W$), similarly to the case with $W^{\rm op}$.
\end{definition}

\begin{obs}[Identifying $\HH^2$ in $\HH^3$]
  Let $\pi_{K_G}:G \to \HH^3$ denote the standard projection. 
  Note that $\stab_{G}(w_0) = H$ and
  \[K_G.w_0 = \{(0,w_2, w_3, w_4)^t:w_2^2 + w_3^2 + w_4^2 = 1\} = \{w\in W:Q(w) = -1, \|w\|=1\}.\]   
  Hence 
  \[\{g\in G:\|g.w_0\| = 1\} = \{g\in G:g.w_0\in K_G.w_0\} = K_gH.\]
  Hence $K_GH = \pi_{K_G}^{-1}(\HH^2)$. 
\end{obs}
\begin{definition}[Hyperbolic geometry relative to a hyperbolic plane]\label{def: hyperbolic near plane}
  Let $\varphi:\HH^3\to \RR$ be the signed distance form $\HH^2$, that is,
  \[\varphi(p) = \begin{cases}
      d_{\HH^3}(p, \HH^2),  & \text{if }p\in (\HH^3)^+, \\
      -d_{\HH^3}(p, \HH^2), & \text{if }p\in (\HH^3)^-, \\
      0,                    & \text{if }p\in \HH^2.
    \end{cases}\]
  Note that $\varphi$ is differentiable and the gradient is of fixed size $1$.

  Through every point $p\in \HH^2$ passes a unique geodesic $\xi_p:\RR\to \HH^3$ with $\xi_p(0) = p$, which is orthogonal to $\HH^2$ and oriented towards $(\HH^3)^+$.
  These geodesics forms a foliation of $\HH^3$, and satisfy $\varphi(\xi_p(t)) = t$ for all $p\in \HH^2, t\in \RR$.
  For every $h\in H$ we have that $\xi_{p.h} = \xi_{p}.h$.
  Recall $w_0 = (0,0,0,1)^t\in W$, and define $\psi:G\to \RR$ by $\psi(g) = (g.w_0)_1$.
  Since $\psi$ is invariant from the left to $K_H$ it descends to a map $\psi:\HH^3\to\RR$. 
\end{definition}
\begin{claim}\label{claim: phi psi relation}
  For every $p\in \HH^3$ we have $\sinh(\varphi(p)) = \psi(p)$. 
\end{claim}
\begin{proof}
  Both functions $p\mapsto \psi(p)$ and $p\mapsto \sinh(\varphi(\pi_{K_G}(p)))$ are invariant from the right to $H$, which allows us to test this equality only on points of the form $\xi_{p_0}(t) = (\sinh(t), 0,0,\cosh(t))$, on which the equality holds.
\end{proof}
\begin{corollary}\label{cor: hyp dist is lenght of vec}
  Let $\HH^2.g_1$ be a hyperbolic plane and $p_2 = \pi_{K_G}(g_2)$ be a point. 
  Then $d_{\HH^3}(p_2, \HH^2.g_1) = \log \|g_2g_1^{-1}.w_0\| + O(1)$. 
\end{corollary}
\begin{proof}
  Using Claim \ref{claim: phi psi relation} we deduce that 
  \begin{align*}
    d_{\HH^3}(p_2, \HH^2.g_1) &
    = d_{\HH^3}(p_2.g_1^{-1}, \HH^2)
    = |\varphi(p_2.g_1^{-1})| \\ &
    \leftstackrel{\ref{claim: phi psi relation}}{=} |\sinh^{-1}(\psi(g_2g_1^{-1}))|
    = |\sinh^{-1}((g_2g_1^{-1}.w_0)_1)|.
  \end{align*}
  The result follows from the fact that for every vector $w\in W$ with $Q(w) = -1$ we have
  \[|\sinh^{-1}(w_1)| = \log \|w\| + O(1),\]
  which is a direct computation. 
\end{proof}
Denote by $C_0$ the implicit constant in Corollary \ref{cor: hyp dist is lenght of vec} so that 
\begin{align}\label{eq: distance to plane}
  \left|d_{\HH^3}(p_2, \HH^2.g_1) - \log \|g_2g_1^{-1}.w_0\|\right| \le C_0.
\end{align}

Finally, we prove the following claim
\begin{claim}\label{claim: hyperbolic regulation}
  Let $v\in W$ so that $Q(v) = -1$. Then there are $k, k'\in K_G$ and $t\in \RR$ such that
  $k'\ta(t)k. v = w_0$ and $\cosh(t) \le \|v\|$.
\end{claim}
\begin{proof}
  Express $v = (v_1, v_2, v_3, v_4)^t$.
  For some $k\in K_G$, we have $k. v = (v_1, v_2', 0, 0)^t$ with $v_2'>0$ and $v_1^2 - (v_2')^2 = -1$.
  This implies that for some $t'>0$ we have $v_1 = \sinh t'$ and $v_2' = \cosh(t')$.
  In particular $\cosh(t') < \|v\|$ and $\ta(-t')k.v = (0,1,0,0)^t$. The desired follows for $t=-t'$ and $k' = \left(\begin{smallmatrix}
      1&0&0&0\\
      0&0&0&-1\\
      0&0&1&0\\
      0&1&0&0
    \end{smallmatrix}\right)$.
\end{proof}

\subsubsection*{Arithmetic properties}
Fix $i=1,2$ for the entire subsection. 
Consider the vector $w_0 = (0,0,0,1)^t\in W$ and the set $W_{\Gamma_i} = \Gamma_i. w_0\subset W$.
\begin{claim}\label{claim: discreteness of tiny orbits}
  Let $v\in W_{\Gamma_i}$. Then either $v = w_0$, or $\|v\| \ge \sqrt {A_i/7}$.
\end{claim}
\begin{proof}
  Let $W_\ZZ\cong \ZZ^4$ be the integer vectors in $W \cong \RR^4$. 
  By the definition of $\Gamma_i$ we deduce that $\Gamma_i$ preserves that lattice $W_{\ZZ, i} = \sqrt A_i g_i^{-1}W_\ZZ\subset W$, where $g_i = \diag(\sqrt 7, 1, 1, \sqrt A_i)$.

  Let $v\in W_{\Gamma_i}\setminus \{w_0\}\subseteq W_{\ZZ, i}$.
  Note that $Q(v) = -1$. Assume to the contrary that $v\in \RR w_0$.
  Then we must have $v= - w_0$.
  However, since $\Gamma_i$-s action on $W_{\ZZ, i}$ descends to a trivial action on $W_{\ZZ, i}/7W_{\ZZ, i}$, we deduce that $W_{\Gamma_i}\cap -W_{\Gamma_i} = \emptyset$, which contradicts $v= - w_0$. Hence we have $v\nin \RR w_0$.
  Let $j=1,2,3$ be an index satisfying $(v)_j\neq 0$.
  Then since $v\in W_{\ZZ, i}$ we must have $|(v)_j| \ge \begin{cases}
      \sqrt {A_i},         & \text{if }j=2,3, \\
      \sqrt {A_i}/\sqrt 7, & \text{if }j=1.   \\
    \end{cases}$.
  Hence $\|v\| \ge \sqrt {A_i/7}$.
\end{proof}
\begin{claim}\label{claim: new alpha unique}
  There is $c>0$ independent of $A_i$ such that for all $g\in G$ there is at most one $z\in g.W_{\Gamma_i}$ such that $\|z\| < cA_i^{1/4}$.
\end{claim}
\begin{proof}
  Suppose that $z\neq z' \in g.W_{\Gamma_i}$ and $\|z\|, \|z'\| < cA_i^{1/4}$.
  Assume that
  \begin{align}\label{eq: z z' by gamma}
    z = g \gamma.w_0\quad \text{and}\quad z' = g \gamma'.w_0.
  \end{align}
  Every $z\in g.W_{\Gamma_i}$ has $Q(z) = -1$. Applying Claim \ref{claim: hyperbolic regulation} to $z$, we get that there are $k, k'\in K_G$ and $t\in \RR$ with $\cosh(t) \le \|z\| \le cA_i^{1/4}$ such that $k' \ta(t) k. z = w_0$. In particular, $e^t < 2cA_i^{1/4}$.
  Substituting $z = g \gamma.w_0$ to the previous equality, we obtain $k' \ta(t) kg \gamma. w_0 = w_0$. Using $H = \stab_G(w_0)$, we deduce that $k' \ta(t) kg \gamma \in H$.

  Since $\Gamma_i\cap H = \Gamma_3$ is cocompact in $H$ (see Claim \ref{claim: cocompact}) and independent of $A_i$, there is $C>0$ such that for every $h\in H$ there is $\gamma_3\in H\cap \Gamma_i$ such that $\|\gamma_3h\|_{\rm op} < C$.
  Here we use the operator norm defined by the action on $W$.
  Applying this to $(k' \ta(t) kg \gamma)^{-1} \in H$, we deduce that for some $\gamma_3\in H\cap \Gamma_i$ we have $\|\gamma_3(k' \ta(t) kg \gamma)^{-1}\|_{\rm op} < C$.
  Set
  \begin{align}\label{eq: def h}
    h = \gamma_3(k' \ta(t) kg \gamma)^{-1},\quad\text{satisfying}\quad\|h\|_{\rm op} < C.
  \end{align}
  Then $h k' \ta(t) k .z = w_0$.
  Denote $v = h k' \ta(t) k.z = w_0$ and $v' = h k' \ta(t) k .z'$.
  We will now estimate $v'$. Note that
  \[v' = h k' \ta(t) k .z' \stackrel{\eqref{eq: z z' by gamma}}{=} h k' \ta(t) k g\gamma'.w_0 \stackrel{\eqref{eq: def h}}{=} \gamma_3 \gamma^{-1} \gamma'.w_0 \in W_{\Gamma_i}.\]
  By Claim \ref{claim: discreteness of tiny orbits}, we obtain that $\|v'\|\ge \sqrt {A_i/7}$.
  On the other hand,
  \[\|v'\| \le \|z'\|\cdot \|k \ta(t) k' h\|_{\rm op} \le
    cA_i^{1/4}\|h\|_{\rm op}e^t \le
    cA_i^{1/4}\cdot C\cdot 2cA_i^{1/4}\le 2c^2 CA_i^{1/2}.
  \]
  Therefore, choosing $c$ for which $c^2 < \frac{1}{2C\sqrt 7}$, we obtain a contradiction to Claim \ref{claim: discreteness of tiny orbits} and the desired uniqueness follows.
\end{proof}

\begin{claim}\label{claim: disjoint nbds}
  Let $\cA = \frac{1}{4}\log A_i - C_0 + \log c = \Theta(\log A_i)$, wehre $C_0$ is as in Eq. \eqref{eq: distance to plane} and $c$ as in Claim \ref{claim: new alpha unique}. 
  Define
  \[S = \{p\in \HH^3: d(p, \HH^2) < \cA\}.\]
  Then for every $\gamma\in \Gamma_i$, either 
  \begin{enumerate}
    \item
    $\gamma\in \Gamma_3$ and then $S.\gamma = S$,
    \item or $\gamma\nin H$ and then $S.\gamma \cap S = \emptyset$.
  \end{enumerate}
\end{claim}
\begin{proof}
  If $\gamma\in H$ then the first option holds. 
  If $\gamma \nin H$, assume to the contrary that $\pi_{K_G}(g)\in S.\gamma \cap S$. 
  Then 
  \[\log \|g.w_0\| \stackrel{\eqref{eq: distance to plane}}{\le} d(\pi_{K_G}(g), \HH^2) + C_0 < \cA + C_0 = \frac{1}{4}\log A_i + \log c\]
  Similarly, 
  \[\log \|g\gamma^{-1}.w_0\| \stackrel{\eqref{eq: distance to plane}}{\le} d(\pi_{K_G}(g), \HH^2.\gamma) + C_0 < \cA + C_0 = \frac{1}{4}\log A_i + \log c\]
  Hence, by Claim \ref{claim: new alpha unique} we deduce that $g\gamma^{-1}.w_0 = g.w_0$, which implies that $\gamma\in H$. This contradicts the assumption and completes the proof.
\end{proof}
The following corollary is immediate. 
\begin{corollary}\label{cor: far planes}
  For every two different hyperbolic planes $\HH^2.\gamma, \HH^2.\gamma'$ for $\gamma, \gamma'\in \Gamma_i$ we have $d_{\HH^3}(\HH^2.\gamma, \HH^2.\gamma') \ge 2\cA$. \qed
\end{corollary}
\begin{corollary}[Strengthening of Claim \ref{claim: V injects}]\label{cor: S injects}
  Recall the projection $\bar\rho:\HH^3\to \HH^3 / \Gamma_3 = \bar M$ and recall the standard projection $\tau_i:\bar M\to M_i$. 
  Let $\bar S = \bar\rho(S)$. Then $\tau_i|_{\bar S}$ is one to one.
\end{corollary}
\begin{proof}
  Assume that for $p_1, p_2 \in \bar S$ we have $\tau_i(p_1) = \tau_i(p_2)$. 
  Choose $\tilde p_j \in \bar\rho^{-1}(p_j)$ for $j=1,2$. Since $\tau_i\circ \bar\rho = \pi_{\Gamma_i}$ agrees on $\tilde p_1, \tilde p_2$ we deduce that for some $\gamma\in \Gamma_i$ we have $\tilde p_1 = \tilde p_2.\gamma$. Hence $\tilde p_1 \in S\cap S.\gamma$. Claim \ref{claim: disjoint nbds} implies that $\gamma \in \Gamma_3$, which in turn implies that $p_1 = p_2$. This proves the injectivity of $\tau_i$ on $\bar S$. 
\end{proof}
\begin{proof}[Proof of Claim \ref{claim: big volume}]
  In view of Corollary \ref{cor: S injects} it is sufficient to show that $\vol(\bar S) = \Omega(A_i^{1/2})$. 

  Recall the map $\varphi:\HH^3\to \RR$ form definition \ref{def: hyperbolic near plane}. 
  Its gradient is of fixed size $1$.
  This implies that for every set $\Omega\subseteq \HH^3$,
  \begin{align}\label{eq: volume formula 1}
    {\rm Vol}(\Omega) = \int_{-\infty}^\infty{\rm Area}\left(\Omega\cap \varphi^{-1}(t)\right)\bd t.
  \end{align}
  Recall the foliation $\{\xi_p:p\in \HH^2\}$ of $\HH^3$ form definition \ref{def: hyperbolic near plane}.
  This gives a parametrization $\HH^2\to \varphi^{-1}(t_0)$ for every $t_0\in \RR$ by $p\mapsto \xi_p(t_0)$.
  This parametrization can be seen to expand the Riemannian metric by $\cosh(t)$.
  Therefore, for every $\Omega\subseteq \HH^2, t_0\ge 1$, 
  \begin{align}\label{eq: volume formula 2}
    \vol(\{\xi_p(t):t\in [ -t_0, t_0], p\in \Omega\}) = {\rm Area}(\Omega)\int_{-t_0}^{t_0}\cosh^2(t)\bd t = \Theta(e^{2t_0}{\rm Area}(\Omega)).
  \end{align}

  The function $\varphi$ is $H$ invariant and hence descends to a function $\bar\varphi:\HH^3/\Gamma_3=\bar M\to \RR$.
  For every $h\in H$, we have that $\xi_{p.h} = \xi_{p}.h$.
  Thus, the foliation $\xi_\bullet$ descends to $\bar M$ as follows:
  for every $q\in V = \HH^2 / \Gamma_3$ there is a geodesic $\bar \xi_q:\RR\to \bar M$, and these geodesics form a foliation of $\bar M$.
  Choosing a fundamental domain $\Omega\subseteq \HH^2$ to $V = \HH^2 / \Gamma_3$ we deduce from Eq. \eqref{eq: volume formula 2} that 
  \begin{align*}
    \vol(\bar S) 
    &= \vol(\{\bar\xi_q(t):t\in (-\cA, \cA), q\in V\})
    = \vol(\{\xi_p(t):t\in (-\cA, \cA), p\in \Omega\})\\&
    = \Theta(e^{2\cA}{\rm Area}(\Omega))
    = \Theta(\sqrt A_i{\rm Area}(V)).
  \end{align*}
  Since ${\rm Area}(V)$ is fixed we deduce that $\vol(\bar S) = \Theta(\sqrt A_i)$. 
  By Corollary \ref{cor: S injects} we obtain $\vol(M_i) = \Omega(\sqrt A_i)$.
  The equality $\vol(G/\Gamma_i) = \Omega(\vol(M_i))$ completes the proof.
\end{proof}
\begin{proof}[Proof of Proposition \ref{prop: far lines}]
  Let $\zeta: \HH^2 \to M_i$ be a locally isometric immersion. 
  Recall the standard projection $\rho_i:\HH^3\to M_i$. 
  Then $\zeta$ factors as $\zeta = \rho_i\circ \tilde \zeta$ for some isometric embedding 
  $\tilde \zeta:\HH^2\to \HH^3$. 

  Note that since $V_i = \rho_i(\HH^2)$ we have 
  \[\rho_i^{-1}(V_i) = \rho_i^{-1}(\rho_i(\HH^2)) = \bigcup_{\gamma\in \Gamma_i}\HH^2.\gamma.\]
  Hence \[\zeta^{-1}(V_i) = \bigcup_{\gamma\in \Gamma_i}\tilde \zeta^{-1} (\HH^2.\gamma)\]
  This is a representation of $\zeta^{-1}(V_i)$ as a union of lines.
  To complete the proof we need to show that for every $\gamma, \gamma'$ for which $\tilde \zeta^{-1} (\HH^2.\gamma)\neq \tilde \zeta^{-1} (\HH^2.\gamma')$ we have 
  \[d_{\HH^2}(\tilde \zeta^{-1} (\HH^2.\gamma), \tilde \zeta^{-1} (\HH^2.\gamma')) \ge 2\frac{1}{2} \log A_i + O(1).\]
  However, since $\tilde \zeta^{-1}$ is an isometric embedding we obtain
  \[d_{\HH^2}(\tilde \zeta^{-1} (\HH^2.\gamma), \tilde \zeta^{-1} (\HH^2.\gamma')) \ge 
  d_{\HH^3}(\HH^2.\gamma, \HH^2.\gamma') \stackrel{\ref{cor: far planes}}{\ge} 
  2\cA.\]
\end{proof}
\subsection{Proof of Proposition \ref{prop: high dimension away of lines}}
\label{ssec: proof of geometric}
To prove Proposition \ref{prop: high dimension away of lines}, we will first rephrase it as a question on an estimate of the Hausdorff dimension of a certain Cantor set, and then bound it. 
\subsubsection*{Reformulation of Proposition \ref{prop: high dimension away of lines}}
Let $\cL = \bigcup_{\ell \in L} \ell$ such that for every $\ell_1, \ell_2\in L$ we have $d_{\HH^2}(\ell_1, \ell_2) \ge \cA$. We may assume without loss of generality that $\cA \ge 10$. Let $p\in \HH^2\setminus \cL$. 
Denote by $U$ the connected component of $\HH^2\setminus \cL$ containing $p$. Denote by $L'\subseteq L$ the collection of lines composing the boundary of $U$. 
Denote by $D(U)\subseteq \partial \HH^2$ the limit set of $U$. Since $U$ is convex, 
\[D(U)=\{q\in \partial \HH^2: [p,q)\subseteq U\}. \]
Hence we have an equality of Hausdorff dimensions 
\begin{align*}
  H.\dim(\left\{\gamma\in {\rm Ray}_{p}:\gamma(t)\nin \cL \ \forall t\ge 0\right\}) &= 
  H.\dim(\left\{\gamma\in {\rm Ray}_{p}:\gamma(t)\in U \ \forall t\ge 0\right\})\\&=
  H.\dim(D(U)).  
\end{align*}For every geodesic line $\ell\subseteq \HH^2\setminus \{p\}$ denote by $x_\ell, y_\ell \in \partial \HH^2$ the limit points of $\ell$
so that the ray $[p,x_\ell)$ can be rotated less then $\pi$ degrees counterclockwise about $p$ to obtain $[p,y_\ell)$. 
Denote by $I_\ell\subset \partial \HH^2$ the open interval with the boundary points $x_\ell$ and $y_\ell$, which lies on the other side of $\ell$ than $p$.
Then $D(U) = \partial \HH^2 \setminus \bigcup_{\ell \in L'}I_L$. 
The intervals $I_L$ are disjoint. 
\begin{claim}
  If $\ell_1, \ell_2$ are nonintersecting lines in $\HH^2\setminus \{p\}$ then 
  \begin{align}\label{eq: dist is sinh}
    \sinh(d(\ell_1, \ell_2)/2) = \sqrt{|[x_{\ell_1}, x_{\ell_2}; y_{\ell_2}, y_{\ell_1}]|},
  \end{align}
  where $[a,b;c,d] = \frac{(a-c)(b-d)}{(a-d)(b-c)}$ is the cross ratio on $\PP^1_\RR \cong \partial \HH^2$. 
\end{claim}
\begin{proof}
  The choice to labeling of the limit points of $\ell_1,\ell_2$ ensures that  and $x_{\ell_1}, y_{\ell_1}, x_{\ell_2}, y_{\ell_2}$ are in this circular order on $\partial\HH^2$.
  Up to an isometry, we may assume that 
  \[
    x_{\ell_1} = e^t, 
    y_{\ell_1} = -e^t, 
    x_{\ell_2} = -1, 
    y_{\ell_2} = 1,
  \]
  where $t = d_{\HH^2}(\ell_1, \ell_2)$. 
  Then Eq. \eqref{eq: dist is sinh} is a direct computation.
\end{proof}
Identify $\HH^2$ with the Poincaré half-plane model in $\CC\cup \{\infty\}$. 
Sample one $\ell_0\in L'$. Up to an isometry we may assume that 
$x_{\ell_0} = 1, y_{\ell_0} = 0$ so that $|p-1/2| < 1/2$ and $I_{\ell_0} = \PP^1_\RR\setminus [0,1]$. Let $L'' = L'\setminus \{\ell_0\}$. 
Then $I_\ell = (x_\ell, y_\ell)$ for every $\ell\in L''$ and $D(U) = [0,1]\setminus \bigcup_{\ell\in L} (x_{\ell}, y_\ell)$, where 
\begin{enumerate}[label=L-\alph*), ref=(L-\alph*)]
  \item for all $\ell\in L''$ we have $x_\ell < y_\ell \in (0,1)$;
  \item\label{point: cross ratio one} for all $\ell\in L''$ we have $\frac{x_\ell(1- y_\ell)}{x_\ell - y_\ell} \ge \sinh(\cA/2)^2$;
  \item \label{point: cross ratio two}for all $\ell_1, \ell_2\in L''$ we have $\frac{(x_{\ell_2}-y_{\ell_1})(y_{\ell_2}-x_{\ell_1})}{(y_{\ell_1} - x_{\ell_1})(y_{\ell_2} - x_{\ell_2})} \ge \sinh(\cA/2)^2$.
\end{enumerate}
Denote $\cA' = \sinh(\cA/2)^2 > 5000$.
\subsubsection*{Lower bound on the dimension of the Cantor set $D(U)$}
\begin{obs}\label{obs: stupid ineq}
  The function $a,b,c\mapsto \frac{b(a+b+c)}{ac}$ is monotone increasing in $b$ and monotone decreasing in $a,c$ whenever $a,b,c > 0$. 
\end{obs}
\begin{definition}[A random variable in $z$ in $D(U)$]\label{def: random var}
  We construct a random sequence of decreasing intervals 
  $[0,1] = J_0\supset J_1\supset J_2\supset \cdots$ such that $J_{k+1}$ is one of the three thirds of $J_{k}$ for every $k$. That is, if $J_k = [a_k, a_k+3^{-k}]$, then $J_{k+1} = [a_{k+1}, a_{k+1} + 3^{-k-1}]$ for some $a_{k+1} \in  \{a_k, a_k + 3^{-k-1}, a_k + 2\cdot 3^{-k-1}\}$. 
  We will show how to sample iteratively $J_1,J_2,J_3,\ldots$ so that so that for every $\ell\in L'', k\ge 0$ we have 
  \begin{align}\label{eq: recursive length}
    |J_k \cap I_\ell| < 3^{-k-1}.
  \end{align}
  Note that Eq. \eqref{eq: recursive length} is satisfied for $J_0 = [0,1]$ as for every $\ell\in L''$ we have 
  \[|J_k \cap I_\ell| = x_\ell - y_\ell < \frac{x_\ell - y_\ell}{x_\ell(1- y_\ell)} 
  \stackrel{\ref{point: cross ratio one}}{\le}
  \frac{1}{\cA'} \le \frac{1}{3}.\] 
  Suppose that we have constructed $J_k$ that satisfies Eq. \eqref{eq: recursive length}.
  We say that $J_k$ is a \emph{regular interval} if for all $\ell\in L''$ we have $|J_k \cap I_\ell| < 3^{-k-2}$, and \emph{irregular interval} otherwise. 
  If $J_k$ is a regular interval, we may choose each of the three thirds of $J_k$ to be $J_{k+1}$. We sample $J_{k+1}$ uniformly from these three thirds. 
  \begin{claim}
    If $J_k$ is irregular, then the interval $\ell\in L''$ with $|J_k \cap I_\ell| \ge 3^{-k-2}$ is unique. 
  \end{claim}
  \begin{proof}
    Otherwise there are $\ell_1\neq  \ell_2\in L''$ with $|J_k \cap I_{\ell_i}| \ge 3^{-k-2}$ for $i=1,2$. 
    This implies that $y_{\ell_i} - x_{\ell_i} \ge 3^{-k-2}$. Assume without loss of generality $x_{\ell_2} > y_{\ell_1}$. Then since the two intervals intersect $J_k$ we get that $x_{\ell_2} - y_{\ell_1} < 3^k$. 
    Then 
    \[\cA' \le \frac{(x_{\ell_2}-y_{\ell_1})(y_{\ell_2}-x_{\ell_1})}{(y_{\ell_1} - x_{\ell_1})(y_{\ell_2} - x_{\ell_2})} \le 
    \frac{3^{-k}\cdot (2 \cdot 3^{-k-2} + 3^{-k})}{3^{-2(k+2)}} = 99
    \]
    which is a contradiction. 
    The last inequality follows from Observation \ref{obs: stupid ineq} applied to $a = y_{\ell_1} - x_{\ell_1} \ge 3^{-k-2}$, $b = x_{\ell_2}-y_{\ell_1} \le 3^{-k}$, $c = y_{\ell_2}-x_{\ell_2} \ge 3^{-k-2}$. 
  \end{proof}
  Consequently, if $J_k$ is an irregular interval, then there is a unique $\ell_k\in L''$ such that $|J_k \cap I_{\ell_k}| \ge 3^{-k-2}$. 
  By Eq. \eqref{eq: recursive length} we obtain that
  $|J_k \cap I_{\ell_k}| \in [3^{-k-2}, 3^{-k-1})$. 
  Hence at least one of the three thirds $J$ of $J_k$ satisfies $J\cap I_{\ell_k} = \emptyset$ and hence we choose $J_{k+1}$ uniformly among these intervals. 
  For every $\ell\in L''$, eihter $\ell = \ell_k$ and then $J_{k+1}\cap I_{\ell} = \emptyset$, or 
  $\ell \neq \ell_k$, and then 
  \[|J_{k+1}\cap I_{\ell}|\le |J_{k}\cap I_{\ell}| \le 3^{-k-2}.\]
  Hence $J_{k+1}$ satisfies Eq. \eqref{eq: recursive length}, as desired for the iterative process to continue.
  Let $z$ be the unique element in $\bigcap_{k=0}^\infty J_k$. 
\end{definition}
\begin{claim}\label{claim: z is good}
  Sample $z$ as in Definition \ref{def: random var}. Then $z\in D(U)$. 
\end{claim}
\begin{proof}
  By its definition $z\in [0,1]$. Suppose that $z\in I_\ell$ for some $\ell\in L''$. Then since $I_\ell$ is open, for some $k$ we have $J_k\subseteq I_\ell$. However, by Eq. \eqref{eq: recursive length} we have $|J_k \cap  I_\ell| < 3^{-k-1} < 3^{-k} = |J_k|$. This contradicts $J_k \subseteq I_\ell$ and hence $z \in [0,1]\setminus \bigcup_{\ell\in L} I_\ell= D(U)$. 
\end{proof}
\begin{claim}
  For every $J = [a/3^m, (a+1)/3^m]$ we have \[\PP(J_m = J) < 3^{-(1-1/\cA'')m + 1},\]
  where $\cA'' = \log_3 \cA' - 5 > 2$. 
\end{claim}
\begin{proof}
  Let $F_m = \{k=0,\dots,m-1: J_k \text{ is an irregular interval}\}$. 
  Let $k_1 < k_2 \in F_m$. 
  Then
  $I_{\ell_{k_1}} \cap J_{k_1}\neq \emptyset$, $I_{\ell_{k_1}} \cap J_{k_1 + 1} = \emptyset$ and 
  $I_{\ell_{k_2}} \cap J_{k_2} \neq \emptyset$. 
  This implies that $\ell_{k_1} \neq \ell_{k_2}$.
  Note that $|I_{\ell_{k_1}}| \ge 3^{-k_1-2}, |I_{\ell_{k_2}}| \ge 3^{-k_2-2}$ and since both intervals intersect $J_{k_1}$ we deduce that $d_{\RR}(I_{\ell_{k_1}}, I_{\ell_{k_2}}) < 3^{-k_1}$. 
  Applying Observation \ref{obs: stupid ineq} and Point \ref{point: cross ratio two} to $I_{\ell_{k_1}}, I_{\ell_{k_2}}$ we deduce that 
  \[\cA' \le \frac{3^{-k_1}(3^{-k_1} + 3^{-k_1-2} + 3^{-k_2-2})}{3^{-k_1-2}\cdot 3^{-k_2-2}} \le 3^5 \cdot 3^{k_2 - k_1}.\]
  Hence $k_2 - k_1 \ge \log_3 \cA' - 5$. 
  Therefore, $\#F_m < m/(\log_3 \cA' - 5) + 1$. 
  Note that when sampling $J_{k+1}$ iteratively for $k=0,...,m-1$, if $k\nin F_m$ then $J_{k+1}$ is sampled uniformly between the the three options. 
  Hence the probability $J_m$ was sampled is at most 
  \[\frac{1}{3^{m - \#F_m}} \le \frac{1}{3^{m - m/(\log_3 \cA' - 5) - 1}}.\]
\end{proof}
Consequently, for every $J = [a/3^m, (a+1)/3^m]$ we have $\PP(z\in J) \le 3^{-(1-1/\cA'')m + 1}$. 
This fact, together with Claim \ref{claim: z is good} and standard covering arguments shows that 
\[H.\dim(D(U)) \ge 1-1/\cA'' = 1-O(1/\log \cA') = 1-O(1/\cA). \]
This concludes the proof of Proposition \ref{prop: high dimension away of lines}.
\bibliographystyle{plain}
\bibliography{BibErg}{}
\end{document}